\documentclass[a4paper]{article}
\usepackage[utf8]{inputenc}
\usepackage[english]{babel}
\usepackage[T1]{fontenc}
\usepackage{amsmath,amssymb,bm}

\usepackage{hyperref}

\renewcommand{\bar}{\overline}

\renewcommand{\tilde}{\widetilde}

\usepackage{microtype}

\usepackage{xspace}

\usepackage[all,line]{xy}

\usepackage{mathtools}

\newcommand{\spa}[2]{\operatorname{span}_{#1}\ensuremath{\left\{#2\right\}}}

\newcommand{\inv}{^{-1}}

\newcommand{\iso}{\cong}

\DeclareMathOperator{\Tr}{Tr}

\DeclareMathOperator{\End}{End}

\DeclareMathOperator{\ch}{ch}

\DeclareMathOperator{\wt}{wt}

\DeclareMathOperator{\ad}{ad}

\DeclareMathOperator{\height}{ht}

\newcommand{\Suppess}{\ensuremath{\operatorname{Supp}_{\operatorname{ess}}}}

\newcommand{\tensor}{\otimes}

\renewcommand{\phi}{\varphi}

\renewcommand{\epsilon}{\varepsilon}

\newcommand{\setC}{\mathbb{C}}

\newcommand{\setN}{\mathbb{N}}

\newcommand{\setZ}{\mathbb{Z}}

\newcommand{\setQ}{\mathbb{Q}}

\usepackage[amsmath,thmmarks,hyperref]{ntheorem}

\newtheorem{thm}{Theorem}[section]

\newtheorem{prop}[thm]{Proposition}

\newtheorem{lemma}[thm]{Lemma}

\newtheorem{cor}[thm]{Corollary}

\newtheorem{defn}[thm]{Definition}

\theoremstyle{nonumberplain}
\theoremheaderfont{%
\normalfont\bfseries}
\theorembodyfont{\normalfont}
\theoremsymbol{\ensuremath{\square}}
\theoremseparator{.}
\newtheorem{proof}{Proof}

\usepackage{enumitem}

\begin{document}

\title{Irreducible quantum group modules with finite dimensional
  weight spaces. II} \author{Dennis Hasselstrøm Pedersen} \date{}

\maketitle

\begin{abstract}
  We classify the simple quantum group modules with finite dimensional
  weight spaces when the quantum parameter $q$ is transcendental and
  the Lie algebra is not of type $G_2$. This is part $2$ of the
  story. The first part being~\cite{DHP1}. In~\cite{DHP1} the
  classification is reduced to the classification of torsion free
  simple modules. In this paper we follow the procedures
  of~\cite{Mathieu} to reduce the classification to the classification
  of infinite dimensional admissible simple highest weight modules. We
  then classify the infinite dimensional admissible simple highest
  weight modules and show among other things that they only exist for
  types $A$ and $C$.  Finally we complete the classification of simple
  torsion free modules for types $A$ and $C$ completing the
  classification of the simple torsion free modules.
\end{abstract}

\tableofcontents

\section{Introduction}
\label{sec:intr-notat}

This is part 2 of the classification of simple quantum group modules
with finite dimensional weight spaces. In this paper we focus on the
non root of unity case.  Let $\mathfrak{g}$ be a simple Lie
algebra. Let $q\in \setC$ be a non root of unity and let $U_q$ be the
quantized enveloping algebra over $\setC$ with $q$ as the quantum
parameter (defined below). We want to classify all simple weight
modules for $U_q$ with finite dimensional weight spaces. In the
papers~\cite{Fernando} and~\cite{Mathieu} this is done for
$\mathfrak{g}$-modules. Fernando proves in~\cite{Fernando} that the
classification of simple $\mathfrak{g}$ weight modules with finite
dimensional weight spaces essentially boils down to classifying two
classes of simple modules: Finite dimensional simple modules over a
reductive Lie algebra and so called 'torsion free' simple modules over
a simple Lie algebra. The classification of finite dimensional modules
is well known in the classical case (as well as the quantum group
case) so the remaining problem is to classify the torsion free simple
modules. O. Mathieu classifies all torsion free
$\mathfrak{g}$-modules in~\cite{Mathieu}. The classification
uses the concept of a $\mathfrak{g}$ coherent family which are huge
$\mathfrak{g}$ modules with weight vectors for every possible weight,
see~\cite[Section~4]{Mathieu}. Mathieu shows that every torsion free
simple module is a submodule of a unique irreducible semisimple
coherent family~\cite[Proposition~4.8]{Mathieu} and each of these
irreducible semisimple coherent families contains an admissible simple
highest weight module as well~\cite[Proposition~6.2
ii)]{Mathieu}. This reduces the classification to the classification
of admissible simple highest weight modules.  In this paper we will
follow closely the methods described in~\cite{Mathieu}. We will focus
only on the case when $q$ is not a root of unity. The root of unity
case is studied in~\cite{DHP1}. Some of the results
of~\cite{Mathieu} translate directly to the quantum group case but in
several cases there are obstructions that need to be handled
differently. In particular the case by case classification in types A
and C is done differently. This is because our analog of
$\mathcal{EXT}(L)$ given an admissible simple infinite dimensional
module $L$ is slightly different from the classical case see
e.g. Section~\ref{sec:an-example-u_qm}. The proof when reducing to
types A and C in~\cite{Fernando} and~\cite{Mathieu} uses some
algebraic geometry to show that torsion free modules can only exist in
types A and C. In this paper we show that infinite dimensional
admissible simple highest weight modules only exist in types A and C
and use this fact to show that torsion free modules can not exist for
modules other than types A and C. For this we have to restrict to
transcendental $q$. Specifically we use Theorem~\ref{thm:integral}. If
this theorem is true for a general non-root-of-unity $q$ we can remove
this restriction. The author is not aware of such a result in the
litterature.

\subsection{Main results}
To classify simple weight modules with finite dimensional modules we
follow the procedures of S. Fernando and O. Mathieu in~\cite{Fernando}
and~\cite{Mathieu}. The analog of the reduction done
in~\cite{Fernando} is taken care of in the quantum group case
in~\cite{DHP1} so what remains is to classify the torsion free
modules. We will first recall some results from~\cite{DHP1}
and~\cite{DHP-twist} concerning the reduction and some formulas for
commutating root vectors. This is recalled in
Section~\ref{sec:nonroot-unity-case} and
Section~\ref{sec:u_a-calculations}. In
Section~\ref{sec:class-tors-free-1} we do some prelimary calculations
concerning Ore localization and certain 'twists' of modules necessary
to define the 'Coherent families' of
Section~\ref{sec:coherent-families}. Here we don't define the concept
of a general coherent family but instead directly define the analog of
coherent irreducible semisimple extensions $\mathcal{EXT}(L)$ of an
admissible simple infinite dimensional module $L$. In analog with the
classical case we show that for any admissible simple infinite
dimensional module $L$, $\mathcal{EXT}(L)$ contains a submodule
isomorphic to a simple highest weight module, see
Theorem~\ref{thm:EXT_contains_highest_weight}. We also prove a result
in the other direction: If $\mathfrak{g}$ is such that there exists a
simple infinite dimensional admissible module $L$ then there exists a
torsion free $U_q(\mathfrak{g})$-module, see
Theorem~\ref{thm:existence_of_torsion_free_modules}. So the existence
of torsion free modules over the quantized enveloping algebra of a
specific $\mathfrak{g}$ is equivalent to the existence of an
admissible infinite dimensional highest weight simple module over
$U_q(\mathfrak{g})$. Using this we show that torsion free modules
exist only for types $A$ and $C$ in the
Sections~\ref{sec:class-admiss-simple}, \ref{sec:rank-2-calculations},
\ref{sec:class-admiss-modul}, \ref{sec:quantum-shale-weyl} and
\ref{sec:class-admiss-modul-1} where we also classify the admissible
simple highest weight modules which are infinite dimensional. Finally
in Section~\ref{sec:type-a_n-calc} and Section~\ref{sec:type-c_n-calc}
we complete the classification in types $A$ and $C$, respectively, by
showing exactly which submodules of $\mathcal{EXT}(L(\lambda))$ are
torsion free for a $\lambda$ of a specific form see
Theorem~\ref{thm:clas_of_b_such_that_twist_is_torsion_free} and
Theorem~\ref{thm:clas_C}.

\subsection{Acknowledgements}
I would like to thank my advisor Henning H. Andersen for great
supervision and many helpful comments and discussions and Jacob
Greenstein for introducing me to this problem when I was visiting him
at UC Riverside in the fall of 2013. The authors research was
supported by the center of excellence grant 'Center for Quantum
Geometry of Moduli Spaces' from the Danish National Research
Foundation (DNRF95).

\subsection{Notation}
We will fix some notation: We denote by $\mathfrak{g}$ a fixed simple
Lie algebra over the complex numbers $\setC$. We assume $\mathfrak{g}$
is not of type $G_2$ to avoid unpleasant computations.

Fix a triangular decomposition of $\mathfrak{g}$: $\mathfrak{g} =
\mathfrak{g}^- \oplus \mathfrak{h} \oplus \mathfrak{g}^+$: Let
$\mathfrak{h}$ be a maximal toral subalgebra and let $\Phi \subset
\mathfrak{h}^*$ be the roots of $\mathfrak{g}$ relative to
$\mathfrak{h}$. Choose a simple system of roots $\Pi =
\{\alpha_1,\dots,\alpha_n\} \subset \Phi$. Let $\Phi^+$
(resp. $\Phi^-$) be the positive (resp. negative) roots. Let
$\mathfrak{g}^{\pm}$ be the positive and negative part of
$\mathfrak{g}$ corresponding to the simple system $\Pi$. Let $W$ be
the Weyl group generated by the simple reflections $s_i :=
s_{\alpha_i}$. For a $w\in W$ let $l(w)$ be the length of $W$ i.e. the
smallest amount of simple reflections such that $w=s_{i_1}\cdots
s_{i_{l(w)}}$. Let $(\cdot|\cdot)$ be a standard $W$-invariant
bilinear form on $\mathfrak{h}^*$ and $\left<\alpha,\beta^\vee\right>
= \frac{2(\alpha|\beta)}{(\beta|\beta)}$. Since $(\cdot|\cdot)$ is
standard we have $(\alpha|\alpha)=2$ for any short root $\alpha\in
\Phi$ and since $\mathfrak{g}$ is not of type $G_2$ we have
$(\beta|\beta)=4$ for any long root $\beta\in \Phi$. Let
$Q=\spa{\setZ}{\alpha_1,\dots,\alpha_n}$ denote the root lattice and
$\Lambda=\spa{\setZ}{\omega_1,\dots,\omega_n}\subset \mathfrak{h}^*$
the integral lattice where $\omega_i\in \mathfrak{h}^*$ is the
fundamental weights defined by $(\omega_i|\alpha_j)=\delta_{ij}$.

Let $U_v=U_v(\mathfrak{g})$ be the corresponding quantized enveloping
algebra defined over $\mathbb{Q}(v)$, see e.g.~\cite{Jantzen} with
generators $E_\alpha,F_\alpha,K_\alpha^{\pm 1}$, $\alpha\in\Pi$ and
certain relations which can be found in chapter~4
of~\cite{Jantzen}. We define $v_\alpha = v^{(\alpha|\alpha)/2}$
(i.e. $v_\alpha = v$ if $\alpha$ is a short root and $v_\alpha = v^2$
if $\alpha$ is a long root) and for $n\in\setZ$, $[n]_v
=\frac{v^n-v^{-n}}{v-v\inv}$.  Let $[n]_\alpha := [n]_{v_\alpha} =
\frac{v_\alpha^n-v_\alpha^{-n}}{v_\alpha-v_\alpha\inv}$. We omit the
subscripts when it is clear from the context. For later use we also
define the quantum binomial coefficients: For $r\in \setN$ and $a\in
\setZ$:
\begin{equation*}
  {a \brack r}_v = \frac{[a][a-1]\cdots [a-r+1]}{[r]!}
\end{equation*}
where $[r]! := [r][r-1]\cdots [2][1]$. Let $A=\setZ[v,v\inv]$ and let
$U_A$ be Lusztigs $A$-form, i.e. the $A$ subalgebra generated by the
divided powers $E_\alpha^{(n)}:=\frac{1}{[n]_\alpha!}E_\alpha^{n}$,
$F_\alpha^{(n)}:=\frac{1}{[n]_\alpha!}F_\alpha^{n}$ and $K_\alpha^{\pm
  1}$, $\alpha\in\Pi$.

Let $q\in \setC^*=\setC\backslash\{0\}$ be a nonzero complex number
that is not a root of unity and set $U_q = U_A \tensor_A \setC_q$
where $\setC_q$ is the $A$-algebra $\setC$ where $v$ is sent to $q$.

We have a triangular decomposition of Lusztigs $A$-form $U_A = U_A^-
\tensor U_A^0 \tensor U_A^+$ with $U_A^-=\left<
  F_\alpha^{(n)}|\alpha\in \Pi,n\in \setN \right>\in U_A$,
$U_A^+=\left< E_\alpha^{(n)}|\alpha\in \Pi,n\in \setN \right>\in U_A$
and $U_A^0 = \left< K_\alpha^{\pm 1}, { K_\alpha ; c \brack
    r}|\alpha\in \Pi, c\in \setZ, r\in \setN\right>$ where
\begin{equation*}
  {K_\alpha ; c \brack r} := \prod_{j=1}^r \frac{ K_\alpha v_\alpha^{c-j+1}-K_\alpha\inv v_\alpha^{-c+j-1}}{v_\alpha^{j}-v_\alpha^{-j}}.
\end{equation*}
We have the corresponding triangular decomposition of $U_q$: $U_q =
U_q^- \tensor U_q^0 \tensor U_q^+$ with $U_q^{\pm} = U_A^{\pm}
\tensor_A \setC_q$ and $U_q^0 = U_A^0 \tensor_A \setC_q$.

For a $q\in \setC^*$ define ${a \brack r}_q$ as the image of ${a \brack
  r}_v$ in $\setC_q$.  We will omit the subscript from the notation when
it is clear from the context.  $q_\beta\in \setC$ and $[n]_\beta\in
\setC$ are defined as the image of $v_\beta\in A$ and $[n]_\beta\in
A$, respectively abusing notation. Similarly, we will abuse notation
and write ${K_\alpha ; c \brack r}$ also for the image of ${K_\alpha ;
  c \brack r}\in U_A$ in $U_q$. Define for $\mu\in Q$, $K_\mu =
\prod_{i=1}^n K_{\alpha_i}^{a_i}$ if $\mu = \sum_{i=1}^n a_i \alpha_i$
with $a_i\in \setZ$.

There is a braid group action on $U_v$ which we will describe now. We
use the definition from~\cite[Chapter~8]{Jantzen}. The definition is
slightly different from the original in~\cite[Theorem~3.1]{MR1066560}
(see \cite[Warning~8.14]{Jantzen}). For each simple reflection $s_i$
there is a braid operator that we will denote by $T_{s_i}$ satisfying
the following: $T_{s_i}:U_v\to U_v$ is a $\setQ(v)$ automorphism. For
$i\neq j \in \{1,\dots,n\}$
\begin{align*}
  T_{s_i}(K_\mu)=&K_{s_i(\mu)}
  \\
  T_{s_i}(E_{\alpha_i}) =& -F_{\alpha_i}K_{\alpha_i}
  \\
  T_{s_i}(F_{\alpha_i})=& - K_{\alpha_i}\inv E_{\alpha_i}
  \\
  T_{s_i}(E_{\alpha_j})=&
  \sum_{i=0}^{-\left<\alpha_j,\alpha_i^\vee\right>} (-1)^i
  v_{\alpha_i}^{-i} E_{\alpha_i}^{(r-i)}E_{\alpha_j}E_{\alpha_i}^{(i)}
  \\
  T_{s_i}(F_{\alpha_j})=&
  \sum_{i=0}^{-\left<\alpha_j,\alpha_i^\vee\right>} (-1)^i
  v_{\alpha_i}^{i} E_{\alpha_i}^{(i)}E_{\alpha_j}E_{\alpha_i}^{(r-i)}.
\end{align*}
The inverse $T_{s_i}\inv$ is given by conjugating with the
$\setQ$-algebra anti-automorphism $\Psi$
from~\cite[section~1.1]{MR1066560} defined as follows:
\begin{align*}
  \Psi(E_{\alpha_i}) = E_{\alpha_i}, \quad \Psi(F_{\alpha_i}) =
  F_{\alpha_i}, \quad \Psi(K_{\alpha_i}) = K_{\alpha_i}\inv, \quad
  \Psi(q) = q.
\end{align*}
The braid operators $T_{s_i}$ satisfy braid relations so we can define
$T_w$ for any $w\in W$: Choose a reduced expression of $w$:
$w=s_{i_1}\cdots s_{i_n}$. Then $T_w = T_{s_{i_1}}\cdots T_{s_{i_n}}$
and $T_w$ is independent of the chosen reduced expression, see
e.g.~\cite[Theorem~3.2]{MR1066560}. We have
$T_w(K_\mu)=K_{w(\mu)}$. The braid group operators restrict to
automorphisms $U_A\to U_A$ and extend to automorphisms $U_q\to U_q$.

Let $M$ be a $U_q$-module and $\lambda: U_q^0 \to \setC$ a character
(i.e. an algebra homomorphism into $\setC$). Then
\begin{equation*}
  M_\lambda = \{ m\in M | \forall u\in U_q^0, u m = \lambda(u)m\}.
\end{equation*}
Let $X$ denote the set of characters $U_q^0\to \setC$. Since
$U_q^0\iso \setC[X_1^{\pm 1},\dots,X_n^{\pm 1}]$ we can identify $X$
with $(\setC^*)^n$ by $U_q^0\ni \lambda \mapsto
(\lambda(K_{\alpha_1}),\dots,\lambda(K_{\alpha_n}))\in (\setC^*)^n$.

\subsection{Basic definitions}
\begin{defn}
  \label{sec:intr-notat-2}
  Let $\wt M$ denote all the weights of $M$, i.e. $\wt M = \{
  \lambda\in X | M_\lambda \neq 0 \}$.
  
  For $\mu\in \Lambda$ and $b\in \setC^*$ define the character $b^\mu$
  by $b^\mu(K_\alpha) = b^{(\mu|\alpha)}$, $\alpha\in \Pi$. In
  particular for $b=q$ we get $q^\mu(K_\alpha)=q^{(\mu|\alpha)}$. We
  say that $M$ only has integral weights if $\lambda(K_\alpha)\in \pm
  q_\alpha^\setZ$ for all $\lambda \in \wt M$, $\alpha\in \Pi$.
\end{defn}
There is an action of $W$ on $X$. For $\lambda\in X$ define $w\lambda$
by
\begin{equation*}
  (w\lambda)(u) = \lambda(T_{w\inv}(u))
\end{equation*}
Note that $w q^\mu = q^{w(\mu)}$.

\begin{defn}
  \label{sec:intr-notat-1}
  Let $M$ be a $U_q$-module and $w\in W$. Define the twisted module
  ${^w}M$ by the following:
  
  As a vector space ${^w}M=M$ but the action is given by twisting with
  $w\inv$: For $m\in {^w}M$ and $u \in U_q$:
  \begin{equation*}
    u\cdot m = T_{w\inv}(u)m.
  \end{equation*}
  
  We also define ${^{\bar{w}}}M$ to be the inverse twist, i.e. for
  $m\in {^{\bar{w}}}M$, $u\in U_q$:
  \begin{equation*}
    u \cdot m = T_{w\inv}\inv(u) m.
  \end{equation*}
  Hence for any $U_q$-module ${^{\bar{w}}}(^{w}M) = M =
  {^w}(^{\bar{w}}M)$.
\end{defn}
Note that $\wt {^w}M = w(\wt M)$ and that ${^w}(^{w'}M)\iso {^{ww'}}M$
for $w,w'\in W$ with $l(ww')=l(w)+l(w')$ because the braid operators
$T_w$ satisfy braid relations. Also ${^{\bar{w}}}(^{\bar{w'}}M) \iso
{^{\bar{w'w}}}M$

\begin{defn}
  \label{def:1}
  We define the category $\mathcal{F}=\mathcal{F}(\mathfrak{g})$ as
  the full subcategory of $U_q-\operatorname{Mod}$ such that for every
  $M\in \mathcal{F}$ we have
  \begin{enumerate}
  \item $M$ is finitely generated as a $U_q$-module.
  \item $M = \bigoplus_{\lambda\in X} M_\lambda$ and $\dim M_\lambda <
    \infty$.
  \end{enumerate}
\end{defn}

Note that the assignment $M\mapsto {^w}M$ is an endofunctor on
$\mathcal{F}$ (in fact an auto-equivalence).

Let $w_0$ be the longest element in $W$ and let $s_{i_1}\cdots
s_{i_N}$ be a reduced expression of $w_0$. We define root vectors
$E_\beta$ and $F_\beta$ for any $\beta\in \Phi^+$ by the following:

First of all set
\begin{equation*}
  \beta_{j} = s_{i_1}\cdots s_{i_{j-1}}(\alpha_{i_j}), \, \text{ for } i=1,\dots,N
\end{equation*}
Then $\Phi^+ = \{\beta_1,\dots,\beta_N\}$. Set
\begin{equation*}
  E_{\beta_j} = T_{s_{i_1}}\cdots T_{s_{i_{j-1}}}(E_{\alpha_{i_j}})
\end{equation*}
and
\begin{equation*}
  F_{\beta_j} = T_{s_{i_1}}\cdots T_{s_{i_{j-1}}}(F_{\alpha_{i_j}})
\end{equation*}
In this way we have defined root vectors for each
$\beta\in\Phi^+$. These root vectors depend on the reduced expression
chosen for $w_0$ above. For a different reduced expression we might
get different root vectors. It is a fact that if $\beta\in\Pi$ then
the root vectors $E_\beta$ and $F_\beta$ defined above are the same as
the generators with the same notation
(cf. e.g.~\cite[Proposition~8.20]{Jantzen}) so the notation is not
ambigious in this case. By ``Let $E_\beta$ be a root vector'' we will
just mean a root vector constructed as above for some reduced
expression of $w_0$.

\section{Reductions}
\label{sec:nonroot-unity-case}
We recall the following results from~\cite{DHP1}.

\begin{prop}
  \label{prop:2}
  Let $\beta$ be a positive root and $E_\beta,F_\beta$ root vectors
  corresponding to $\beta$. Let $M\in\mathcal{F}$. The sets
  $M^{[\beta]}=\{m\in M| \dim \left<E_\beta\right> m < \infty \}$ and
  $M^{[-\beta]}=\{m\in M| \dim \left<F_\beta\right> m < \infty \}$ are
  submodules of $M$ and independent of the chosen root vectors
  $E_\beta$, $F_\beta$.
\end{prop}
\begin{proof}
  This is shown for $E_\beta$ in Proposition~2.3 and Lemma~2.4
  in~\cite{DHP1} and the proofs are the same for $F_\beta$.
\end{proof}

\begin{defn}
  Let $M\in \mathcal{F}$. Let $\beta\in\Phi$. $M$ is called
  $\beta$-free if $M^{[\beta]}=0$ and $\beta$-finite if
  $M^{[\beta]}=M$.
\end{defn}

Suppose $L\in \mathcal{F}$ is a simple module and $\beta$ a root. Then
by Proposition~\ref{prop:2} $L$ is either $\beta$-finite or
$\beta$-free.

\begin{defn}
  Let $M\in \mathcal{F}$. Define $F_M = \{\beta \in \Phi| \text{$M$ is
    $\beta$-finite}\}$ and $T_M = \{ \beta \in \Phi | \text{$M$ is
    $\beta$-free} \}$. For later use we also define $F_M^s := F_M \cap
  (-F_M)$ and $T_M^s := T_M \cap (-T_M)$ to be the symmetrical parts
  of $F_M$ and $T_M$.
\end{defn}
Note that $\Phi = F_L \cup T_L$ for a simple module $L$ and this is a
disjoint union.

\begin{defn}
  A module $M$ is called torsion free if $T_M = \Phi$.
\end{defn}

\begin{prop}
  \label{prop:3}
  Let $L\in \mathcal{F}$ be a simple module and $\beta$ a root. $L$ is
  $\beta$-free if and only if $q^{\setN \beta}\wt L \subset \wt L$.
\end{prop}
\begin{proof}
  Proposition~2.9 in~\cite{DHP1}.
\end{proof}

\begin{prop}
  \label{prop:8}
  Let $L\in \mathcal{F}$ be a simple module. $T_L$ and $F_L$ are
  closed subsets of $Q$. That is, if $\beta,\gamma\in F_L$
  (resp. $\beta,\gamma\in T_L$) and $\beta+\gamma\in \Phi$ then
  $\beta+\gamma \in F_L$ (resp. $\beta+\gamma\in T_L$).
\end{prop}
\begin{proof}
  Proposition~2.10 and Proposition~2.11 in~\cite{DHP1}.
\end{proof}

\begin{thm}
  \label{thm:Lemire}
  Let $\lambda\in X$.  There is a $1-1$ correspondence between simple
  $U_q$-modules with weight $\lambda$ and simple $(U_q)_0$-modules
  with weight $\lambda$ given by: For $V$ a $U_q$-module, $V_\lambda$
  is the corresponding simple $(U_q)_0$-module.
\end{thm}
\begin{proof}
  Theorem~2.21 in~\cite{DHP1}.
\end{proof}

\begin{thm}
  \label{thm:classification}
  Let $L\in\mathcal{F}$ be a simple $U_q(\mathfrak{g})$-module. Then
  there exists a $w\in W$, subalgebras
  $U_q(\mathfrak{p}),U_q(\mathfrak{l}),U_q(\mathfrak{u}),U_q(\mathfrak{u}^-)$
  of $U_q$ with $U_q = U_q(\mathfrak{u}^-) U_q(\mathfrak{p})$,
  $U_q(\mathfrak{p})= U_q(\mathfrak{l})U_q(\mathfrak{u})$ and a simple
  $U_q(\mathfrak{l})$-module $N$ such that ${^w}L$ is the unique
  simple quotient of $U_q\tensor_{U_q(\mathfrak{p})}N$ where $N$ is
  considered a $U_q(\mathfrak{p})$-module with $U_q({\mathfrak{u}})$
  acting trivially.

  Furthermore there exists subalgebras $U_{fr},U_{fin}$ of
  $U_q(\mathfrak{l})$ such that $U_q(\mathfrak{l})\iso U_{fr}\tensor
  U_{fin}$ and simple $U_{fr}$ and $U_{fin}$ modules $X_{fr}$ and
  $X_{fin}$ where $X_{fin}$ is finite dimensional and $X_{fr}$ is
  torsion free such that $N \iso X_{fin}\tensor X_{fr}$ as a
  $U_{fr}\tensor U_{fin}$-module.

  $U_{fr}$ is the quantized enveloping algebra of a semisimple Lie
  algebra $\mathfrak{t}=\mathfrak{t}_1\oplus \cdots \oplus
  \mathfrak{t}_r$ where $\mathfrak{t}_1,\dots,\mathfrak{t}_r$ are some
  simple Lie algebras. There exists simple torsion free
  $U_q(\mathfrak{t}_i)$-modules $X_i$, $i=1,\dots,r$ such that
  $X_{fr}\iso X_1\tensor \cdots \tensor X_r$ as
  $U_q(\mathfrak{t}_1)\tensor \cdots \tensor
  U_q(\mathfrak{t}_r)$-modules.
\end{thm}
\begin{proof}
  Theorem~2.23 in~\cite{DHP1}.
\end{proof}

So the problem of classifying simple modules in $\mathcal{F}$ is
reduced to the problem of classifying finite dimensional simple
modules and classifying simple torsion free modules of
$U_q(\mathfrak{t})$ where $\mathfrak{t}$ is a simple Lie algebra.

\section{$U_A$ calculations}
\label{sec:u_a-calculations}
In this section we recall from~\cite{DHP-twist} some formulas for
commuting root vectors with each other that will be used later
on. Recall that $A=\setZ[v,v\inv]$ where $v$ is an indeterminate and
$U_A$ is the $A$-subspace of $U_v$ generated by the divided powers
$E_{\alpha}^{(n)}$ and $F_\alpha^{(n)}$, $n\in\setN$.
\begin{defn}
  \label{sec:twisting-functors-2}
  Let $x\in (U_q)_\mu$ and $y\in (U_q)_\gamma$ then we define
  \[ [x,y]_q:=xy-q^{-(\mu|\gamma)}yx
  \]
\end{defn}

\begin{thm}
  \label{thm:DP}
  Suppose we have a reduced expression of $w_0 = s_{i_1}\cdots
  s_{i_N}$ and define root vectors
  $F_{\beta_1},\dots,F_{\beta_N}$. Let $i<j$. Let $A=\setZ[q,q\inv]$
  and let $A'$ be the localization of $A$ in $[2]$ if the Lie algebra
  contains any $B_n,C_n$ or $F_4$ part. Then
  \begin{equation*}
    [F_{\beta_j},F_{\beta_i}]_q=F_{\beta_j}F_{\beta_i}-q^{-(\beta_i|\beta_j)}F_{\beta_i}F_{\beta_j}\in \spa{A'}{F_{\beta_{j-1}}^{a_{j-1}}\cdots F_{\beta_{i+1}}^{a_{i+1}}}
  \end{equation*}
\end{thm}
\begin{proof}
  \cite[Proposition~5.5.2]{Levendorski-Soibelman}. A proof
  following~\cite[Theorem~9.3]{DP} can also be found
  in~\cite[Theorem~2.9]{DHP-twist}.
\end{proof}

\begin{defn}
  Let $u\in U_A$ and $\beta\in \Phi^+$.  Define
  $\ad(F_\beta^i)(u)
  :=[[\dots[u,F_\beta]_q\dots]_q,F_\beta]_q$ and
  $\tilde{\ad}(F_\beta^i)(u)
  :=[F_\beta,[\dots,[F_\beta,u]_q\dots]]_q$ where the
  commutator is taken $i$ times from the left and right respectively.
\end{defn}

\begin{prop}
  \label{prop:17}
  Let $a\in\setN$, $u\in (U_A)_\mu$ and
  $r=\left<\mu,\beta^\vee\right>$. In $U_A$ we have the identities
  \begin{align*}
    u F_{\beta}^{a} =& \sum_{i=0}^a v_\beta^{(i-a)(r+i)} {a\brack
      i}_\beta F_\beta^{a-i}\ad(F_\beta^{i})(u)
    \\
    =& \sum_{i=0}^{a} (-1)^i v_\beta^{a(r+i)-i} {a\brack i}_\beta
    F_\beta^{a-i} \tilde{\ad}(F_\beta^{i})(u)
  \end{align*}
\end{prop}
\begin{proof}
  Proposition~2.13 in~\cite{DHP-twist}.
\end{proof}


Let $s_{i_1}\dots s_{i_N}$ be a reduced expression of $w_0$ and
construct root vectors $F_{\beta_i}$, $i=1,\dots,N$. In the next lemma
$F_{\beta_i}$ refers to the root vectors constructed as such. In
particular we have an ordering of the root vectors.

\begin{lemma}
  \label{lemma:22}
  Let $n\in \setN$. Let $1\leq j<k\leq N$.
  
  $\ad(F_{\beta_j}^{i})(F_{\beta_k}^{n})=0$ and
  $\tilde{\ad}(F_{\beta_k}^{i})(F_{\beta_j}^{n})=0$ for $i\gg 0$.
\end{lemma}
\begin{proof}
  Lemma~2.16 in~\cite{DHP-twist}.
\end{proof}

\section{Ore localization and twists of localized modules}
\label{sec:class-tors-free-1}
In this section we present some results towards classifying simple
torsion free modules following~\cite{Mathieu}.

We need the equivalent of Lemma~3.3 in~\cite{Mathieu}. The proofs are
essentially the same but for completeness we include most of the
proofs here.

\begin{defn}
  A cone $C$ is a finitely generated submonoid of the root lattice $Q$
  containing $0$.  If $L$ is a simple module define the cone of $L$,
  $C(L)$, to be the submonoid of $Q$ generated by $T_L$.
\end{defn}

\begin{lemma}
  \label{lemma:11}
  Let $L\in\mathcal{F}$ be an infinite dimensional simple module. Then
  the group generated by the submonoid $C(L)$ is $Q$.
\end{lemma}
Compare~\cite{Mathieu} Lemma~3.1
\begin{proof}
  First consider the case where $T_L \cap (-F_L) = \emptyset$. Then in
  this case we have $\Phi = T_L^s \cup F_L^s$. We claim that $T_L^s$
  and $F_L^s$ correspond to different connected components of the
  Dynkin diagram: Suppose $\alpha \in F_L^s$ is a simple root and
  suppose $\alpha'\in \Pi$ is a simple root that is connected to
  $\alpha$ in the Dynkin diagram. So $\alpha+\alpha'$ is a root. There
  are two possibilities. Either $\alpha+\alpha' \in F_L$ or
  $\alpha+\alpha'\in T_L$. If $\alpha+\alpha'\in F_L$: Since $F_L^s$
  is symmetric we have $-\alpha\in F_L^s$ and since $F_L$ is closed
  (Proposition~\ref{prop:8}) $\alpha' = \alpha+\alpha' + (-\alpha) \in
  F_L$. If $\alpha+\alpha'\in T_L$ and $\alpha' \in T_L$ then we get
  similarly $\alpha \in T_L$ which is a contradiction. So $\alpha'\in
  F_L$. We have shown that if $\alpha\in F_L$ then any simple root
  connected to $\alpha$ is in $F_L$ also. So $F_L$ and $T_L$ contains
  different connected components of the Dynkin diagram. Since $L$ is
  simple and infinite we must have $\Phi = T_L^s$ and therefore
  $C(L)=Q$.
  
  Next assume $T_L \cap (-F_L) \neq \emptyset$. By Lemma~4.16
  in~\cite{Fernando} $P_L=T_L^s\cup F_L$ and $P_L^-=T_L\cup F_L^s$ are
  two opposite parabolic subsystems of $\Phi$. So we have that
  $T_L\cap (-F_L)$ and $(-T_L)\cap F_L$ must be the roots
  corresponding to the nilradicals $\mathfrak{v}^\pm$ of two opposite
  parabolic subalgebras $\mathfrak{p}^\pm$ of $\mathfrak{g}$. Since we
  have $\mathfrak{g}=\mathfrak{v}^+ + \mathfrak{v}^- +
  [\mathfrak{v}^+,\mathfrak{v}^-]$ we get that $T_L\cap (-F_L)$
  generates $Q$. Since $C(L)$ contains $T_L\cap (-F_L)$ it generates
  $Q$.
\end{proof}

We define $\rho$ and $\delta$ like in~\cite[Section~3]{Mathieu}:
\begin{defn}
  Let $x\geq 0$ be a real number. Define $\rho(x) =
  \operatorname{Card}B(x)$ where $B(x)=\{\mu\in Q| \sqrt{(\mu|\mu)}
  \leq x\}$

  Let $M$ be a weight module with support lying in a single $Q$-coset,
  say $q^Q \lambda:=\{q^\mu \lambda|\mu\in Q\}$. The density of $M$ is
  $\delta(M) = {\lim \inf}_{x\to \infty} \rho(x)\inv \sum_{\mu\in
    B(x)} \dim M_{q^{\mu}\lambda} $

  For a cone $C$ we define $\delta(C) = {\lim \inf}_{x\to \infty}
  \rho(x)\inv \operatorname{Card}(C \cap B(x))$
\end{defn}

\begin{lemma}
  \label{lemma:12}
  There exists a real number $\epsilon > 0$ such that
  $\delta(L)>\epsilon$ for all infinite dimensional simple modules
  $L$.
\end{lemma}
\begin{proof}
  Note that since $q^{C(L)} \lambda \subset \wt L$ for all $\lambda\in
  \wt L$ we have $\delta(L) \geq \delta(C(L))$.

  Since $C(L)$ is the cone generated by $T_L$ and $T_L\subset \Phi$ (a
  finite set) there can only be finitely many different cones.

  Since there are only finitely many different cones attached to
  infinite simple dimensional modules and since any cone $C$ that
  generates $Q$ has $\delta(C)>0$ we conclude via Lemma~\ref{lemma:11}
  that there exists an $\epsilon>0$ such that $\delta(L)>\epsilon$ for
  all infinite dimensional simple modules.
\end{proof}

\begin{defn}
  A module $M\in \mathcal{F}$ is called admissible if its weights are
  contained in a single coset of $X / q^Q$ and if the dimensions of
  the weight spaces are uniformly bounded. $M$ is called admissible of
  degree $d$ if $d$ is the maximal dimension of the weight spaces in
  $M$.
\end{defn}

Of course all finite dimensional simple modules are admissible but the
interesting admissible modules are the infinite dimensional simple
ones. In particular simple torsion free modules are admissible. We
show later that each infinite dimensional admissible simple module $L$
gives rise to a 'coherent family' $\mathcal{EXT}(L)$ containing at
least one torsion free module and at least one simple highest weight
module that is admissible of the same degree.

\begin{lemma}
  \label{lemma:10}
  Let $M\in \mathcal{F}$ be an admissible module. Then $M$ has finite
  Jordan-Hölder length.
\end{lemma}
\begin{proof}
  The length of $M$ is bounded by $A+\delta(M)/ \epsilon$ where $A=
  \sum_{\lambda\in Y} \dim M_\lambda$ and $Y=\{\nu \in X | \, \nu =
  \sigma q^\mu, |\left< \mu,\alpha^\vee\right>|\leq 1,
  \sigma(K_{\alpha})\in\{\pm 1\}\, \text{ for all } \alpha\in \Pi\}$.
  Check~\cite[Lemma~3.3]{Mathieu} for details. Here we use the fact
  that finite dimensional simple quantum group modules have the same
  character as their corresponding Lie algebra simple modules. This is
  proved for transcendental $q$ in~\cite[Theorem~5.15]{Jantzen} and
  for general non-roots-of-unity in~\cite[Corollary~7.7]{APW}.
\end{proof}

\begin{lemma}
  \label{lemma:27}
  Let $\beta$ be a positive root and let $F_\beta$ be a corresponding
  root vector. The set $\{F_\beta^{n}|n\in \setN \}$ is an Ore subset
  of $U_q$.
\end{lemma}
\begin{proof}
  A proof can be found in~\cite{HHA-kvante} for $\beta$ a simple
  root. If $\beta$ is not simple then $F_\beta$ is defined as
  $T_w(F_\alpha)$ for some $w\in W$ and some $\alpha\in \Pi$. Since
  $S:= \{F_\alpha^{n}|n\in \setN\}$ is an Ore subset of $U_q$ we get
  for any $n\in \setN$ and $u\in U_q$ that
  \begin{equation*}
    F_\alpha^{n} U_q \cap u S \neq \emptyset.
  \end{equation*}
  Let $u'\in U_q$ and set $u= T_w\inv(u')$, then from the above
  \begin{equation*}
    \emptyset \neq T_w(F_\alpha^{n}) T_w(U_q) \cap T_w(u) T_w(S) = F_\beta^n U_q \cap u' T_w(S).
  \end{equation*}
  Since $T_w(S) = \{ F_\beta^n | n\in \setN \}$ we have proved the
  lemma.
\end{proof}

We denote the Ore localization of $U_q$ in the above set by
$U_{q(F_\beta)}$.

\begin{lemma}
  \label{lemma:37}
  Let $p$ be Laurent polynomial. If
  \begin{equation*}
    p(q^{r_1},\dots,q^{r_n})=0
  \end{equation*}
  for all $r_1,\dots,r_n \in \setN$ then $p=0$.
\end{lemma}
\begin{proof}
  If $n=1$ we have a Laurent polynomial of one variable with
  infinitely many zero-points so $p=0$. Let $n>1$, then for constant
  $r_1 \in \setN$, $p(q^{r_1},-,\dots,-)$ is a Laurent polynomial in
  $n-1$ variables equal to zero in $(q^{r_2},\dots,q^{r_n})$ for all
  $r_2,\dots,r_n\in\setN$ so by induction $p(q^{r_1},c_2,\dots,c_n)=0$
  for all $c_2,\dots,c_n$. Now for arbitrary $c_2,\dots,c_n\in
  \setC^*$ we get $p(-,c_2,\dots,c_n)$ is a Laurent polynomial in one
  variable that is zero for all $q^{r_1}$, $r_1\in \setN$ hence
  $p(c_1,\dots,c_n)=0$ for all $c_1\in \setC^*$.
\end{proof}

The next lemma is crucial for the rest of the results in this
paper. We will use this result again and again.

\begin{lemma}
  \label{lemma:9}
  Let $\beta\in \Phi^+$ and let $F_\beta$ be a corresponding root
  vector. There exists automorphisms $\phi_{F_{\beta}
    ,b}:U_{q(F_\beta)}\to U_{q(F_\beta)}$ for each $b\in \setC^*$ such
  that $\phi_{F_{\beta} ,q^i}(u)=F_\beta^{-i} u F_\beta^{i}$ for
  $i\in\setZ$ and such that for $u\in U_{q(F_\beta)}$ the map $\setC^*
  \to U_{q(F_\beta)}$, $b\mapsto \phi_{F_{\beta} ,b}(u)$ is of the
  form $b \mapsto p(b)$ for some Laurent polynomial $p \in
  U_{q(F_\beta)} [X,X\inv]$.  Furthermore for $b,b'\in \setC^*$,
  $\phi_{F_\beta,b}\circ \phi_{F_\beta,b'} = \phi_{F_\beta,bb'}$.
\end{lemma}
\begin{proof}
  We can assume $\beta$ is simple since if $F_\beta=T_w(F_{\alpha'})$
  for some $\alpha'\in\Pi$ then we can just define the homomorphism on
  $T_w(E_{\alpha}),T_w(K_\alpha^{\pm 1}),T_w(F_\alpha)$ for
  $\alpha\in\Pi$ i.e. in this case we define $\phi_{F_{\beta} ,b}(u) =
  T_w( \phi_{\alpha',b}(T_w\inv(u)))$ where we extend $T_w$ to a
  homomorphism $T_w:U_{q(F_{\alpha'})}\to U_{q(F_\beta)}$ by
  $T_w(F_{\alpha'}\inv) = F_\beta\inv$.
  
  So $\beta$ is assumed simple. For $b\in \setC^*$ define
  $b_\beta=b^{(\beta|\beta)/2}$ i.e. $b_\beta=b$ if $\beta$ is short
  and $b_\beta=b^2$ when $\beta$ is long. We will define the map on
  the generators $E_\alpha,K_\alpha,F_\alpha$ for $\alpha\in \Pi$. If
  $\alpha=\beta$ the map is defined as follows:
  \begin{align*}
    \phi_{F_{\beta} ,b}(F_\beta^{\pm 1}) =& F_\beta^{\pm 1}
    \\
    \phi_{F_{\beta} ,b}(K_\beta^{\pm 1}) =& b_\beta^{\mp 2}
    K_\beta^{\pm 1}
    \\
    \phi_{F_{\beta} ,b}(E_\beta) =& E_\beta + F_\beta\inv
    \frac{(b_\beta-b_\beta\inv)(q_\beta b_\beta\inv K_\beta -
      q_\beta\inv b_\beta K_\beta\inv)}{(q_\beta-q_\beta\inv)^{2}}.
  \end{align*}

  Assume $\alpha\neq \beta$. Let $r=\left<\alpha,\beta^\vee
  \right>$. Note that $\operatorname{ad}(F_\beta^{-r+1})(F_\alpha)=0$
  because this is one of the defining relations of $U_q$. We define
  the map as follows:
  \begin{align*}
    \phi_{F_{\beta} ,b}(F_\alpha) =& \sum_{i=0}^{-r}
    b_\beta^{-r-i}q_\beta^{i(i+r)} \prod_{t=1}^i \frac{b_\beta
      q_{\beta}^{1-t}-b_\beta\inv q_{\beta}^{t-1}}{q_\beta^t -
      q_\beta^{-t}} F_\beta^{-i}
    \operatorname{ad}(F_\beta^i)(F_\alpha)
    \\
    \phi_{F_{\beta} ,b}(K_\alpha) =& b_\beta^{-r}K_\alpha
    =b^{-(\alpha|\beta)}K_\alpha
    \\
    \phi_{F_{\beta} ,b}(E_\alpha) =& E_\alpha.
  \end{align*}
  Note that if $b=q^{j}$ for some $j\in\setZ$ then $\prod_{t=1}^i
  \frac{b_\beta q_{\beta}^{1-t}-b_\beta\inv q_{\beta}^{t-1}}{q_\beta^t
    - q_\beta^{-t}}={j \brack i}_\beta$.  Since the map $b\mapsto
  \phi_{F_{\beta} ,b}(u)$ is of the form $b \mapsto \sum_{i=1}^r
  p_i(b) u_i$ with $p_i$ Laurent polynomial in $b$ for each generator
  of $U_q$ it is of this form for all $u\in U_q$. It's easy to check
  that $\phi_{F_{\beta} ,b}(u) = F_\beta^{-i} u F_\beta^i$ when $b=
  q^i$, $i \in \setN$. So $\phi_{F_{\beta} ,b}$ satisfies the
  generating relations of $U_q$ for $b= q^i$, $i \in \setN$. By
  Lemma~\ref{lemma:37} $\phi_{F_{\beta} ,b}$ must satisfy the
  generating relations for all $b\in \setC$.

  Consider the last claim of the lemma: Let $u\in U_q$, then by the
  above $b\mapsto \phi_{F_\beta,b}(u)$ is a Laurent polynomial and so
  $b \mapsto \phi_{F_\beta,bb'}(u)$ and
  $\phi_{F_\beta,b}(\phi_{F_\beta,b'}(u))$ for a constant $b'\in
  \setC^*$ is a Laurent polynomial as well. Now we know from above
  that for $b'=q^j$ for some $j\in \setZ$ and $i\in \setZ$:
  \begin{align*}
    \phi_{F_\beta,q^i} \circ \phi_{F_\beta,b'}(u) =& F_\beta^{-i}
    F_\beta^{-j} u F_\beta^j F_\beta^i
    \\
    =& F_\beta^{-i-j}u F_\beta^{i+j}
    \\
    =& \phi_{F_\beta,q^i q^j}(u)
  \end{align*}
  So $\phi_{F_\beta,b}(\phi_{F_\beta,q^j}(u))=\phi_{F_\beta,b q^j}(u)$
  for all $b\in \setC^*$ since both sides are Laurent polynomials in
  $b$ and they are equal in infinitely many points. In the same way we
  get the result for all $b'\in \setC$.
\end{proof}

Note that if $\beta$ is long then the above automorphism is a Laurent
polynomial in $b^2$. So if $b_1^2=b_2^2$ for $b_1,b_2\in \setC^*$ then
$\phi_{F_\beta,b_1} = \phi_{F_\beta,b_2}$. We could have defined
another automorphism $\phi'_{F_\beta,b} :=
\phi_{F_\beta,b^{2/(\beta|\beta)}}$ and proved the lemma above with
the modification that $\phi'_{F_{\beta},q_\beta^i}(u) = F_{\beta}^{-i}
u F_{\beta}^i$. The author has chosen the first option to avoid having
to write the $\beta$ in $q_\beta$ all the time in results like
Lemma~\ref{lemma:29} and Corollary~\ref{cor:6}. On the other hand this
choice means that we have to take some squareroots sometimes when
doing concrete calculations involving long roots see~e.g the proof
of~Lemma~\ref{lemma:30}. The choice of squareroot doesn't matter by
the above.

We can use the formulas in Section~\ref{sec:u_a-calculations} to find
the value of $\phi_{F_{\beta},b}(F_{\beta'})$ and
$\phi_{F_\beta,b}(E_{\beta'})$ for general root vectors $F_\beta$,
$F_{\beta'}$ and $E_{\beta'}$, $\beta,\beta'\in \Phi^+$.
\begin{prop}
  \label{prop:25}
  Let $s_{i_1}\dots s_{i_N}$ be a reduced expression of $w_0$ and
  define root vectors $F_{\beta_1},\dots,F_{\beta_N}$ and
  $E_{\beta_1},\dots,E_{\beta_N}$ using this expression
  (i.e. $F_{\beta_j}=T_{s_{i_1}}\dots
  T_{s_{i_{j-1}}}(F_{\alpha_{i_j}})$ and $E_{\beta_j}=T_{s_{i_1}}\dots
  T_{s_{i_{j-1}}}(E_{\alpha_{i_j}})$). Let $1\leq j<k\leq N$ and set
  $r=\left<\beta_k,\beta_j^\vee\right>$.
  \begin{align*}
    \phi_{F_{\beta_j},b}(F_{\beta_k}^n) =& \sum_{i\geq 0}
    q_{\beta_j}^{i(nr+i)} b_{\beta_j}^{-nr-i} \prod_{t=1}^i
    \frac{q_{\beta_j}^{1-t}b_{\beta_j} -
      q_{\beta_j}^{t-1}b_{\beta_j}\inv}{q_{\beta_j}^t-q_{\beta_j}^{-t}}
    F_{\beta_j}^{-i} \ad(F_{\beta_j}^i)(F_{\beta_k}^n)
    \\
    \phi_{F_{\beta_k},b}(F_{\beta_j}^n) =& \sum_{i\geq 0} (-1)^i
    q_{\beta_k}^{-i} b_{\beta_k}^{nr+i} \prod_{t=1}^i
    \frac{q_{\beta_k}^{1-t}b_{\beta_k} -
      q_{\beta_k}^{t-1}b_{\beta_k}\inv}{q_{\beta_k}^t-q_{\beta_k}^{-t}}
    F_{\beta_k}^{-i} \tilde{\ad}(F_{\beta_k}^i)(F_{\beta_j}^n)
    \\
    \phi_{F_{\beta_j},b}(E_{\beta_k}) =& \sum_{i\geq 0}
    b_{\beta_j}^{-i} \prod_{t=1}^i \frac{q_{\beta_j}^{1-t}b_{\beta_j}
      -
      q_{\beta_j}^{t-1}b_{\beta_j}\inv}{q_{\beta_j}^t-q_{\beta_j}^{-t}}
    F_{\beta_j}^{-i} u_i
    \\
    \phi_{F_{\beta_k},b}(E_{\beta_j}) =& \sum_{i\geq 0}
    b_{\beta_k}^{i} \prod_{t=1}^i \frac{q_{\beta_k}^{1-t}b_{\beta_k} -
      q_{\beta_k}^{t-1}b_{\beta_k}\inv}{q_{\beta_k}^t-q_{\beta_k}^{-t}}
    F_{\beta_k}^{-i} \tilde{u_i}
  \end{align*}
  for some $u_i,\tilde{u_i}\in U_q$ (independent of $b$) such that
  $u_i=\tilde{u_i}=0$ for $i\gg 0$.  In particular for any $j,k \in
  \{1,\dots,N\}$:
  \begin{align*}
    \phi_{F_{\beta_j},-1}(F_{\beta_k}) =
    (-1)^{(\beta_j|\beta_k)}F_{\beta_k}
    \\
    \phi_{F_{\beta_j},-1}(E_{\beta_k}) = E_{\beta_k}.
  \end{align*}
\end{prop}
Note that the sums are finite because of Lemma~\ref{lemma:22}.
\begin{proof}
  By Proposition~\ref{prop:17} we have for any $a \in \setN$
  \begin{align*}
    F_{\beta_k}^{n} F_{\beta_j}^{a} =& \sum_{i=0}^a
    q_{\beta_j}^{(i-a)(nr+i)} {a\brack i}_{\beta_j}
    F_{\beta_j}^{a-i}\ad(F_{\beta_j}^{i})(F_{\beta_k}^n)
    \\
    =& \sum_{i=0}^\infty q_{\beta_j}^{i(nr+i)}q_{\beta_j}^{-a(nr+i)}
    \prod_{t=1}^i \frac{q_{\beta_j}^{1-t}q_{\beta_j}^a -
      q_{\beta_j}^{t-1}q_{\beta_j}^{-a}}{q_{\beta_j}^t-q_{\beta_j}^{-t}}
    F_{\beta_j}^{a-i}\ad(F_{\beta_j}^{i})(F_{\beta_k}^n).
  \end{align*}
  Here we use the fact that ${a \brack i}_{\beta_j} = 0$ for $i>a$. So
  \begin{align*}
    F_{\beta_j}^{-a}F_{\beta_k}^{n} F_{\beta_j}^{a} =& \sum_{i\geq 0}
    q_{\beta_j}^{i(nr+i)}q_{\beta_j}^{-a(nr+i)} \prod_{t=1}^i
    \frac{q_{\beta_j}^{1-t}q_{\beta_j}^a -
      q_{\beta_j}^{t-1}q_{\beta_j}^{-a}}{q_{\beta_j}^t-q_{\beta_j}^{-t}}
    F_{\beta_j}^{-i}\ad(F_{\beta_j}^{i})(F_{\beta_k}^n).
  \end{align*}
  Now using the fact that $\phi_{F_{\beta_j},q^{a}}(F_{\beta_k}^n)=
  F_{\beta_j}^{-a}F_{\beta_k}^{n} F_{\beta_j}^{a}$, the fact that
  $\phi_{F_{\beta_j},b}(F_{\beta_k}^n)$ is Laurent polynomial and
  Lemma~\ref{lemma:37} we get the first identity.  The second identity
  is shown similarly by using the second identity in
  Proposition~\ref{prop:17}.

  To prove the last two identities we need to calculate
  $F_{\beta_j}^{-a}E_{\beta_k}^n F_{\beta_j}^a$
  (resp. $F_{\beta_k}^{-a}E_{\beta_j}^n F_{\beta_k}^a$) for any
  $a\in\setN$. Let $w=s_{i_1}\cdots s_{i_{j-1}}$ and
  $w'=s_{i_{j+1}}\cdots s_{i_{k-1}}$. Then
  $E_{\beta_j}=T_w(E_{\alpha_{i_j}})$ and
  $F_{\beta_k}=T_wT_{s_{i_j}}T_{w'}(F_{\alpha_{i_k}})$. 
  \begin{align*}
    E_{\beta_j} F_{\beta_k}^a =& T_w \left(E_{\alpha_{i_j}}
      T_{s_{i_j}}T_{w'}(F_{\alpha_{i_k}}^a)\right)
    \\
    =& T_wT_{s_{i_j}}\left( -K_{\alpha_{i_j}}\inv F_{\alpha_{i_j}}
      T_{w'}(F_{\alpha_{i_k}}^a) \right).
  \end{align*}
  Expand $s_{i_{j}}\cdots s_{i_N}$ from the right to a reduced
  expression $s_{i_j}\cdots s_{i_N}s_{m_1}\cdots s_{m_{j-1}}$ of
  $w_0$. Do the same with $s_{i_{j+1}}\cdots s_{i_N}s_{m_1}\cdots
  s_{m_{j-1}}$ to get a reduced expression $s_{i_{j+1}}\cdots
  s_{i_N}s_{m_1}\cdots s_{m_{j}}$. We claim that if we use the reduced
  expression $s_{i_{j+1}}\cdots s_{i_N}s_{m_1}\cdots s_{m_{j}}$ to
  construct roots $\beta_1'\dots,\beta_N'$ and root vectors
  $F_{\beta_j'}'$ then $F_{\beta_N'}' = T_{s_{i_{j+1}}}\cdots
  T_{s_{i_N}}T_{s_{m_1}}\cdots T_{s_{m_{j-1}}}(F_{\alpha_{m_j}}) =
  F_{\alpha_{i_j}}$. This is easy to see since $\beta_N'$ is positive
  but $s_{i_j}\beta_N' = w_0(\alpha_{m_j})<0$. We have
  $T_{w'}(F_{\alpha_{i_k}}^a) = F_{\beta_{k-j}'}^a$. Since $k-j<N$ we
  can use what we just calculated above: (set $d=k-j$)
  \begin{align*}
    F_{\beta_{d}'}'^{-a}F_{\beta_N'}' F_{\beta_{d}'}'^{a} =&
    \sum_{i\geq 0} q_{\beta_{d}'}^{i(r+i)}q_{\beta_{d}'}^{-a(r+i)}
    \prod_{t=1}^i \frac{q_{\beta_{d}'}^{1-t}q_{\beta_{d}'}^a -
      q_{\beta_{d}'}^{t-1}q_{\beta_d'}^{-a}}{q_{\beta_d'}^t-q_{\beta_d'}^{-t}}
    F_{\beta_d'}'^{-i}\ad(F_{\beta_d'}'^{i})(F_{\beta_N'}').
  \end{align*}
  so
  \begin{align*}
    F_{\beta_k}^{-a} E_{\beta_j} F_{\beta_k}^a =& K_{\beta_j}T_w
    T_{s_{i_j}}\left( \sum_{i\geq 0}
      q_{\beta_{d}'}^{i(r+i)}q_{\beta_{d}'}^{-ai} \prod_{t=1}^i
      \frac{q_{\beta_{d}'}^{1-t}q_{\beta_{d}'}^a -
        q_{\beta_{d}'}^{t-1}q_{\beta_d'}^{-a}}{q_{\beta_d'}^t-q_{\beta_d'}^{-t}}
      F_{\beta_d'}'^{-i}\ad(F_{\beta_d'}'^{i})(F_{\beta_N'}') \right).
  \end{align*}
  This shows the third identity. The fourth is shown similarly.

  Setting $b=-1$ in the above formulas we get the last claim of the
  proposition.
\end{proof}

\begin{defn}
  Let $M$ be a $U_{q(F_\beta)}$-module. We define a new module
  $\phi_{F_{\beta} ,b}.M$ (with elements $\phi_{F_{\beta} ,b}.m$,
  $m\in M$) where the module structure is given by composing with the
  above automorphism $\phi_{F_{\beta} ,b}$. -- i.e. $u\phi_{F_{\beta}
    ,b}.m = \phi_{F_{\beta} ,b}.\phi_{F_{\beta} ,b}(u)m$ for all $u\in
  U_{q(F_\beta)}$, $m\in M$.
\end{defn}
Note that $\wt \phi_{F_{\beta} ,b}.M = b^{-\beta}\wt M$ where
$b^{-\beta}$ is the character such that $b^{-\beta}(K_\alpha) =
b^{-\left( \alpha| \beta \right)}$ for $\alpha\in\Pi$.

The homomorphisms from Lemma~\ref{lemma:9} preserve degree so we can
restrict to $(U_{q(F_\beta)})_0$ which we will do in the next
lemma. The twist of a $(U_{q(F_\beta)})_0$-module is defined in the
same way as the definition above.  It is an important fact of these
twists that they do not neccesarily preserve simplicity of
$U_q$-modules: If $L$ is a $U_{q(F_\beta)}$-module that is simple as a
$U_q$-module then $\phi_{F_\beta,b}.L$ can be nonsimple as a
$U_q$-module for some $b\in \setC^*$, see e.g.~Lemma~\ref{lemma:24}.

\begin{lemma}
  \label{lemma:29}
  Let $M$ be a $U_{q(F_\beta)}$-module. Let $i\in \setZ$. Then
  \begin{equation*}
    \phi_{F_{\beta} , q^i}.M \iso M
  \end{equation*}
  as $U_{q(F_\beta)}$-modules.  Furthermore for $\lambda\in \wt M$ we
  have an isomorphism of $(U_{q(F_\beta)})_0$-modules:
  \begin{equation*}
    \phi_{F_{\beta}, q^i}.M_\lambda \iso M_{q^{-i \beta} \lambda}.
  \end{equation*}
\end{lemma}
\begin{proof}
  The isomorphism in both cases is given by $\phi_{F_{\beta} , q^i}.m
  \mapsto F_\beta^im$, $\phi_{F_{\beta} , q^i}.M\to M$. The inverse is
  given by multiplying by $ F_\beta^{-i}$. By Lemma~\ref{lemma:9}: For
  $u\in U_{q(F_\beta)}$, $m\in M$; $\phi_{F_{\beta} ,q^i}(u) =
  F_\beta^{-i}uF_\beta^{i}$ so $u \phi_{F_{\beta} ,q^i}.m =
  \phi_{F_{\beta} ,q^i}.F_\beta^{-i}uF_\beta^{i}m \mapsto F_\beta^i
  F_\beta^{-i}uF_\beta^{i}m = uF_\beta^i m$. Thus the given map is a
  homomorphism.
\end{proof}

\begin{defn}
  \label{def:commuting_roots}
  Let $\Sigma\subset \Phi^+$. Then $\Sigma$ is called a set of
  commuting roots if there exists an ordering of the roots in
  $\Sigma$; $\Sigma=\{\beta_1,\dots,\beta_s\}$ such that for some
  reduced expression of $w_0$ and corresponding construction of the
  root vectors $F_\beta$ we have: $[F_{\beta_j},F_{\beta_i}]_q=0$ for
  $1\leq i < j \leq s$.

  For any subset $I\subset \Pi$, let $Q_I$ be the subgroup of $Q$
  generated by $I$, $\Phi_I$ the root system generated by $I$ ,
  $\Phi_I^+=\Phi^+ \cap \Phi_I$ and $\Phi_I^- = -\Phi_I^+$.
\end{defn}

The following three lemmas have exactly the same proofs as their
counterparts (\cite[Lemma~5.6]{DHP1}, \cite[Lemma~5.22]{DHP1} and
\cite[Lemma~5.23]{DHP1}) in the root of unity case in~\cite{DHP1}. We
include the proofs here as well for completeness.

We have the following equivalent of Lemma~4.1 in~\cite{Mathieu}:
\begin{lemma}
  \label{lemma:7}
  \begin{enumerate}
  \item Let $I\subset \Pi$ and let $\alpha\in I$. There exists a set
    of commuting roots $\Sigma'\subset \Phi_I^+$ with
    $\alpha\in\Sigma'$ such that $\Sigma'$ is a basis of $Q_I$.
  \item Let $J,F$ be subsets of $\Pi$ with $F\neq \Pi$. Let $\Sigma'
    \subset \Phi_J^+ \backslash \Phi_{J\cap F}^+$ be a set of
    commuting roots which is a basis of $Q_J$. There exists a set of
    commuting roots $\Sigma$ which is a basis of $Q$ such that
    $\Sigma' \subset \Sigma \subset \Phi^+ \backslash \Phi_F^+$
  \end{enumerate}
\end{lemma}
\begin{proof}
  The first part of the proof is just combinatorics of the root system
  so it is identical to the first part of the proof of Lemma~4.1
  in~\cite{Mathieu}: Let us first prove assertion $2.$: If $J$ is
  empty we can choose $\alpha\in \Pi\backslash F$ and replace $J$ and
  $\Sigma'$ by $\{\alpha\}$. So assume from now on that $J\neq
  \emptyset$. Set $J' = J \backslash F$, $p= |J'|$, $q=|J|$. Let
  $J_1,\dots, J_k$ be the connected components of $J$ and set $J_i' =
  J' \cap J_i$, $F_i = F \cap J_i$, and $\Sigma_i' = \Sigma\cap
  \Phi_{J_i}$, for any $1\leq i \leq k$. Since $\Sigma'\subset \Phi_J$
  is a basis of $Q_J$, each $\Sigma_i'$ is a basis of $Q_{J_i}$. Since
  $\Sigma_i'$ lies in $\Phi_{J_i}^+ \backslash \Phi_{F_i}^+$, the set
  $J_i'=J_i\backslash F_i$ is not empty. Hence $J'$ meets every
  connected component of $J$. Therefore we can write
  $J=\{\alpha_1,\dots, \alpha_q\}$ in such a way that
  $J'=\{\alpha_1,\dots,\alpha_p\}$ and, for any $s$ with $p+1\leq
  s\leq q$, $\alpha_s$ is connected to $\alpha_i$ for some
  $i<s$. Since $\Pi$ is connected we can write $\Pi \backslash J =
  \{\alpha_{q+1},\dots, \alpha_n\}$ in such a way that for any $s\geq
  q+1$, $\alpha_s$ is connected to $\alpha_i$ for some $i$ with $1\leq
  i < s$. So $\Pi = \{\alpha_1,\dots, \alpha_n\}$ such that for $s> p$
  we have that $\alpha_s$ is connected to some $\alpha_i$ with $1\leq
  i < s$.
  
  Let $\Sigma' = \{\beta_1,\dots,\beta_q\}$. We will define
  $\beta_{q+1},\dots,\beta_l$ inductively such that for each $s\geq
  q$, $\{\beta_1,\dots,\beta_s\}$ is a commuting set of roots which is
  a basis of $\Phi_{\{\alpha_1,\dots,\alpha_s\}}$. So assume we have
  defined $\beta_1,\dots,\beta_s$. Let $w_s$ be the longest word in
  $s_{\alpha_1},\dots,s_{\alpha_s}$ and let $w_{s+1}$ be the longest
  word in $s_{\alpha_1},\dots,s_{\alpha_{s+1}}$. Choose a reduced
  expression of $w_s$ such that the corresponding root vectors
  $\{F_{\beta_k}\}_{k=1}^s$ satisfies $[F_{\beta_j},F_{\beta_i}]_q=0$
  for $i<j$. Choose a reduced expression of $w_{s+1}=w_{s}w'$ starting
  with the above reduced expression of $w_s$. Let $N_s$ be the length
  of $w_s$ and $N_{s+1}$ be the length of $w_{s+1}$. So we get an
  ordering of the roots generated by
  $\{\alpha_1,\dots,\alpha_{s+1}\}$:
  $\Phi_{\{\alpha_1,\dots,\alpha_{s+1}\}}^+=\{\gamma_1,\dots,\gamma_{N_s},\gamma_{N_s+1},\dots,\gamma_{N_{s+1}}\}$
  with
  $\Phi_{\{\alpha_1,\dots,\alpha_{s}\}}^+=\{\gamma_1,\dots,\gamma_{N_s}\}$. Consider
  $\gamma_{N_s+1}=w_s(\alpha_{s+1})$. Since $w_s$ only consists of the
  simple reflections corresponding to $\alpha_1,\dots,\alpha_s$ we
  must have that $\gamma_{N_s+1}=\alpha_{s+1}+\sum_{i=1}^s m_i
  \alpha_i$ for some coefficients $m_i\in \setN$. So
  $\{\beta_1,\dots,\beta_s,\gamma_{N_s+1}\}$ is a basis of
  $\Phi_{\{\alpha_1,\dots,\alpha_{s+1}\}}$. From Theorem~\ref{thm:DP}
  we get for $1\leq i \leq s$
  \begin{equation*}
    [F_{\gamma_{N_s+1}},F_{\beta_i}]_q \in \spa{\setC}{F_{\gamma_{N_s}}^{a_{N_s}}\cdots F_{\gamma_{2}}^{a_2}|a_i\in \setN\}}
  \end{equation*}
  But since $\{\gamma_1,\dots,\gamma_{N_s}\} =
  \Phi^+_{\{\alpha_1,\dots,\alpha_s\}}$ and since
  $\gamma_{N_s+1}=\alpha_{s+1}+\sum_{i=1}^s m_i \alpha_i$ we get
  $[F_{\gamma_{N_s+1}},F_{\beta_i}]_q=0$.

  All that is left is to show that $\gamma_{N_s+1}\not \in \Phi_F$. By
  the above we must have that $\alpha_{s+1}$ is connected to some
  $\alpha_i\in J'$. We will show that the coefficient of $\alpha_i$ in
  $\gamma_{N_s+1}$ is nonzero. Otherwise $(\gamma_{N_s+1}|\alpha_i)<0$
  and so $\gamma_{N_s+1}+\alpha_i\in
  \Phi_{\{\alpha_1,\dots,\alpha_{s+1}\}}$ and by Theorem~1
  in~\cite{Papi}, $\gamma_{N_s+1}+\alpha_i = \gamma_{j}$ for some $1 <
  j \leq s$. This is impossible since $\gamma_{N_s+1}+\alpha_i\not \in
  \Phi_{\{\alpha_1,\dots,\alpha_{s}\}}$. So we can set $\beta_{s+1} =
  \gamma_{N_s+1}$ and the induction step is finished.

  To prove assertion $1.$ it can be assumed that $I=\Pi$. Thus
  assertion $1.$ follows from assertion $2.$ with $J=\{\alpha\}$ and
  $F=\emptyset$.
\end{proof}

\begin{lemma}
  \label{lemma:6}
  Let $L\in \mathcal{F}$ be a simple module. Then there exists a $w\in
  W$ such that $w(F_L\backslash F_L^s) \subset \Phi^+$ and
  $w(T_L\backslash T_L^s) \subset \Phi^-$.
\end{lemma}
\begin{proof}
  Since $L$ is simple we have $\Phi= F_L\cup T_L$. By
  Proposition~\ref{prop:8} $F_L$ and $T_L$ are closed subsets. Then
  Lemma~4.16 in~\cite{Fernando} tells us that there exists a basis $B$
  of the root system $\Phi$ such that the antisymmetrical part of
  $F_L$ is contained in the positive roots $\Phi_B^+$ corresponding to
  the basis $B$ and the antisymmetrical part of $T_L$ is contained in
  the negative roots $\Phi_B^-$ corresponding to the basis. Since all
  bases of a root system are $W$-conjugate the claim follows.
\end{proof}

\begin{lemma}
  \label{lemma:26}
  Let $L$ be an infinite dimensional admissible simple module. Let
  $w\in W$ be such that $w(F_L\backslash F_L^s) \subset \Phi^+$. Let
  $\alpha\in \Pi$ be such that $-\alpha\in w(T_L)$ (such an $\alpha$
  always exists). Then there exists a commuting set of roots $\Sigma$
  with $\alpha\in \Sigma$ which is a basis of $Q$ such that $ -\Sigma
  \subset w(T_L)$.
\end{lemma}
\begin{proof}
  Set $L' = {^w}L$. Since $w(T_L) = T_{{^{w}}L}=T_{L'}$ we will just
  work with $L'$. Then $F_{L'}\backslash F_{L'}^s \subset \Phi^+$.
  
  Note that it is always possible to choose a simple root $\alpha\in
  -T_{L'}$ since $L'$ is infinite dimensional: If this was not
  possible we would have $\Phi^- \subset F_{L'}$. But since
  $F_{L'}\backslash F_{L'}^s \subset \Phi^+$ this would imply $F_L =
  \Phi$.

  Set $F = F_{L'}^s \cap \Pi$. Since $L'$ is infinite dimensional
  $F\neq \Pi$.  By Lemma~\ref{lemma:7} $2.$ applied with
  $J=\{\alpha\}=\Sigma'$ there exists a commuting set of roots
  $\Sigma$ that is a basis of $Q$ such that $\Sigma \subset \Phi^+
  \backslash \Phi^+_F$. Since $F_{L'}\backslash F_{L'}^s \subset
  \Phi^+$ we have $\Phi^- = T_{L'}^- \cup (F_{L'}^s)^-$.  To show
  $-\Sigma \subset T_{L'}$ we show $\left( \Phi^- \backslash \Phi^-_F
  \right)\cap F_{L'}^s=\emptyset$ or equivalently $(F_{L'}^s)^-
  \subset \Phi_F^-$.

  Assume $\beta\in F_{L'}^s\cap \Phi^+$, $\beta = \sum_{\alpha\in \Pi}
  a_\alpha \alpha$, $a_\alpha\in\setN$. The height of $\beta$ is the
  sum $\sum_{\alpha\in \Pi} a_\alpha$. We will show by induction on
  the height of $\beta$ that $-\beta\in \Phi_F^-$. If the height of
  $\beta$ is $1$ then $\beta$ is a simple root and so $\beta\in
  F$. Clearly $-\beta\in \Phi_F^-$ in this case. Assume the height of
  $\beta$ is greater than $1$. Let $\alpha'\in \Pi$ be a simple root
  such that $\beta-\alpha'$ is a root. There are two possibilities:
  $-\alpha' \in T_{L'}$ or $\pm \alpha' \in F_{L'}^s$.
  
  In the first case where $-\alpha'\in T_{L'}$ we must have $-\beta +
  \alpha' \in F_{L'}^s$ since if $-\beta + \alpha' \in T_L$ then
  $-\beta = (-\beta + \alpha') - \alpha' \in T_{L'}$. So
  $\beta-\alpha' \in F_{L'}^s$ and $\beta \in F_{L'}^s$. Since
  $F_{L'}$ is closed (Proposition~\ref{prop:8}) we get $-\alpha' =
  (\beta-\alpha') - \beta \in F_L$ which is a contradiction. So the
  first case ($-\alpha' \in T_{L'}$) is impossible.
  
  In the second case since $F_{L'}$ is closed we get $\pm (\beta -
  \alpha') \in F_{L'}$ i.e. $\beta- \alpha' \in F_{L'}^s$. By the
  induction $-(\beta-\alpha') \in \Phi_F^-$ and since $-\beta =
  -(\beta-\alpha') - \alpha'$ we are done.
\end{proof}

\begin{prop}
  \label{prop:11}
  Let $\Sigma = \{\beta_1,\dots,\beta_r\}$ be a set of commuting
  roots.  The set $\{q^a F_{\beta_1}^{a_1}\cdots
  F_{\beta_r}^{a_r}|a_i\in \setN,a\in \setZ \}$ is an Ore subset of
  $U_q$.
\end{prop}
\begin{proof}
  We will prove it by induction over $r$. $r=1$ is
  Lemma~\ref{lemma:27}.
  
  Let $S_r = \{q^a F_{\beta_1}^{a_1}\cdots F_{\beta_r}^{a_r}|a_i\in
  \setN, a\in \setZ \}$. Let $a_1,\dots,a_r \in \setN$, $a\in \setZ$
  and $u\in U_q$, then we need to show that
  \begin{equation}
    \label{eq:2}
    q^a F_{\beta_1}^{a_1}\cdots F_{\beta_{r}}^{a_{r}} U_q \cap u S_{r} \neq \emptyset.
  \end{equation}
  and
  \begin{equation}
    \label{eq:3}
    U_q q^a F_{\beta_1}^{a_1}\cdots F_{\beta_{r}}^{a_{r}} \cap S_{r} u \neq \emptyset.
  \end{equation}
  By Lemma~\ref{lemma:27} there exists $\tilde{u}\in U_q$ and $b \in
  \setN$ such that
  \begin{equation}
    \label{eq:1}
    F_{\beta_r}^{a_r}\tilde{u} = u
    F_{\beta_r}^{b}.
  \end{equation}
  By induction
  \begin{equation*}
    q^a F_{\beta_1}^{a_1}\cdots F_{\beta_{r-1}}^{a_{r-1}} U_q \cap \tilde{u} S_{r-1} \neq \emptyset
  \end{equation*}
  so
  \begin{equation*}
    q^a F_{\beta_r}^{a_r} F_{\beta_1}^{a_1}\cdots F_{\beta_{r-1}}^{a_{r-1}} U_q \cap F_{\beta_r}^{a_r}\tilde{u} S_{r-1} \neq \emptyset
  \end{equation*}
  Since $\Sigma$ is a set of commuting roots $F_{\beta_r}^{a_r}
  F_{\beta_1}^{a_1}\cdots F_{\beta_{r-1}}^{a_{r-1}} = q^k
  F_{\beta_1}^{a_1}\cdots F_{\beta_{r-1}}^{a_{r-1}}F_{\beta_r}^{a_r}$
  for some $k\in \setZ$. Using this and~\eqref{eq:1} we get
  \begin{equation*}
    \emptyset \neq q^{a+k}F_{\beta_1}^{a_1}\cdots F_{\beta_{r}}^{a_{r}} U_q \cap u F_{\beta_r}^{b}S_{r-1} \subset q^a F_{\beta_1}^{a_1}\cdots F_{\beta_{r}}^{a_{r}} U_q \cap u S_{r}
  \end{equation*}
  where $F_{\beta_r}^bS_{r-1}\subset S_r$ because $F_{\beta_r}$
  q-commutes with all the other root vectors.

  \eqref{eq:3} is shown similarly.
\end{proof}

\begin{lemma}
  \label{lemma:16}
  Let $\nu \in X$ and let $\Sigma=\{\beta_1,\dots,\beta_n\}$ be a
  basis of $Q$. Then there exists $\mathbf{b}=(b_1,\dots,b_n)\in
  (\setC^*)^n$ such that
  \begin{equation*}
    \nu = b_1^{\beta_1}b_2^{\beta_2}\cdots b_n^{\beta_n}
  \end{equation*}
  and there are only finitely many different $\mathbf{b}\in
  (\setC^*)^n$ satisfying this.
\end{lemma}
\begin{proof}
  If $\gamma_1,\gamma_2\in X$ satisfy $\gamma_1(K_{\beta_i})=
  \gamma_2(K_{\beta_i})$ for $i=1,\dots,n$ then $\gamma_1=\gamma_2$
  because $\{\beta_1,\dots,\beta_n\}$ is a basis of $Q$. Since for
  $a_1,\dots,a_n\in\setC^*$, $a_1^{\beta_1}a_2^{\beta_2}\cdots
  a_n^{\beta_n}(K_{\beta_i}) =
  a_1^{\left(\beta_i|\beta_1\right)}a_2^{\left(\beta_i|\beta_2\right)}\cdots
  a_n^{\left(\beta_i|\beta_n \right)}$ we have to solve the system in
  $n$ unknown variables $x_1,\dots,x_n$:
  \begin{align*}
    x_1^{\left(\beta_1|\beta_1\right)}x_2^{\left(\beta_1|\beta_2\right)}\cdots
    x_n^{\left(\beta_1|\beta_n \right)} =& \nu(K_{\beta_1})
    \\
    x_1^{\left(\beta_2|\beta_1 \right)}x_2^{\left(\beta_2|\beta_2
      \right)}\cdots x_n^{\left(\beta_2|\beta_n \right)} =&
    \nu(K_{\beta_2})
    \\
    \vdots&
    \\
    x_1^{\left(\beta_n|\beta_1\right)}x_2^{\left( \beta_n| \beta_2
      \right)}\cdots x_n^{\left(\beta_n|\beta_n \right)} =&
    \nu(K_{\beta_n}).
  \end{align*}
  Let $c_j\in \setC$, $j=1,\dots,n$ be such that
  $\nu(K_{\beta_j})=e^{c_j}$. There is a choice here since any
  $c_j+2k\pi i$, $k\in\setZ$ could be chosen instead. Consider the
  linear system in $n$ unknowns $X_1,\dots,X_n$
  \begin{align*}
    \left(\beta_1|\beta_1 \right) X_1+\left( \beta_1| \beta_2 \right)
    X_2 \cdots \left( \beta_1 | \beta_n \right)X_n =& c_1
    \\
    \left(\beta_2|\beta_1 \right) X_1+\left( \beta_2 | \beta_2
    \right)X_2 \cdots \left( \beta_2| \beta_n \right)X_n =& c_2
    \\
    \vdots&
    \\
    \left(\beta_n|\beta_1\right) X_1+\left( \beta_n| \beta_2
    \right)X_2 \cdots \left( \beta_n| \beta_n \right)X_n =& c_n.
  \end{align*}
  This system has a unique solution $a_1,\dots,a_n\in \setC$ since the
  matrix $(\left(\beta_i|\beta_j\right))_{i,j}$ is invertible. So $x_i
  = e^{a_i}$ is a solution to the above system. Any other solution to
  the original system corresponds to making a different choice when
  taking the logarithm of $\nu(K_{\beta_i})$. So another solution
  would be of the form $x_i = e^{a_i + a'_i}$ where $a'_i$,
  $i=1,\dots,n$ is a solution to a system of the form:
  \begin{align*}
    \left(\beta_1|\beta_1 \right) X_1+\left( \beta_1| \beta_2 \right)
    X_2 \cdots \left( \beta_1 | \beta_n \right)X_n =& 2k_1\pi i
    \\
    \left(\beta_2|\beta_1 \right) X_1+\left( \beta_2 | \beta_2
    \right)X_2 \cdots \left( \beta_2| \beta_n \right)X_n =& 2k_2 \pi i
    \\
    \vdots&
    \\
    \left(\beta_n|\beta_1\right) X_1+\left( \beta_n| \beta_2
    \right)X_2 \cdots \left( \beta_n| \beta_n \right)X_n =& 2k_n \pi
    i.
  \end{align*}
  for some $k_1,\dots,k_n\in \setZ$. Since $A= (\left(\beta_i|\beta_j
  \right))_{i,j}$ is a matrix with only integer coefficients we have
  $A\inv = \frac{1}{\det A} \tilde{A}$ for some $\tilde{A}$ with only
  integer coefficients. So the solution to the system above is integer
  linear combinations in $\frac{2k_i \pi i}{\det A}$, $i=1,\dots,n$
  hence $\{(e^{a'_1},\dots,e^{a'_n})|(a'_1,\dots, a'_n) \text{ is a
    solution to the above system} \}$ has fewer than $n \det A$
  elements so it is a finite set.
\end{proof}

In the next definition we would like to compose the $\phi$'s for
different $\beta$. In particular let
$\Sigma=\{\beta_1,\dots,\beta_n\}$ be a set of commuting roots and
$F_{\beta_1},\dots,F_{\beta_n}$ corresponding root vectors. Let
$F_\Sigma:=\{q^a F_{\beta_1}^{a_1}\cdots F_{\beta_n}^{a_n}|a_i\in
\setN,a\in\setZ\}$ and let $U_{q(F_\Sigma)}$ be the Ore localization
in $F_\Sigma$. For $i<j$ we have
\begin{equation*}
  F_{\beta_i}^{-k} F_{\beta_j} F_{\beta_i}^{k} = q^{-k(\beta_i|\beta_j)} F_{\beta_j}
\end{equation*}
or equivalently $\phi_{F_{\beta_i},q^k}(F_{\beta_j}) =
\left(q^k\right)^{-(\beta_i|\beta_j)} F_{\beta_j}$. This implies
$\phi_{F_{\beta_i},b}(F_{\beta_j})=b^{-(\beta_i|\beta_j)} F_{\beta_j}$
for $b\in \setC^*$ because $b\mapsto
\phi_{F_{\beta_i},b}(F_{\beta_j})$ is Laurent polynomial. Similarly
$\phi_{F_{\beta_j},b}(F_{\beta_i})=b^{(\beta_i|\beta_j)}
F_{\beta_i}$. This shows that we can define
$\phi_{F_\beta,b}(F_{\beta'}\inv) = \phi_{F_\beta,b}(F_{\beta'})\inv$
for $\beta,\beta'\in\Sigma$ extending $\phi_{F_\beta,b}$ to a
homomorphism $U_{q(F_\Sigma)}\to U_{q(F_\Sigma)}$. Also note that the
$\phi$'s commute because
\begin{align*}
  F_{\beta_i}^{-k_1}F_{\beta_j}^{-k_2} u F_{\beta_j}^{k_2}
  F_{\beta_i}^{k_1} =& q^{k_1k_2(\beta_i|\beta_j)}
  F_{\beta_j}^{-k_2}F_{\beta_i}^{-k_1} u
  q^{-k_1k_2(\beta_i|\beta_j)}F_{\beta_i}^{k_1}F_{\beta_j}^{k_2}
  \\
  =& F_{\beta_j}^{-k_2}F_{\beta_i}^{-k_1} u
  F_{\beta_i}^{k_1}F_{\beta_j}^{k_2}
\end{align*}

\begin{defn}
  \label{def:twist_by_weight}
  Let $\Sigma=\{\beta_1,\dots,\beta_r\}$ be a set of commuting roots
  and let $F_{\beta_1},\dots,F_{\beta_r}$ be corresponding root
  vectors such that $[F_{\beta_j},F_{\beta_i}]_q=0$ for $i<j$.  Let
  $U_{q(F_\Sigma)}$ denote the Ore localization of $U_q$ in the Ore
  set $F_\Sigma:=\{q^a F_{\beta_1}^{a_1}\cdots
  F_{\beta_n}^{a_r}|a_i\in \setN,a\in\setZ\}$. Said in words we invert
  $F_\beta$ for all $\beta\in \Sigma$.
  
  Let $M$ be a $U_{q}$-module. We define $M_{F_\Sigma}$ to be the
  $U_{q(F_\Sigma)}$-module $U_{q(F_\Sigma)} \tensor_{U_q} M$.  Let
  $\mathbf{b} = (b_1,\dots,b_r)\in (\setC^*)^r$. Then for a
  $U_{q(F_\Sigma)}$-module $N$ we define
  $\phi_{F_{\Sigma},\mathbf{b}}.N$ to be the twist of the module by
  $\phi_{F_{\beta_1},b_1}\circ \dots \circ \phi_{F_{\beta_r},b_r}$.

  For $\mathbf{i}=(i_1,\dots,i_r)\in \setZ^r$ define $q^{\mathbf{i}} =
  (q^{i_1},\dots,q^{i_r})\in (\setC^*)^r$ and
  $q^{\setZ^r}=\{q^{\mathbf{i}}|\mathbf{i}\in\setZ^r\}\subset
  (\setC^*)^r$.

  For $\mathbf{b}=(b_1,\dots,b_r)\in (\setC^*)^r$ we set
  $\mathbf{b}^\Sigma := b_1^{\beta_1}\cdots b_r^{\beta_r}\in X$. If
  $\Sigma$ is a basis of $Q$ then the map $\mathbf{b} \mapsto
  \mathbf{b}^\Sigma$ is surjective by Lemma~\ref{lemma:16} but not
  neccesarily injective.
\end{defn}

\begin{cor}[to Lemma~\ref{lemma:29}]
  \label{cor:6}
  Let $\Sigma$ be a set of commuting roots that is a $\setZ$ basis of
  $Q$, let $F_\Sigma$ be an Ore subset corresponding to $\Sigma$, let
  $M$ be a $U_{q(F_\Sigma)}$-module and let
  $\mathbf{i}=(i_1,\dots,i_n)\in \setZ^n$. Then
  \begin{equation*}
    \phi_{F_{\Sigma} ,q^{\mathbf{i}}}.M \iso M
  \end{equation*}
  as $U_{q(F_\Sigma)}$-modules.  Furthermore for $\lambda\in \wt M$ we
  have an isomorphism of $(U_{q(F_\Sigma)})_0$-modules:
  \begin{equation*}
    \phi_{F_\Sigma,q^{\mathbf{i}}}.M_\lambda \iso M_{\left(q^{-\mathbf{i}}\right)^\Sigma \lambda} = M_{q^{-\mu}\lambda}
  \end{equation*}
  where $\mu = \sum_{j=1}^n i_j \beta_j$.
\end{cor}
\begin{proof}
  The corollary follows from Lemma~\ref{lemma:29} because $\Sigma$ is
  a $\setZ$ basis of $Q$.
\end{proof}

\begin{defn}
  Let $L$ be an admissible module of degree $d$. The essential support
  of $L$ is defined as
  \begin{equation*}
    \Suppess(L) := \{ \lambda \in \wt L | \dim L_\lambda = d\}
  \end{equation*}
\end{defn}

\begin{lemma}
  \label{lemma:13}
  Let $M$ be an admissible module. Let $\Sigma\subset \Phi^+$ be a set
  of commuting roots and $F_\Sigma$ a corresponding Ore subset. Assume
  $-\Sigma\subset T_M$. Then for $\lambda \in X$:
  \begin{equation*}
    \dim (M_{F_\Sigma})_\lambda = \max_{\mu\in \setZ \Sigma}\{\dim M_{q^\mu\lambda}\}
  \end{equation*}
  and if $\dim M_\lambda = \max_{\mu\in \setZ \Sigma}\{\dim
  M_{q^\mu\lambda}\}$ then $(M_{F_\Sigma})_\lambda \iso M_\lambda$ as
  $(U_q)_0$-modules.

  In particular if $\Sigma\subset T_M$ as well then $M_{F_\Sigma}\iso
  M$ as $U_q$-modules.
\end{lemma}
Compare to Lemma~4.4(ii) in~\cite{Mathieu}.
\begin{proof}
  We have $\Sigma=\{\beta_1,\dots,\beta_r\}$ for some
  $\beta_1,\dots,\beta_r\in \Phi^+$ and corresponding root vectors
  $F_{\beta_1},\dots,F_{\beta_r}$. Let $\lambda\in X$ and set
  $d=\max_{\mu\in \setZ \Sigma}\{\dim M_{q^\mu\lambda}\}$. Let $V$ be
  a finite dimensional subspace of $(M_{F_\Sigma})_\lambda$. Then
  there exists a homogenous element $s\in F_\Sigma$ such that
  $sV\subset M$. Let $\nu\in \setZ\Sigma$ be the degree of $s$. So $sV
  \subset M_{q^\nu \lambda}$ hence $\dim sV \leq d$. Since $s$ acts
  injectively on $M_{F_\Sigma}$ we have $\dim V \leq d$. Now the first
  claim follows because $F_\beta^{\pm 1}$ acts injectively on
  $M_{F_\Sigma}$ for all $\beta\in \Sigma$.
  
  We have an injective $U_q$-homomorphism from $M$ to $M_{F_\Sigma}$
  sending $m\in M$ to $1\tensor m\in M_{F_\Sigma}$ that restricts to a
  $(U_q)_0$-homomorphism from $M_\lambda$ to
  $(M_{F_\Sigma})_\lambda$. If $\dim M_\lambda = d$ then this is
  surjective as well. So it is an isomorphism. The last claim follow
  because $\pm \Sigma \subset T_M$ implies $\dim M_\lambda = \dim
  M_{q^\mu\lambda}$ for any $\mu\in \setZ \Sigma$; so $M_\lambda \iso
  (M_{F_\Sigma})_\lambda$ for any $\lambda\in X$. Since $M$ is a
  weight module this implies that $M \iso M_{F_\Sigma}$ as
  $U_q$-modules.
\end{proof}

\begin{lemma}
  \label{lemma:24}
  Let $L$ be a simple infinite dimensional admissible module. Let
  $\beta\in (T_L^s)^+$. Then there exists a $b\in \setC^*$ such that
  $\phi_{F_\beta,b}.L_{F_\beta}$ contains a simple admissible
  $U_q$-submodule $L'$ with $T_{L'}\subset T_L$ and $\beta\not \in
  T_L$.
\end{lemma}
\begin{proof}
  Since $\beta\in T_L^s$ we have $L\iso L_{F_\beta}$ as $U_q$-modules
  by Lemma~\ref{lemma:13}. So we will consider $L$ as a
  $U_{q(F_\beta)}$-module via this isomorphism when taking twist etc.

  Let $E_\beta$ and $F_\beta$ be root vectors corresponding to
  $\beta$.  Let $\lambda\in \wt L$. Consider $F_\beta E_\beta$ as a
  linear operator on $L_\lambda$. Since $\setC$ is algebraically
  closed $F_\beta E_\beta$ must have an eigenvalue $c_\beta$ and an
  eigenvector $v\in L_\lambda$. By (the proof of) Lemma~\ref{lemma:9}
  \begin{equation*}
    F_\beta E_\beta \phi_{F_{\beta} ,b} . v = \phi_{F_{\beta} ,b}.(c_\beta - (q_\beta-q_\beta\inv)^{-2}(b_\beta-b_\beta\inv)(q_\beta b_\beta\inv \lambda(K_\beta) - q_\beta\inv b_\beta \lambda(K_\beta)\inv) v.
  \end{equation*}
  The Laurent polynomial, in $b$, $c_\beta -
  (q-q\inv)^{-2}(b_\beta-b_\beta\inv)(b_\beta \lambda(K_\beta) -
  b_\beta\inv \lambda(K_\beta\inv))$ has a zero point $c\in \setC^*$.

  Thus $\phi_{F_{\beta} ,c}.L$ contains an element $v'$ such that
  $F_\beta E_\beta v'=0$ and since $F_\beta$ acts injectively on
  $\phi_{F_{\beta} ,c}.L$, we have $E_\beta v'=0$. Set $V=\{m\in
  \phi_{F_{\beta} ,c}.L| E_\beta^N m = 0, N>>0\}=(\phi_{F_{\beta}
    ,c}.L)^{[\beta]}$. By Proposition~\ref{prop:2} this is a
  $U_q$-submodule of the $U_q$-module $\phi_{F_{\beta} ,c}.L$. It is
  nonzero since $v'\in V$. By Lemma~\ref{lemma:10} $V$ has a simple
  $U_q$-submodule $L'$.

  We want to show that $T_{L'} \subset T_L$. Assume $\gamma \in
  T_{L'}$. Then $q^{\setN \gamma}\wt L' \subset \wt L'$. But since
  $\wt L' \subset c^{-\beta} \wt L$ we get for some $\nu \in \wt L$,
  $q^{\setN \gamma}c^{-\beta} \nu \subset c^{-\beta} \wt L$ or
  equivalently $q^{\setN \gamma} \nu \subset \wt L$. But this shows
  that $\gamma \not \in F_L$ and since $L$ is a simple $U_q$-module
  this implies that $\gamma \in T_L$. By construction we have
  $\beta\not \in T_{L'}$.
\end{proof}

\section{Coherent families}
\label{sec:coherent-families}
For a $U_q$-module $M\in\mathcal{F}$ define $\Tr^M: X \times (U_q)_0 \to
\setC$ by $\Tr^M(\lambda,u) = \Tr u|_{M_\lambda}$.

\begin{lemma}
  \label{lemma:32}
  Let $M,N\in \mathcal{F}$ be semisimple $U_q$-modules. If $\Tr^M =
  \Tr^N$ then $M\iso N$.
\end{lemma}
\begin{proof}
  Theorem~7.19 in~\cite{Lam} states that this is true for modules over
  a \emph{finite dimensional} algebra. So we will reduce to the case
  of modules over a finite dimensional algebra. Let $L$ be a
  composition factor of $M$ and $\lambda$ a weight of $L$. Then the
  multiplicity of the $U_q$-composition factor $L$ in $M$ is the
  multiplicity of the $(U_q)_0$-composition factor $L_\lambda$ in
  $M_\lambda$ by Theorem~\ref{thm:Lemire}. $M_\lambda$ is a finite
  dimensional $(U_q)_0$-module. Let $I$ be the kernel of the
  homomorphism $(U_q)_0 \to \End_{\setC}(M_\lambda )$ given by the
  action of $(U_q)_0$. Then $(U_q)_0 / I$ is a finite dimensional
  $\setC$ algebra and $M_\lambda$ is a module over $(U_q)_0/
  I$. Furthermore since $\Tr^M(\lambda,u)=0$ for all $u\in I$ the
  trace of an element $u\in (U_q)_0$ is the same as the trace of $u+I
  \in (U_q)_0/I$ on $M_\lambda$ as a $(U_q)_0/I$-module. So if $\Tr^M
  = \Tr^N$ the multiplicity of $L_\lambda$ in $M_\lambda$ and
  $N_\lambda$ are the same and hence the multiplicity of $L$ in $M$ is
  the same as in $N$.
\end{proof}

We will use the Zariski topology on $(\setC^*)^n$: $V$ is a closed set
if it is the zero-points of a Laurent polynomial $p\in \setC[X_1^{\pm
  1},\dots,X_n^{\pm 1}]$.

\begin{prop}
  \label{prop:24}
  Let $L$ be an infinite dimensional admissible simple module of
  degree $d$. Let $\Sigma$ be a set of commuting roots that is a basis
  of $Q$ and $w\in W$ such that $-\Sigma\subset w(T_L)$. Let $F_\Sigma$ be
  a corresponding Ore subset. Let $\lambda\in \Suppess(L)$. The set
  \begin{equation*}
    \{ \mathbf{b}\in (\setC^*)^n |\, \, {^{\bar{w}}}\left( \phi_{F_\Sigma,\mathbf{b}}.\left( \left( {^w}L \right)_{F_\Sigma} \right)_{w(\lambda)} \right) \text{ is a simple $(U_q)_0$-module} \}
  \end{equation*}
  is a Zariski open set of $(\setC^*)^n$.
\end{prop}
\begin{proof}
  The $(U_q)_0$-module $V:={^{\bar{w}}}\left(
    \phi_{F_\Sigma,\mathbf{b}}.\left( \left( {^w}L \right)_{F_\Sigma}
    \right)_{w(\lambda)} \right)$ is simple if and only if the
  bilinear map $B_{\mathbf{b}}(u,v)\in (U_q)_0 \times (U_q)_0 \mapsto
  \Tr\left( uv|_{{^{\bar{w}}}\left( \phi_{F_\Sigma,\mathbf{b}}.\left(
          \left( {^w}L \right)_{F_\Sigma} \right)_{w(\lambda)}
      \right)}\right)$ has maximal rank $d^2$: The map factors through
  $\End_{\setC}(V)\times \End_{\setC}(V)$ given by the representation
  $(U_q)_0\to \End_{\setC}(V)$ on $V$. $B_{\mathbf{b}}$ has maximal
  rank $d^2$ if and only if the representation is surjective onto
  $\End_{\setC}(V)$ which is equivalent to $V$ being simple.

  For any finite dimensional subspace $E\subset (U_q)_0$, the set
  $\Omega_E$ of all $\mathbf{b}$ such that $B_{\mathbf{b}}|_E$ has
  rank $d^2$ is either empty or the non-zero points of the Laurent
  polynomial $\det M$ for some $d^2\times d^2$ minor $M$ of the matrix
  $\left( B_{\mathbf{b}}(e_i,e_j) \right)_{i,j}$ where $\{e_i\}$ is a
  basis of $E$. Therefore $\Omega = \cup_E \Omega_E$ is open.
\end{proof}

For a module $M$ that is a direct sum of modules of finite length we
define $M^{ss}$ to be the unique (up to isomorphism) semisimple module
with the same composition factors as $M$.

\begin{lemma}
  \label{lemma:1}
  Let $L$ be an infinite dimensional simple admissible $U_q$-module of
  degree $d$, $w\in W$ and $\Sigma=\{\beta_1,\dots,\beta_n\}\subset
  \Phi^+$ a set of commuting roots that is a basis of $Q$ such that
  $-\Sigma\subset w(T_L)$. Let $F_{\Sigma}$ be a corresponding Ore
  subset to $\Sigma$. Let $\mathbf{c}\in (\setC^*)^n$ and let $L'$ be
  another infinite dimensional $U_q$-module such that $L'$ is
  contained in
  ${^{\bar{w}}}\left(\phi_{F_{\Sigma},\mathbf{c}}.({^w}L)_{F_{\Sigma}}\right)^{ss}$
  (i.e. $L'$ is a composition factor of
  ${^{\bar{w}}}\left(\phi_{F_{\Sigma},\mathbf{c}}.({^w}L)_{F_\Sigma}\right)$). Assume
  that $\Sigma'=\{\beta_1',\dots,\beta_n'\} \subset \Phi^+$ is another
  set of commuting roots that is a basis of $Q$ and $w' \in W$ is such
  that $-\Sigma' \subset w'(T_{L'})$. Let $F_{\Sigma'}$ be a
  corresponding Ore subset.
  
  Define $a_{i,j}\in \setZ$ by $w(w')\inv (\beta_i') = \sum_{j=1}^n
  a_{i,j} \beta_j$ and define $f:(\setC^*)^n \to (\setC^*)^n$ by
  \begin{equation*}
    f(b_1,\dots,b_n)= \left( \prod_{i=1}^n b_i^{a_{i,1}},\dots,\prod_{i=1}^{n}b_i^{a_{i,n}}\right).
  \end{equation*}
  Then $L'$ is admissible of degree $d$ and
  \begin{equation*}
    {^{\bar{w'}}}\left( \phi_{F_{\Sigma'},\mathbf{b}}.({^{w'}}L')_{F_{\Sigma'}}\right)^{ss} \iso {^{\bar{w}}}\left( \phi_{F_{\Sigma},f(\mathbf{b})\mathbf{c}}.({^w}L)_{F_{\Sigma}}\right)^{ss}
  \end{equation*}
\end{lemma}
\begin{proof}
  We will show that $\Tr^{{^{\bar{w'}}}\left(
      \phi_{F_{\Sigma'},\mathbf{b}}.({^{w'}}L')_{F_{\Sigma'}}\right)^{ss}}
  = \Tr^{{^{\bar{w}}}\left(
      \phi_{F_{\Sigma},f(\mathbf{b})\mathbf{c}}.({^w}L)_{F_{\Sigma}}\right)^{ss}}$.

  Let $\lambda\in \Suppess(L)$. Then $w(\lambda)\in
  \Suppess({^w}L)$. As a $(U_q)_0$-module we have $\left(
    {^{\bar{w}}}\left( \phi_{F_\Sigma,\mathbf{c}}.({^w}L)_{F_\Sigma}
    \right) \right)^{ss}\iso \bigoplus_{\mathbf{i}\in \setZ^n}
  {^{\bar{w}}}\left( \phi_{F_\Sigma,q^{\mathbf{i}}\mathbf{c}}.\left(
      \left( {^w}L \right)_{F_\Sigma} \right)_{w(\lambda)}
  \right)^{ss}$ (Corollary~\ref{cor:6}). Let
  $\lambda'\in\Suppess(L')$.  Then $L'_{\lambda'}$ is a
  $(U_q)_0$-submodule of ${^{\bar{w}}}\left(
    \phi_{F_\Sigma,q^{\mathbf{j}}\mathbf{c}}.\left( \left( {^w}L
      \right)_{F_\Sigma} \right)_{w(\lambda)} \right)^{ss}$ for some
  $\mathbf{j}\in\setZ^n$. We can assume $\mathbf{j}=0$ by replacing
  $\mathbf{c}$ with $q^{\mathbf{j}}\mathbf{c}$ (note that we have then
  $(\mathbf{c}\inv)^\Sigma = w\left( \lambda' \lambda\inv \right)$).
  So $L'_{\lambda'}$ is a $(U_q)_0$-submodule of ${^{\bar{w}}}\left(
    \phi_{F_\Sigma,\mathbf{c}}.\left( \left( {^w}L \right)_{F_\Sigma}
    \right)_{w(\lambda)} \right)^{ss}$. For any other $\mu\in
  \Suppess(L')$ there is a unique $\mathbf{j}_\mu'\in \setZ^n$ such
  that $\mu = (w')\inv \left(\left(
      q^{-\mathbf{j}_\mu'}\right)^{\Sigma'}\right) \lambda'$ and a
  unique $\mathbf{j}_\mu\in \setZ^n$ such that $w\inv \left(\left(
      q^{-\mathbf{j}_\mu}\mathbf{c}\inv \right)^\Sigma\right)
  \lambda=\mu$. For such $\mathbf{j}_\mu$, $L_\mu'$ is a submodule of
  ${^{\bar{w}}}\left(
    \phi_{F_\Sigma,q^{\mathbf{j}_\mu}\mathbf{c}}.\left( \left( {^w}L
      \right)_{F_\Sigma} \right)_{w(\lambda)}
  \right)^{ss}$. 

  $f$ is bijective, $f(q^{\setZ^n})=q^{\setZ^n}$,
  $f(\mathbf{b})^\Sigma = w(w')\inv \left(\mathbf{b}^{\Sigma'}\right)$
  for all $\mathbf{b}\in (\setC^*)^n$ and for any $\mu\in
  \Suppess(L')$, $f(q^{\mathbf{j}'_\mu})=q^{\mathbf{j}_\mu}$. For a
  Laurent polynomial $p$, $p\circ f$ is Laurent polynomial as
  well. Since $q^{\setN^n}$ is Zariski dense in $(\setC^*)^n$
  (Lemma~\ref{lemma:37}) and $f$ is a Laurent polynomial the set
  $D=\{q^{\mathbf{j}_\mu}\mathbf{c}\in (\setC^*)^n|\mu\in
  \Suppess(L')\}$ is Zariski dense. By Proposition~\ref{prop:24} the
  $(U_q)_0$-module ${^{\bar{w}}}\left(
    \phi_{F_\Sigma,\mathbf{b}}.\left( \left( {^w}L \right)_{F_\Sigma}
    \right)_{w(\lambda)} \right)$ is simple for all $\mathbf{b}\in
  \Omega$ for some Zariski open set $\Omega$ of $(\setC^*)^n$. Since
  $D$ is dense and $\Omega$ is open $D\cap \Omega$ is nonempty. So
  there exists a $\mu_0 \in \Suppess(L')$ such that
  ${^{\bar{w}}}\left( \phi_{F_\Sigma,q^{\mathbf{j}_{\mu_0}}
      \mathbf{c}}.\left( \left( {^w}L \right)_{F_\Sigma}
    \right)_{w(\lambda)} \right)$ is simple and contains the nonzero
  simple $(U_q)_0$-module $L'_{\mu_0}$ as a submodule. Thus
  $L'_{\mu_0} \iso {^{\bar{w}}}\left(
    \phi_{F_\Sigma,q^{\mathbf{j}_{\mu_0}} \mathbf{c}}.\left( \left(
        {^w}L \right)_{F_\Sigma} \right)_{w(\lambda)} \right)$. We get
  now from Lemma~\ref{lemma:13} that $L'$ is admissible of degree $d$
  and that for every $\mu \in \Suppess(L')$,
  \begin{align*}
    L'_\mu \iso& {^{\bar{w}}}\left( \phi_{F_\Sigma,q^{\mathbf{j}_\mu}
        \mathbf{c}}.\left( \left( {^w}L \right)_{F_\Sigma}
      \right)_{w(\lambda)} \right)
    \\
    \iso& {^{\bar{w}}}\left( \phi_{F_\Sigma,f(q^{\mathbf{j}'_\mu})
        \mathbf{c}}.\left( \left( {^w}L \right)_{F_\Sigma}
      \right)_{w(\lambda)} \right).
  \end{align*}
  
  By Lemma~\ref{lemma:13}, Corollary~\ref{cor:6} and the definition of
  $\mathbf{j}'_\mu$ we have for any $\mu\in \Suppess(L')$
  \begin{align*}
    {^{\bar{w'}}}\left(
      \phi_{F_{\Sigma'},q^{\mathbf{j}'_\mu}}.\left(({^{w'}}L')_{F_{\Sigma'}}\right)_{w'(\lambda')}
    \right) \iso L'_{\mu}.
  \end{align*}
  Let $u\in (U_q)_0$. We see that for $\mathbf{b} =
  q^{\mathbf{j}'_\mu}$
  \begin{align*}
    \Tr u|_{{^{\bar{w}}}\left( \phi_{F_\Sigma,f(\mathbf{b})
          \mathbf{c}}.\left( \left( {^w}L \right)_{F_\Sigma}
        \right)_{w(\lambda)} \right)} = \Tr u|_{L'_\mu} = \Tr
    u|_{{^{\bar{w'}}}\left(
        \phi_{F_{\Sigma'},\mathbf{b}}.\left(({^{w'}}L')_{F_{\Sigma'}}\right)_{w'(\lambda')}
      \right)}.
  \end{align*}
  Since $\mathbf{b}\mapsto \Tr u|_{{^{\bar{w}}}\left(
      \phi_{F_\Sigma,f(\mathbf{b}) \mathbf{c}}.\left( \left( {^w}L
        \right)_{F_\Sigma} \right)_{w(\lambda)} \right)^{ss}}$ and
  $\mathbf{b}\mapsto \Tr u|_{{^{\bar{w'}}}\left(
      \phi_{F_{\Sigma'},\mathbf{b}}.\left(({^{w'}}L')_{F_{\Sigma'}}\right)_{w'(\lambda')}
    \right)^{ss}}$ are both Laurent polynomials and equal on the
  Zariski dense subset $\{q^{\mathbf{j}'_\mu}|\mu\in \Suppess(L')\}$
  they are equal for all $\mathbf{b}\in (\setC^*)^n$. Thus by
  Lemma~\ref{lemma:32}
  \begin{align*}
    {^{\bar{w}}}\left( \phi_{F_\Sigma,f(\mathbf{b}) \mathbf{c}}.\left(
        \left( {^w}L \right)_{F_\Sigma} \right)_{w(\lambda)}
    \right)^{ss} \iso {^{\bar{w'}}}\left(
      \phi_{F_{\Sigma'},\mathbf{b}}.\left(({^{w'}}L')_{F_{\Sigma'}}\right)_{w'(\lambda')}
    \right)^{ss}
  \end{align*}
  as $(U_q)_0$-modules. Since (by Corollary~\ref{cor:6})
  \begin{align*}
    {^{\bar{w'}}}\left(
      \phi_{F_{\Sigma'},\mathbf{b}}.({^{w'}}L')_{F_{\Sigma'}}\right)^{ss}
    \iso \bigoplus_{\mathbf{i}\in \setZ^n} {^{\bar{w'}}}\left(
      \phi_{F_{\Sigma'},q^{\mathbf{i}}\mathbf{b}}.\left(({^{w'}}L')_{F_{\Sigma'}}\right)_{w'(\lambda')}\right)^{ss}
  \end{align*}
  and
  \begin{align*}
    {^{\bar{w}}}\left(
      \phi_{F_{\Sigma},f(\mathbf{b})\mathbf{c}}.({^w}L)_{F_{\Sigma}}\right)^{ss}
    \iso \bigoplus_{\mathbf{i}\in\setZ^n}{^{\bar{w}}}\left(
      \phi_{F_{\Sigma},q^{\mathbf{i}}
        f(\mathbf{b})\mathbf{c}}.\left(({^w}L)_{F_{\Sigma}}\right)_{w(\lambda)}\right)^{ss}
  \end{align*}
  we get
  \begin{align*}
    {^{\bar{w'}}}\left(
      \phi_{F_{\Sigma'},\mathbf{b}}.({^{w'}}L')_{F_{\Sigma'}}\right)^{ss}
    \iso& \bigoplus_{\mathbf{i}\in \setZ^n} {^{\bar{w'}}}\left(
      \phi_{F_{\Sigma'},q^{\mathbf{i}}\mathbf{b}}.\left(({^{w'}}L')_{F_{\Sigma'}}\right)_{w'(\lambda')}\right)^{ss}
    \\
    \iso& \bigoplus_{\mathbf{i}\in\setZ^n}{^{\bar{w}}}\left(
      \phi_{F_{\Sigma}, f(q^{\mathbf{i}}
        \mathbf{b})\mathbf{c}}.\left(({^w}L)_{F_{\Sigma}}\right)_{w(\lambda)}\right)^{ss}
    \\
    \iso& \bigoplus_{\mathbf{i}\in\setZ^n}{^{\bar{w}}}\left(
      \phi_{F_{\Sigma}, q^{\mathbf{i}} f(
        \mathbf{b})\mathbf{c}}.\left(({^w}L)_{F_{\Sigma}}\right)_{w(\lambda)}\right)^{ss}
    \\
    \iso& {^{\bar{w}}}\left(
      \phi_{F_{\Sigma},f(\mathbf{b})\mathbf{c}}.({^w}L)_{F_{\Sigma}}\right)^{ss}
  \end{align*}
  as $(U_q)_0$-modules. By Theorem~\ref{thm:Lemire} this implies they
  are isomorphic as $U_q$-modules as well.
\end{proof}

Corollary~\ref{cor:6} tells us that twisting with an element of the
form $q^{\mathbf{i}}$ gives us a module isomorphic to the original
module. Thus it makes sense to write $\phi_{F_\Sigma,t}.M$ for a $t
\in (\setC^*)^n / q^{\setZ^n}$ and a $U_{q(F_\Sigma)}$-module
$M$. Just choose a representative for $t$. Any representative gives
the same $U_{q(F_\Sigma)}$-module up to isomorphism.

Let $L$ be an admissible simple module. Assume for a $w\in W$ that
$\Sigma\subset -w(T_L)$ is a set of commuting roots that is a basis of
$Q$ (it is always possible to find such $w$ and $\Sigma$ by
Lemma~\ref{lemma:6} and Lemma~\ref{lemma:26}) and let $F_\Sigma$ be a
corresponding Ore subset. Let $\nu \in X$. The $U_q$-module
\begin{equation*}
  {^{\bar{w}}}\left( \bigoplus_{\mathbf{b}\in (\setC^*)^n: \, \mathbf{b}^{\Sigma}=\nu} \phi_{F_\Sigma,\mathbf{b}}. \left( {^w}L \right)_{F_\Sigma} \right)
\end{equation*}
has finite length by Lemma~\ref{lemma:16}, Lemma~\ref{lemma:13} and Lemma~\ref{lemma:10}.

We define
\begin{equation*}
  \mathcal{EXT}(L) = \left( \bigoplus_{t\in (\setC^*)^n/q^{\setZ^n}} {^{\bar{w}}} \left( \phi_{F_\Sigma,t}.\left( {^w}L \right)_{F_\Sigma} \right) \right)^{ss}.
\end{equation*}
The definition is independent (up to isomorphism) of the chosen $w$,
$\Sigma$ and $F_\Sigma$ as suggested by the notation:

\begin{lemma}
  \label{lemma:20}
  Let $L$ be a simple admissible module. Let $w,w'\in W$ and assume
  $\Sigma\subset -w(T_L),\Sigma'\subset -w'(T_{L'})$ are sets of commuting
  roots that are both a basis of $Q$. Let $F_\Sigma,F'_{\Sigma'}$
  be corresponding Ore subsets. Then
  \begin{equation*}
    \left( \bigoplus_{     t\in (\setC^*)^n/q^{\setZ^n}     } {^{\bar{w}}} \left( \phi_{F_\Sigma,t}.\left( {^w}L \right)_{F_\Sigma} \right) \right)^{ss} \iso \left( \bigoplus_{    t\in (\setC^*)^n/q^{\setZ^n}    } {^{\bar{w'}}} \left( \phi_{F'_{\Sigma'},t}.\left( {^{w'}}L \right)_{F'_{\Sigma'}} \right) \right)^{ss} 
  \end{equation*}
  as $U_q$-modules.
\end{lemma}
\begin{proof}
  Obviously $L$ is a submodule of $\left( {^{\bar{w}}} \left(
      \phi_{F_\Sigma,\mathbf{1}}.\left( {^w}L \right)_{F_\Sigma}
    \right) \right)^{ss}$ where $\mathbf{1}=(1,\dots,1)$. By
  Lemma~\ref{lemma:1} this implies that for $\mathbf{b}\in
  (\setC^*)^n$
  \begin{align*}
    \left( {^{\bar{w'}}} \left( \phi_{F_{\Sigma'},\mathbf{b}}.\left(
          {^{w'}}L \right)_{F_{\Sigma'}} \right) \right)^{ss} \iso
    \left( {^{\bar{w}}} \left( \phi_{F_\Sigma,f(\mathbf{b})}.\left(
          {^w}L \right)_{F_\Sigma} \right) \right)^{ss}
  \end{align*}
  for some $f$ with the property that $f(q^{\setZ^n})=q^{\setZ^n}$. So
  it makes sense to write $f(t)$ for $t\in (\setC^*)^n/q^{\setZ^n}$.
  Thus
  \begin{align*}
    \left( \bigoplus_{ t\in (\setC^*)^n/q^{\setZ^n} } {^{\bar{w'}}}
      \left( \phi_{F'_{\Sigma'},t}.\left( {^{w'}}L
        \right)_{F'_{\Sigma'}} \right) \right)^{ss} \iso& \left(
      \bigoplus_{ t\in (\setC^*)^n/q^{\setZ^n} } {^{\bar{w}}} \left(
        \phi_{F_\Sigma,f(t)}.\left( {^w}L \right)_{F_\Sigma} \right)
    \right)^{ss}
    \\
    \iso& \left( \bigoplus_{ t\in (\setC^*)^n/q^{\setZ^n} }
      {^{\bar{w}}} \left( \phi_{F_\Sigma,t}.\left( {^w}L
        \right)_{F_\Sigma} \right) \right)^{ss}
  \end{align*}
  since $f$ is bijective.
\end{proof}

\begin{prop}
  \label{prop:19}
  Let $L$ be a simple infinite dimensional admissible module. For
  $x\in W$:
  \begin{equation*}
    \mathcal{EXT}({^{x}}L) \iso {^{x}}\left(\mathcal{EXT}(L)\right)
  \end{equation*}
  and
  \begin{equation*}
    \mathcal{EXT}({^{\bar{x}}}L) \iso {^{\bar{x}}}\left(\mathcal{EXT}(L)\right).
  \end{equation*}
\end{prop}
\begin{proof}
  Let $w\in W$ be such that $w(F_L \backslash F_L^s) \subset \Phi^+$
  (exists by Lemma~\ref{lemma:6}). Let $\Sigma$ be a set of commuting
  roots that is a basis of $Q$ such that $-\Sigma \subset w(T_L)$
  (exists by Lemma~\ref{lemma:26}) and let $F_\Sigma$ be a
  corresponding Ore subset. First we will define $\mathcal{EXT}'(L) =
  \left( \bigoplus_{t\in (\setC^*)^n/ q^{\setZ^n}} {^{w\inv }}\left(
      \phi_{F_{\Sigma} ,t}.({^{\bar{w\inv }}} L)_{F_\Sigma}\right)
  \right)^{ss}$ and show that $\mathcal{EXT}'(L) \iso
  \mathcal{EXT}(L)$ as $U_q$-modules: Going through the proof of
  Lemma~\ref{lemma:1} and Lemma~\ref{lemma:20} and and replacing $T_{w\inv}$
  and $T_{w\inv}\inv$ with $T_{w}\inv$ and $T_{w}$ respectively we
  get that $\Tr^{\mathcal{EXT}'(L)}=\Tr^{\mathcal{EXT}(L)}$ so they
  are isomorphic by Lemma~\ref{lemma:32}.

  We will show for any $\alpha\in \Pi$ that
  \begin{equation*}
    \mathcal{EXT}({^{s_\alpha}}L) \iso {^{s_\alpha}}\left(\mathcal{EXT}(L)\right)
  \end{equation*}
  which implies the claim by induction over the length $l(x)$ of $x$
  (where $l(x)$ is the smallest number of simple reflections need to
  write $x$, i.e. there is a reduced expression $x=s_{i_1}\cdots
  s_{i_{l(x)}}$).

  So let $\alpha\in \Pi$ and let $w$ and $\Sigma$ be defined as
  above. Let $w'=ws_\alpha$. Note that $w'
  (F_{{^{s_\alpha}}L}\backslash F_{{^{s_\alpha}}L}^s) \subset \Phi^+$
  and $-\Sigma \subset T_{{^{s_\alpha}}L}$. We split into two cases:
  If $l( w' )<l(w)$ then
  \begin{align*}
    {^{s_\alpha}}\left( \mathcal{EXT}(L) \right)=& {^{s_\alpha}}\left(
      \left( \bigoplus_{t\in (\setC^*)^n/q^{\setZ^n}} {^{\bar{w' s_\alpha}}}
        \left( \phi_{F_\Sigma ,t}. ({^{w' s_\alpha}} L
          )_{F_\Sigma} \right) \right)^{ss} \right)
    \\
    \iso& {^{s_\alpha}} \left( \left( \bigoplus_{t\in (\setC^*)^n/q^{\setZ^n}}
        {^{ \bar{s_\alpha} }}\left({^{ \bar{w'} }} \left(
            \phi_{F_\Sigma,t}.(^{w'}(^{s_\alpha}L))_{F_\Sigma}
          \right) \right) \right)^{ss} \right)
    \\
    \iso& \left( \bigoplus_{t\in (\setC^*)^n/q^{\setZ^n}} {^{ \bar{w'} }} \left(
        \phi_{F_\Sigma,t}.(^{w'}(^{s_\alpha}L))_{F_\Sigma}
      \right) \right)^{ss}
    \\
    =& \mathcal{EXT}({^{s_\alpha}}L).
  \end{align*}

  If $l(w' ) > l(w)$ we get
  \begin{align*}
    {^{s_\alpha}}\left( \mathcal{EXT}(L) \right) \iso&
    {^{s_\alpha}}\left( \mathcal{EXT}'(L) \right)
    \\
    =& {^{s_\alpha}}\left( \left( \bigoplus_{t\in (\setC^*)^n/q^{\setZ^n}}
        {^{w\inv }} \left( \phi_{F_\Sigma
            ,t}. ({^{\bar{w\inv}}} L )_{F_\Sigma} \right)
      \right)^{ss} \right)
    \\
    \iso& \left( \bigoplus_{t\in (\setC^*)^n/q^{\setZ^n}} {^{ (w' )\inv }} \left(
        \phi_{F_\Sigma ,t}. ({^{\bar{ (w') \inv }}}
        (^{s_\alpha}L) )_{F_\Sigma} \right) \right)^{ss}
    \\
    =& \mathcal{EXT}({^{s_\alpha}}L).
  \end{align*}
  The second claim is shown similarly.
\end{proof}

\begin{prop}
  \label{prop:15}
  Let $L$ be an infinite dimensional admissible simple module of
  degree $d$. If $L'$ is an infinite dimensional simple submodule of
  $\mathcal{EXT}(L)$ then $L'$ is admissible of degree $d$ and
  $\mathcal{EXT}(L)\iso \mathcal{EXT}(L')$.
\end{prop}
\begin{proof}
  Let $w\in W$ and let $\Sigma$ be a set of commuting roots that is a
  basis of $Q$ such that $\Sigma\subset -w(T_L)$ (possible by
  Lemma~\ref{lemma:6} and Lemma~\ref{lemma:26}). Then by definition
  \begin{align*}
    \mathcal{EXT}(L)=\left( \bigoplus_{t\in (\setC^*)^n/q^{\setZ^n}}
      {^{\bar{w}}} \left( \phi_{F_\Sigma,t}.\left( {^w}L
        \right)_{F_\Sigma} \right) \right)^{ss}.
  \end{align*}
  $L'$ being a submodule of $\mathcal{EXT}(L)$ implies that $L'$ must
  be a submodule of
  \begin{align*}
    \left( {^{\bar{w}}} \left( \phi_{F_\Sigma,\mathbf{c}}.\left( {^w}L
        \right)_{F_\Sigma} \right) \right)^{ss}
  \end{align*}
  for some $\mathbf{c}\in (\setC^*)^n$. Let $w'\in W$ and let
  $\Sigma'$ be a set of commuting roots that is a basis of $Q$ such
  that $\Sigma' \subset -w'(T_{L'})$. By Lemma~\ref{lemma:1} $L'$ is
  admissible of degree $d$ and there exists a bijective map
  $f:(\setC^*)^n \to (\setC^*)^n$ such that
  $f(q^{\setZ^n})=q^{\setZ^n}$ and
  \begin{align*}
    \left( {^{\bar{w'}}} \left( \phi_{F_{\Sigma'},\mathbf{b}}.\left(
          {^{w'}}L' \right)_{F_{\Sigma'}} \right) \right)^{ss} \iso
    \left( {^{\bar{w}}} \left(
        \phi_{F_\Sigma,f(\mathbf{b})\mathbf{c}}.\left( {^w}L
        \right)_{F_\Sigma} \right) \right)^{ss}.
  \end{align*}
  Since $f(q^{\setZ^n})=q^{\setZ^n}$ it makes sense to write $f(t)$
  for $t\in (\setC^*)^n/q^{\setZ^n}$. So writing
  $t_\mathbf{c}=q^{\setZ^n} \mathbf{c}\in (\setC^*)^n/q^{\setZ^n}$ we get
  \begin{align*}
    \mathcal{EXT}(L')=&\left( \bigoplus_{t\in (\setC^*)^n/q^{\setZ^n}}
      {^{\bar{w'}}} \left( \phi_{F_\Sigma,t}.\left( {^{w'}}(L')
        \right)_{F_{\Sigma'}} \right) \right)^{ss}
    \\
    \iso& \left( \bigoplus_{t\in (\setC^*)^n/q^{\setZ^n}} {^{\bar{w}}}
      \left( \phi_{F_\Sigma,f(t)t_{\mathbf{c}}}.\left( {^w}L
        \right)_{F_\Sigma} \right) \right)^{ss}
    \\
    \iso& \left( \bigoplus_{t\in (\setC^*)^n/q^{\setZ^n}} {^{\bar{w}}}
      \left( \phi_{F_\Sigma,t}.\left( {^w}L \right)_{F_\Sigma} \right)
    \right)^{ss}
    \\
    =& \mathcal{EXT}(L)
  \end{align*}
  since the assignment $t\mapsto f(t)t_{\mathbf{c}}$ is bijective.
\end{proof}

\begin{lemma}
  \label{lemma:15}
  Let $f\in \setC[X_1^{\pm 1},\dots,X_n^{\pm 1}]$ be a nonzero Laurent
  polynomial. There exists $b_1,\dots,b_n\in \setC^*$ such that for
  all $i_1,\dots,i_n\in \setZ$
  \begin{equation*}
    f(q^{i_1}b_1,\dots,q^{i_n}b_n) \neq 0.
  \end{equation*}
\end{lemma}
\begin{proof}
  Assume $f = X_1^{-N_1}\cdots X_n^{-N_n}g$ with $g\in
  \setC[X_1,\dots,X_n]$. $g$ has coefficients in some finitely
  generated (over $\setQ$) subfield $k$ of $\setC$. Let $b_1,\dots
  b_n$ be generators of $n$ disjoint extensions of $k$ of degree
  $>\deg g$. The monomials $b_1^{m_1}\cdots b_n^{m_n}$, $0\leq m_i
  \leq \deg g$ are all linearly independent over $k$. Since $q^i \neq
  0$ for $i\in\setZ$ the same is true for the monomials
  $(q^{i_1}b_1)^{m_1}\cdots (q^{i_n}b_n)^{m_n}$. So
  $g(q^{i_1}b_1,\dots,q^{i_n}b_n)\neq 0$, hence
  $f(q^{i_1}b_1,\dots,q^{i_n}b_n)\neq 0$.
\end{proof}

\begin{thm}
  \label{thm:existence_of_torsion_free_modules}
  Let $L$ be an infinite dimensional admissible simple modules of
  degree $d$. Then $\mathcal{EXT}(L)$ contains at least one simple
  torsion free module.
\end{thm}
\begin{proof}
  Let $\lambda \in w(\wt L)$. Then as a $(U_q)_0$-module
  \begin{equation*}
    \mathcal{EXT}(L) = \left( {^{\bar{w}}}\left( \bigoplus_{\mathbf{b} \in (\setC^*)^n} \phi_{F_\Sigma,\mathbf{b}}.\left(\left({^w}L\right)_{F_\Sigma}\right)_{\lambda} \right) \right)^{ss}
  \end{equation*}
  for some $w\in W$ and some Ore subset $F_{\Sigma}$ corresponding to
  a set of commuting roots $\Sigma$ that is a basis of $Q$. Let $u\in
  (U_q)_0$. Then the map $\mathbf{b} \mapsto \det
  u|_{{^{\bar{w}}}\left(\phi_{F_\Sigma,\mathbf{b}}.\left(\left({^w}L\right)_{F_\Sigma}\right)_\lambda\right)}=
  \det
  \phi_{F_\Sigma,\mathbf{b}}(T_w\inv(u))|_{\left(\left({^w}L\right)_{F_\Sigma}\right)_{\lambda}}$
  is Laurent polynomial. Let $p(\mathbf{b}) = \prod_{\beta\in \Sigma}
  \det
  E_{\beta}F_{\beta}|_{{^{\bar{w}}}\left(\phi_{F_\Sigma,\mathbf{b}}.\left(\left({^w}L\right)_{F_\Sigma}\right)_\lambda\right)}$. $p$
  is a Laurent polynomial by the above.  By Lemma~\ref{lemma:15} there
  exists a $\mathbf{c}\in (\setC^*)^n$ such that $p(\mathbf{b})\neq 0$
  for all $\mathbf{b} \in q^{\setZ^n}\mathbf{c}$ which implies that
  $E_{\beta}F_{\beta}$ acts injectively on the module
  $L':={^{\bar{w}}}\left(
    \phi_{F_\beta,\mathbf{c}}.({^w}L)_{F_\Sigma}\right)$ for all
  $\beta\in \Sigma$. Since $F_{\beta}$ acts injectively on the module
  by construction this implies that $E_\beta$ acts injectively as
  well. So we have $\pm \Sigma \subset T_{L'}$. Any simple submodule
  $V$ of $L'$ is admissible of degree $d$ by Lemma~\ref{lemma:1}
  and since $F_\beta$ and $E_\beta$ act injectively we get $\dim
  V_\lambda = d = \dim L'_\lambda$ for any $\lambda\in \wt L'$ thus $V
  = L'$. So $L'$ is a simple module. Using Proposition~\ref{prop:3} it
  is easy to see that $L'$ is torsion free since $\pm \Sigma \subset
  T_{L'}$ and $\Sigma$ is a basis of $Q$.
\end{proof}

\begin{prop}
  \label{prop:23}
  Let $L$ be an infinite dimensional admissible simple module. Let
  $\beta\in \Phi^+$. If $-\beta \in T_L$ then $\mathcal{EXT}(L)$
  contains $\left( \bigoplus_{t \in \setC^*/q^{\setZ}}
    \phi_{F_\beta,t}.L_{F_\beta}\right)^{ss}$ as a $U_q$-submodule.
\end{prop}
\begin{proof}
  Let $w\in W$ and $\Sigma=\{\beta_1,\dots,\beta_n\}$ be such that
  $\Sigma$ is a set of commuting roots that is a basis of $Q$ and
  $-\Sigma \subset w(T_L)$ and $F_\Sigma$ a corresponding Ore subset
  (always possible by Lemma~\ref{lemma:6} and Lemma~\ref{lemma:26}).

  We have $w(\beta) = \sum_{i=1}^n a_i \beta_i$ for some $a_i\in
  \setZ$. 
  Set $x =
  F_{\beta_1}^{a_1}\cdots F_{\beta_n}^{a_n}\in U_{q(F_\Sigma)}$. Let
  $U_{q(x)}$ be the $U_q$-subalgebra generated by $x$ in
  $U_{q(F_\Sigma)}$. $x$ is playing the role of $F_\beta$ and that is
  why the notation resembles the notation for Ohre localization. The
  Ohre localization of $U_q$ in $x$ does not neccesarily make sense
  though because $x$ is not neccesarily an element of $U_q$.

  Let $V$ be the $U_{q(x)}$-submodule of $({^w}L)_{F_\Sigma}$
  generated by $1\tensor {^w}L$. For any $t\in \setC^*/q^{\setZ}$
  \begin{equation*}
    {^{\bar{w}}}\left(\phi_{F_\Sigma,(t^{a_1},\dots,t^{a_n})}.V\right):=\left\{ \phi_{F_\Sigma,(t^{a_1},\dots,t^{a_n})}.v \in {^{\bar{w}}}\left(\phi_{F_\Sigma,(t^{a_1},\dots,t^{a_n})}.({^w}L)_{F_{\Sigma}}\right) | v \in V\right\}
  \end{equation*}
  is a $U_{q(x)}$-submodule of
  ${^{\bar{w}}}\left(\phi_{F_\Sigma,(t^{a_1},\dots,t^{a_n})}.({^w}L)_{F_{\Sigma}}\right)$:
  To show this we show that for $u\in U_{q(x)}$ and $c\in \setC^*$,
  $\phi_{F_\Sigma,(c^{a_1},\dots,c^{a_n})}(u)\in U_{q(x)}$. We know
  that $\phi_{F_\Sigma,(c^{a_1},\dots,c^{a_n})}(u) \in
  U_{q(F_\Sigma)}[c^{\pm 1}]$ and we also see by construction that for
  $c=q^i$, $i\in\setZ$, we have
  $\phi_{F_\Sigma,(c^{a_1},\dots,c^{a_n})}(u)=x^{-i}ux^i \in
  U_{q(x)}$. Choose a vector space basis of $U_{q(x)}$, $\{u_i\}_{i\in
    I}$ and extend to a basis $\{u_i,u_j'\}_{i\in I, j\in J}$ of
  $U_{q(F_\Sigma)}$ where $I$ and $J$ are some index sets. Then for
  $u\in U_{q(x)}$ we have $\phi_{F_\Sigma,(c^{a_1},\dots,c^{a_n})}(u)
  = \sum_{i\in I'} u_i p_i(c) + \sum_{j\in J'} u_j'p_j'(c)$ for some
  finite $I'\subset I$ and $J'\subset J$ and some $p_i,p_j'\in
  \setC[X^{\pm 1}]$. We see that for $j\in J'$, $p_j'(q^i)=0$ for all
  $i\in\setZ$ so $p_j'=0$. Hence
  $\phi_{F_\Sigma,(c^{a_1},\dots,c^{a_n})}(u) = \sum_{i\in I'} u_i
  p_i(c)\in U_{q(x)}$.  This shows that ${^{\bar{w}}}\left(
    \phi_{F_\Sigma,(t^{a_1},\dots,t^{a_n})}.V \right)$ is a submodule
  of
  ${^{\bar{w}}}\left(\phi_{F_\Sigma,(t^{a_1},\dots,t^{a_n})}.({^w}L)_{F_{\Sigma}}\right)$. Set
  \begin{equation*}
    \mathcal{V} =\left(
      \bigoplus_{t\in \setC^*/q^{\setZ}} {^{\bar{w}}}\left(
        \phi_{F_\Sigma,(t^{a_1},\dots,t^{a_n})}.V\right) \right)^{ss}.
  \end{equation*}
  Clearly $\mathcal{V}$ is a $U_q$-submodule of $\mathcal{EXT}(L)$. We
  claim that $\mathcal{V}\iso \left( \bigoplus_{t \in
      \setC^*/q^{\setZ}} \phi_{F_\beta,t}.L_{F_\beta}\right)^{ss}$ as
  $U_q$-modules. We will show this using Lemma~\ref{lemma:32}.

  Note that for $\lambda\in \wt V$ and $i\in\setZ$ we have
  \begin{equation*}
    {^{\bar{w}}}\left(\phi_{F_\Sigma,((q^i)^{a_1},\dots,(q^i)^{a_n})}.V_\lambda\right) \iso {^{\bar{w}}}\left( V_{q^{-i\sum_{k=1}^n a_k \beta_k}\lambda}\right)
  \end{equation*}
  as a $(U_q)_0$-module by Corollary~\ref{cor:6}.
 
  We have $\wt \mathcal{V} = (\setC^*)^\beta \wt L = \wt \left(
    \bigoplus_{t \in \setC^*/q^{\setZ}}
    \phi_{F_\beta,t}.L_{F_\beta}\right)^{ss}$. Let $\lambda \in \wt L$
  be such that $\dim L_\lambda = \max_{i\in \setZ}\{ \dim
  L_{q^{i\beta}\lambda}\}$ then $V_{w(\lambda)} \iso
  ({^w}L)_{w(\lambda)}\iso {^w}(L_\lambda)$ as a $(U_q)_0$-module by
  Lemma~\ref{lemma:13} and we have for $\nu \in (\setC^*)^\beta
  \lambda$:
  \begin{equation*}
    \mathcal{V}_\nu = \left( \bigoplus_{c\in \setC^*: c^{w(\beta)}=w(\nu\inv \lambda)} {^{\bar{w}}}\left( \phi_{F_\Sigma,(c^{a_1},\dots,c^{a_n})}.V_{w(\lambda)}\right)\right)^{ss}
  \end{equation*}
  so for $u\in (U_q)_0$:
  \begin{align*}
    \Tr u|_{\mathcal{V}_\nu} =& \sum_{c\in \setC^*: c^{\beta}=\nu\inv
      \lambda} \Tr
    \left(\phi_{F_{\Sigma},(c^{a_1},\dots,c^{a_n})}(T_w\inv(u))\right)|_{V_{w(\lambda)}}
  \end{align*}
  (note that $c^{w(\beta)}=w(\nu\inv \lambda)$ if and only if $c^\beta
  = \nu\inv \lambda$ since $c^{w(\beta)}=w(c^\beta)$).

  Set $p(c) = \Tr
  \left(\phi_{F_{\Sigma},(c^{a_1},\dots,c^{a_n})}(T_w\inv(u))\right)|_{V_{w(\lambda)}}$.
  $p$ is Laurent polynomial in $c$ and $p(q^{i}) = \Tr
  u|_{L_{q^{-i\beta}\lambda}}$ for $i\in \setN$.
  
  On the other hand we can show similarly that
  \begin{align*}
    \Tr &u|_{\left(\left( \bigoplus_{t \in \setC^*/q^{\setZ}}
          \phi_{F_\beta,t}.L_{F_\beta}\right)^{ss}\right)_\nu}
    \\
    =& \sum_{c\in \setC^*: c^\beta=\nu\inv \lambda} \Tr
    \left(\phi_{F_{\beta},c}(u)\right)|_{(L_{F_\beta})_{\lambda}}.
  \end{align*}
  Similarly $\Tr
  \left(\phi_{F_{\beta},c}(u)\right)|_{(L_{F_\beta})_{\lambda}}$ is
  Laurent polynomial in $c$ and equal to $\Tr
  u|_{L_{q^{-i\beta}\lambda}}$ for $c=q^{i}$, $i\in\setN$. So $\Tr
  \left(\phi_{F_{\beta},c}(u)\right)|_{(L_{F_\beta})_{\lambda}}=
  p(c)$. We conclude that $\Tr^{\mathcal{V}} = \Tr^{(\left(
      \bigoplus_{t \in \setC^*/q^{\setZ}}
      \phi_{F_\beta,t}.L_{F_\beta}\right)^{ss}}$ so $\mathcal{V}\iso
  \left( \bigoplus_{t \in \setC^*/q^{\setZ}}
    \phi_{F_\beta,t}.L_{F_\beta}\right)^{ss}$ as $U_q$-modules by
  Lemma~\ref{lemma:32}. 
\end{proof}

For any $\lambda\in X$ there is a unique simple highest weight module
which we call $L(\lambda)$. It is the unique simple quotient of the
Verma module $M(\lambda) := U_q \tensor_{U_q^{\geq 0}}
\setC_{\lambda}$ where $\setC_{\lambda}$ is the $1$-dimensional
$U_q^{\geq 0}$-module with $U_q^{+}$ acting trivially and $U_q^0$
acting like $\lambda$. Let $\rho = \frac{1}{2}\sum_{\beta\in \Phi^+}
\beta$. In the following we use the dot action on $X$. For $w\in W$, $w.\lambda := q^{- \rho}
w(q^\rho\lambda)$.

\begin{prop}
  \label{prop:9}
  Let $\lambda\in X$ be such that $L(\lambda)$ is admissible. Let
  $\alpha\in \Pi$. Assume $\lambda(K_{\alpha})\not \in \pm
  q^{\setN}$. Let $a=\frac{2}{(\alpha|\alpha)}$. If $a=\frac{1}{2}$
  choose a squareroot $\lambda(K_\alpha)^{\frac{1}{2}}$ of
  $\lambda(K_\alpha)$. Then
  \begin{itemize}
  \item $-\alpha\in T_{L(\lambda)}$.
  \item $L(s_\alpha.\lambda)$ is admissible.
  \item ${^{s_\alpha}}L(s_\alpha.\lambda)$ is a subquotient of the
    $U_q$-module $L(\lambda)_{F_\alpha}$.
  \item $L(s_\alpha.\lambda)$ and ${^{s_\alpha}}L(\lambda)$ are
    subquotients of the $U_q$-module.
    $\phi_{F_\alpha,\lambda(K_{\alpha})^a}.L(\lambda)_{F_\alpha}$.
  \end{itemize}
\end{prop}
\begin{proof}
  $\lambda(K_\alpha)\not \in \pm q_\alpha^{\setN}$ implies that
  $-\alpha\subset T_{L(\lambda)}$ since for $i\in \setN$:
  \begin{equation*}
    E_\alpha^{(i)} F_\alpha^{(i)}v_\lambda = \prod_{j=1}^i \frac{q_\alpha^{j-1}\lambda(K_\alpha)-q_\alpha^{1-j}\lambda(K_\alpha)\inv}{q_\alpha^j - q_\alpha^{-j}} v_\lambda.
  \end{equation*}
  This is only zero for an $i\in\setN$ when $\lambda(K_\alpha)\in \pm
  q_\alpha^{\setN}$.

  Let $v_\lambda\in L(\lambda)$ be a highest weight vector. Denote the
  vector $\phi_{F_\alpha,\lambda(K_\alpha)^a}.F_\alpha v_\lambda \in
  \phi_{F_\alpha,\lambda(K_\alpha)}.L(\lambda)_{F_\Sigma}$ as
  $v_{s_\alpha.\lambda}$. This is a highest weight vector of weight
  $s_\alpha.\lambda$: For $\mu\in Q$:
  \begin{align*}
    K_\mu v_{s_\alpha.\lambda} =& K_\mu \phi_{F_\alpha,
      \lambda(K_\alpha)^a}.F_\alpha v_\lambda
    \\
    =& \phi_{F_\alpha,q \lambda(K_\alpha)^a}.\left( \left(q
        \lambda(K_\alpha)^a \right)^{-\left(\mu|\alpha \right)}
      \lambda(K_\mu) F_\alpha v_\lambda \right)
    \\
    =& q^{-\left(\mu|\alpha \right)} \lambda\left(
      K_\alpha^{-\left<\mu,\alpha^\vee\right>} K_\mu \right)
    \phi_{F_\alpha,q_\alpha \lambda(K_\alpha)}.F_\alpha v_\lambda
    \\
    =& q^{-(\mu|\alpha)}(s_\alpha \lambda)(K_\mu) v_{s_\alpha.\lambda}
    \\
    =& s_\alpha.\lambda(K_\mu) v_{s_\alpha.\lambda}.
  \end{align*}
  For $\alpha'\in \Pi\backslash\{\alpha\}$
  \begin{equation*}
    E_{\alpha'} \phi_{F_\alpha, \lambda(K_\alpha)^a}.v_\lambda = \phi_{F_\alpha, \lambda(K_\alpha)^a}. E_{\alpha'} v_\lambda
  \end{equation*}
  and for $\alpha'=\alpha$ we have by the formula in the proof of
  Lemma~\ref{lemma:9}
  \begin{align*}
    E_\alpha \phi_{F_\alpha, \lambda(K_\alpha)^a}.&F_\alpha v_\lambda
    \\
    =& \phi_{F_\alpha, \lambda(K_\alpha)^a}.F_\alpha \phi_{F_\alpha,q \lambda(K_\alpha)^a}(E_\alpha) v_\lambda
    \\
    =& \phi_{F_\alpha, \lambda(K_\alpha)^a}. F_\alpha \left( E_\alpha +
      F_\alpha\inv \frac{q_\alpha (q_\alpha\lambda(K_\alpha))\inv
        K_\alpha - q_\alpha\inv
        q_\alpha\lambda(K_\alpha)K_\alpha\inv}{(q_\alpha-q_\alpha\inv)^2}
    \right) v_\lambda
    \\
    =& 0.
  \end{align*}
  So $v_{s_\alpha.\lambda}$ is a highest weight vector of weight
  $s_\alpha.\lambda$ hence $L(s_\alpha.\lambda)$ is a subquotient of
  $\phi_{F_\alpha,\lambda(K_\alpha)^a}.L(\lambda)_{F_\alpha}$. Since
  $L(s_\alpha.\lambda)$ is a subquotient of
  $\phi_{F_\alpha,\lambda(K_\alpha)^a}.L(\lambda)_{F_\alpha}$ it is
  admissible by Lemma~\ref{lemma:13}.

  Consider
  ${^{\bar{s_\alpha}}}\left(\phi_{F_\alpha,\lambda(K_\alpha)^a}.L(\lambda)_{F_\alpha}/(U_qv_{s_\alpha.\lambda})\right)$
  and the vector
  \begin{equation*}
    v' = F_{\alpha}\inv v_{s_\alpha.\lambda} + U_q v_{s_\alpha.\lambda} \in {^{\bar{s_\alpha}}}\left( \phi_{F_\alpha,\lambda(K_\alpha)^a}.L(\lambda)_{F_\alpha} /(U_q v_{s_\alpha.\lambda})\right).
  \end{equation*}
  Then $E_\beta v'=0$ for all $\beta\in \Pi$: First of all
  \begin{align*}
    E_\alpha \cdot v' =& T_{s_\alpha}\inv(E_\alpha) v'
    \\
    =& -K_\alpha F_\alpha v'
    \\
    =& - K_\alpha v_{s_\alpha.\lambda} + U_q v_{s_\alpha.\lambda}
    \\
    =& 0.
  \end{align*}
  For $\beta \in \Pi\backslash\{ \alpha \}$
  \begin{align*}
    E_\beta \cdot v' =& T_{s_\alpha}\inv(E_\beta) v'.
    \\
    =& \sum_{i=0}^{-\left<\beta,\alpha^\vee\right>} (-1)^i
    q_\alpha^{-i} E_\alpha^{(i)}E_\beta
    E_\alpha^{(-\left<\beta,\alpha^\vee\right>-i)}v'
    \\
    =& (-1)^{\left<\beta,\alpha^\vee\right>}
    q_\alpha^{\left<\beta,\alpha^\vee\right>}E_{\alpha}^{\left(-\left<\beta,\alpha^\vee\right>\right)}
    E_{\beta} v'
    \\
    =& (-1)^{\left<\beta,\alpha^\vee\right>}
    q_\alpha^{\left<\beta,\alpha^\vee\right>}E_{\alpha}^{\left(-\left<\beta,\alpha^\vee\right>\right)}
    F_{\alpha}\inv E_{\beta} v_{s_\alpha.\lambda} + U_q
    v_{s_\alpha.\lambda}
    \\
    =& 0
  \end{align*}
  since $E_{\alpha}v'=0$ and $E_{\beta}v_{s_\alpha.\lambda}=0$ by the
  above.

  So $v'$ is a highest weight vector and $v'$ has weight $\lambda$:
  For $\mu\in Q$:
  \begin{align*}
    K_\mu \cdot v' =& K_{s_\alpha \mu } v'
    \\
    =& K_{s_\alpha \mu } F_{\alpha}\inv v_{s_\alpha.\lambda} + U_q
    v_{s_\alpha.\lambda}
    \\
    =& q^{(s_\alpha(\mu)|\alpha)} s_{\alpha}.\lambda(K_{s_\alpha \mu})
    F_{\alpha}\inv v_{s_\alpha.\lambda} + U_q v_{s_\alpha.\lambda}
    \\
    =& \lambda(K_{\mu}) F_{\alpha}\inv v_{s_\alpha.\lambda} + U_q
    v_{s_\alpha.\lambda}.
  \end{align*}

  So $L(\lambda)$ is a subquotient of ${^{\bar{s_\alpha}}}(\phi_{F_\alpha,\lambda(K_\alpha)^a}.L(\lambda)_{F_\alpha})$
  hence ${^{s_\alpha}}L(\lambda)$ is a subquotient of $\phi_{F_\alpha,\lambda(K_\alpha)^a}.L(\lambda)_{F_\alpha}$.
  Consider the vector
  \begin{equation*}
    v''=F_\alpha^{-1} v_\lambda + U_q v_{\lambda}\in
    {^{\bar{s_\alpha}}}\left(L(\lambda)_{F_\alpha}/(U_q v_{\lambda} )\right).
  \end{equation*}
  By an argument analog to above we get $E_\beta \cdot v'' = 0$ for
  all $\beta\in \Pi \backslash\{ \alpha\}$ since $E_\beta$ and
  $F_\alpha\inv$ commutes and $v_\lambda$ is a highest weight
  vector. We get $E_\alpha\cdot v'' = 0$ by the following:
  \begin{align*}
    E_\alpha \cdot v'' =& T_{s_\alpha}\inv(E_\alpha) v''
    \\
    =& -K_\alpha F_\alpha v''
    \\
    =& - q^{-2} F_\alpha K_\alpha F_\alpha\inv v_\lambda + U_q
    v_\lambda
    \\
    =& 0.
  \end{align*}
  
  So $v''$ is a highest weight vector in
  ${^{\bar{s_\alpha}}}\left( L(\lambda)_{F_\alpha}/(U_q v_{\lambda}
    )\right)$. $v''$ has weight $s_\alpha. \lambda$: For $\mu\in Q$:
  \begin{align*}
    K_\mu \cdot v'' =& K_{s_\alpha \mu} v''
    \\
    =& K_{s_\alpha \mu} F_\alpha\inv v_\lambda + U_q v_\lambda
    \\
    =& q^{(s_\alpha(\mu)|\alpha)} \lambda(K_{s_\alpha \mu}) v''
    \\
    =& (q^{-\alpha}s_\alpha \lambda ) (K_{\mu}) v''.
  \end{align*}

  Hence $L(s_\alpha. \lambda)$ is a subquotient of
  ${^{\bar{s_\alpha}}}L(\lambda)_{F_\Sigma}$ and therefore ${^{s_\alpha}}L(
  s_\alpha. \lambda)$ is a subquotient of $L(\lambda)_{F_\Sigma}$.
\end{proof}

\begin{lemma}
  \label{lemma:30}
  Let $\lambda\in X$ be such that $L(\lambda)$ is an infinite
  dimensional admissible module of degree $d$. Let $\alpha\in
  \Pi$. Then
  \begin{equation*}
    \mathcal{EXT}(L(\lambda))\iso \mathcal{EXT}({^{s_\alpha}}L(\lambda))
  \end{equation*}
  and if $\lambda(K_\alpha) \not \in \pm q_\alpha^{\setN}$ then
  $\mathcal{EXT}(L(\lambda))$ contains $L(s_\alpha.\lambda)$ and
  ${^{s_\alpha}}L(s_\alpha.\lambda)$ as $U_q$-submodules, where
  $s_\alpha.\lambda := q^{- \rho}
  s_\alpha(q^\rho\lambda)=q^{-\alpha}s_\alpha \lambda$.
\end{lemma}
\begin{proof}
  Assume first that $\lambda(K_\alpha) \not \in \pm
  q_\alpha^{\setN}$. By Proposition~\ref{prop:9} the $U_q$-module
  $\bigoplus_{t\in \setC^*/q^{\setZ}} \phi_{F_{\alpha} ,t}.
  L(\lambda)_{F_\alpha}$ contains $L(s_\alpha.\lambda)$,
  ${^{s_\alpha}}L(\lambda)$ and ${^{s_\alpha}}L(s_\alpha.\lambda)$ as
  subquotients.  By Proposition~\ref{prop:23} and
  Proposition~\ref{prop:15} this finishes the proof of the claim when
  $\lambda(K_\alpha) \not \in \pm q_\alpha^{\setN}$.

  Assume now that $\lambda(K_\alpha) = \pm q_\alpha^{k}$ for some
  $k\in \setN$: If $\lambda(K_\alpha)= q_\alpha^{k}$ it is easy to
  prove that $L(\lambda) \iso {^{s_\alpha}}L(\lambda)$. Assume from
  now on that $\lambda(K_\alpha)=-q_\alpha^k$. We have
  \begin{equation*}
    \mathcal{EXT}(L(\lambda)) = \left( \bigoplus_{t\in (\setC^*)/q^{\setZ^n}} \phi_{F_\Sigma,t}.L(\lambda)_{F_\Sigma} \right)^{ss}
  \end{equation*}
  for some set of commuting roots $\Sigma=\{\beta_1,\dots,\beta_n\}$
  that is a basis of $Q$ with $-\Sigma\subset T_{L(\lambda)}$. Since
  $\Sigma$ is a basis of $Q$ there exists $a_1,\dots,a_n\in \setZ$
  such that $\alpha = \sum_{i=1}^n a_i \beta_i$. Let $v_\lambda$ be a
  highest weight vector in $L(\lambda)$. We will show that
  $v_0:=\phi_{F_{\Sigma},((-1)^{a_1'},\dots,(-1)^{a_n'})}.F_\alpha^iv_\lambda\in
  {^{\bar{s_\alpha}}}\mathcal{EXT}(L(\lambda))$ is a highest weight
  vector of weight $\lambda$ where
  $a_i'=\frac{2a_i}{(\alpha|\alpha)}$. This will imply
  $\mathcal{EXT}({^{s_\alpha}}L(\lambda)) \iso
  \mathcal{EXT}(L(\lambda))$ by Proposition~\ref{prop:15}. The weight
  of $v_0$: Let $\mu\in Q$:
  \begin{align*}
    K_\mu \cdot v_0 =&
    K_{s_\alpha(\mu)}\phi_{F_\Sigma,((-1)^{a_1'},\dots,(-1)^{a_n'})}.F_\alpha^i
    v_\lambda
    \\
    =& (-1)^{\left(\sum_{i=1}^n a_i'
        \beta_i|\mu\right)}q^{i(\alpha|\mu)} \lambda(K_\mu
    K_{\alpha}^{-\left< \mu,\alpha^\vee\right>})
    \phi_{F_\Sigma,((-1)^{a_1'},\dots,(-1)^{a_n'})}.F_{\alpha}^i
    v_\lambda
    \\
    =& (-1)^{\left<\mu,\alpha^\vee\right>}q_\alpha^{i\left<
        \mu,\alpha^\vee\right>} (- q_\alpha^i)^{-
      \left<\mu,\alpha^\vee\right>} \lambda(K_\mu) v_0
    \\
    =& \lambda(K_\mu) v_0.
  \end{align*}
  By Proposition~\ref{prop:25}
  $\phi_{F_\beta,(-1)^{\frac{2}{(\beta|\beta)}}}(E_{\alpha'})
  =E_{\alpha'}$ and
  $\phi_{F_\beta,(-1)^{\frac{2}{(\beta|\beta)}}}(F_{\alpha'}) = \pm
  F_{\alpha'}$ for any $\alpha'\in \Pi$ and any $\beta\in \Phi^+$. So
  $\phi_{F_\Sigma,((-1)^{a_1'},\dots,(-1)^{a_n'})}(E_\beta)$,
  $\beta\in \Pi\backslash\{\alpha\}$ and
  $\phi_{F_\Sigma,((-1)^{a_1'},\dots,(-1)^{a_n'})}(F_\alpha)$ kills
  $F_{\alpha}^iv_\lambda\in L(\lambda)$ because $E_\beta$ and
  $F_\alpha$ does. Hence $E_\beta$, $\beta\in \Pi$ kills $v_0$ by the
  same argument as in the proof of Proposition~\ref{prop:9} when
  proving that $v'$ is a highest weight vector.
\end{proof}

\begin{thm}
  \label{thm:EXT_contains_highest_weight}
  Let $L$ be an infinite dimensional admissible simple module of
  degree $d$. Then the $U_q$-module $\mathcal{EXT}(L)$ contains an
  infinite dimensional admissible simple highest weight module
  $L(\lambda)$ of degree $d$ for some weight $\lambda\in
  X$. Furthermore for any $x\in W$:
  \begin{equation*}
    {^x}\mathcal{EXT}(L) \iso \mathcal{EXT}(L).
  \end{equation*}
\end{thm}
\begin{proof}
  Let $w\in W$ be such that $w(F_L\backslash F_L^s) \subset \Phi^+$
  and $w(T_L\backslash T_L^s)\subset \Phi^-$. Set
  $L'={^{\bar{w\inv}}}L$ (then ${^{w\inv}}L'= L$). We will show the
  result first for $L'$ by induction on $|T_{L'}^+|$. If
  $|T_{L'}^+|=0$ then $L'$ is itself a highest weight module. Assume
  $|T_{L'}^+|>0$. Let $\beta\in T_{L'}^+$. Then $\beta \in T_{L'}^s$
  since $T_{L'}\backslash T_{L'}^s \subset \Phi^-$. So $-\beta \in
  T_{L'}$. Then by Lemma~\ref{lemma:24} there exists a $b\in \setC^*$
  such that $\phi_{F_\beta,b}.L'_{F_\beta}$ contains a $U_q$-submodule
  $L''$ with $T_{L''}\subset T_{L'}$ and $\beta\not \in T_{L''}$. By
  Proposition~\ref{prop:23} and Proposition~\ref{prop:15}
  $\mathcal{EXT}(L')\iso \mathcal{EXT}(L'')$ as $U_q$-modules. By
  induction $\mathcal{EXT}(L'')$ contains an infinite dimensional
  admissible simple highest weight module $L(\lambda)$ for some
  $\lambda$. So $\mathcal{EXT}(L')\iso \mathcal{EXT}(L(\lambda))$ by
  Proposition~\ref{prop:15}. Choose a reduced expression
  $s_{i_r}\cdots s_{i_1}$ for $w\inv$. By Proposition~\ref{prop:19}
  and Lemma~\ref{lemma:30}
  \begin{align*}
    \mathcal{EXT}(L) \iso& \mathcal{EXT}({^{w\inv}}L')
    \\
    \iso& {^{w\inv}}\mathcal{EXT}(L')
    \\
    \iso& {^{ w\inv }}\mathcal{EXT}(L(\lambda))
    \\
    \iso& {^{s_{i_r}\cdots
        s_{i_{2}}}}\mathcal{EXT}({^{s_{i_1}}}L(\lambda))
    \\
    \iso& {^{s_{i_r}\cdots s_{i_{2}}}}\mathcal{EXT}(L(\lambda))
    \\
    \vdots&
    \\
    \iso& \mathcal{EXT}(L(\lambda)).
  \end{align*}
  So $\mathcal{EXT}(L)$ contains a simple highest weight module
  $L(\lambda)$. For any $x\in W$ we can do as above to show
  ${^x}\mathcal{EXT}(L)\iso \mathcal{EXT}({^x}L(\lambda))\iso
  \mathcal{EXT}(L(\lambda))\iso \mathcal{EXT}(L)$.
\end{proof}

\begin{cor}
  \label{cor:1}
  Let $L$ be a simple torsion free module. Then there exists a set of
  commuting roots $\Sigma$ that is a basis of $Q$ with corresponding
  Ore subset $F_\Sigma$, a $\lambda\in X$ and $\mathbf{b}\in
  (\setC^*)^n$ such that $-\Sigma\subset T_{L(\lambda)}$ and $L\iso
  \phi_{F_\Sigma,\mathbf{b}}.L(\lambda)_{F_{\Sigma}}$
\end{cor}
\begin{proof}
  By Theorem~\ref{thm:EXT_contains_highest_weight}
  $\mathcal{EXT}(L)\iso \mathcal{EXT}(L(\lambda))$ for some
  $\lambda\in X$. So $L$ is a $U_q$-submodule of
  $\mathcal{EXT}(L(\lambda))$. Let $\Sigma$ be a set of commuting
  roots such that $-\Sigma \subset L(\lambda)$ (exists by
  Lemma~\ref{lemma:26} by setting $w=e$, the neutral element in $W$)
  then
  \begin{equation*}
    \mathcal{EXT}(L(\lambda)) = \left( \bigoplus_{t\in (\setC^*)^n/q^{\setZ^n}} \phi_{F_\Sigma,t}.L(\lambda)_{F_\Sigma} \right)^{ss}.
  \end{equation*}
  Since $L$ is simple we must have that $L$ is a submodule of
  $\phi_{F_\Sigma,\mathbf{b}}.L(\lambda)_{F_\Sigma}$ for some
  $\mathbf{b}\in(\setC^*)^n$. By Proposition~\ref{prop:15} and
  Lemma~\ref{lemma:13} $\dim
  \left(\phi_{F_\Sigma,\mathbf{b}}.L(\lambda)_{F_\Sigma}\right)_\lambda
  = \dim L_\lambda$ for all $\lambda \in \wt L$ so we have $L\iso
  \phi_{F_\Sigma,\mathbf{b}}.L(\lambda)_{F_\Sigma}$.
\end{proof}

So to classify torsion free simple modules we need to classify the
admissible infinite dimensional simple highest weight modules
$L(\lambda)$ and then we need to determine the $t\in
(\setC^*)^n/q^{\setZ^n}$ such that
$\phi_{F_\Sigma,t}.L(\lambda)_{F_\Sigma}$ is simple. Furthermore we
have that if there exists an admissible infinite dimensional simple
module then there exists a torsion free simple module. In the classical
case torsion free modules only exists if $\mathfrak{g}$ is of type $A$
or $C$ so we expect the same to be true in the quantum group case. We
show this in Section~\ref{sec:class-admiss-modul-1}.

\section{Classification of simple torsion free
  $U_q(\mathfrak{sl}_2)$-modules}
\label{sec:A_classification_sl2}
In this section let $\mathfrak{g}=\mathfrak{sl}_2$. In this case there
is a single simple root $\alpha$. It is natural to identify $X$ with
$\setC^*$ via $\lambda \mapsto \lambda(K_\alpha)$. We define $F =
F_\alpha$, $E=E_\alpha$ and $K^{\pm 1}= K_\alpha^{\pm 1}$. Let
$\lambda\in \setC^*\backslash\{\pm q^{\setN}\}$ and consider the
simple highest weight module $L(\lambda)$. Let $0\neq v_0\in
L(\lambda)_{\lambda}$. $\wt L = q^{- 2\setN}\lambda$ so $L(\lambda)$
is an admissible infinite dimensional highest weight module. Thus
$\mathcal{EXT}(L(\lambda))$ contains a torsion free module by
Theorem~\ref{thm:existence_of_torsion_free_modules}. Let
$b\in\setC^*$. We will describe the action on the module
$\phi_{F,b}.L(\lambda)_{(F)}$ and determine exactly for which $b$'s
$\phi_{F,b}.L(\lambda)_{(F)}$ is torsion free.

Let $v_i = F^i \phi_{F,b}.v_0$ for all $i\in\setZ$. Then we have for
$i\in \setZ$
\begin{align*}
  F v_i =& v_{i+1}
  \\
  K^{\pm 1} v_i =& q^{-2i} b^{\mp 2} \lambda v_i
  \\
  E v_i =& \frac{(q^{i}b-q^{-i}b\inv)(q^{1-i} b\inv \lambda - q^{i-1}b
    \lambda\inv)}{(q-q\inv)^2} v_{i-1}.
\end{align*}
We see that unless $b=\pm q^{i}$ or $b=\pm q^{i}\lambda$ for some
$i\in\setZ$ then $\phi_{F,b}.L(\lambda)_{(F)}$ is torsion free. In
this case we see that $\phi_{F,-b}=\phi_{F,b}$ since for all $u\in
U_q(\mathfrak{sl}_2)$, $\phi_{F,b}(u)$ is Laurent polynomial in $b^2$.

So in this case $\mathcal{EXT}(L(\lambda))$ contains a maximum of four
different simple submodules which are \emph{not} torsion free: We have
$(\phi_{F,\pm q^{i}}.L(\lambda)_{(F)})^{ss} \iso
(L(\lambda)_{(F)})^{ss}\iso L(\lambda)\oplus
{^{s_\alpha}}L(s_\alpha.\lambda)$ (which can be seen directly from the
calculations but also follows from Corollary~\ref{cor:6} and the fact
that $\phi_{F,-b}=\phi_{F,b}$) and $(\phi_{F,\pm
  q^{i}\lambda}.L(\lambda)_{(F)})^{ss} \iso
(L(s_\alpha.\lambda)_{(F)})^{ss}\iso L(s_\alpha.\lambda)\oplus
{^{s_\alpha}}L(\lambda)$ if $\lambda \not \in \pm q^{\setZ}$.

The weights of $\phi_{F,b}.L(\lambda)_{(F)}$ are $b^{-\alpha} \wt
L(\lambda)_{(F)} = q^{2\setZ} b^{-2} \lambda$. Suppose we want to find
a torsion free $U_q(\mathfrak{sl}_2)$-modules with integral
weights. Then we just need to find $\lambda,b\in \setC^*$ such that
$\lambda\not \in \pm q^{\setZ_{\geq 0}}$, $b\not \in \pm q^{\setZ}$
and $b \not \in \pm q^{\setZ} \lambda$ such that $b^{-2} \lambda \in
q^{\setZ}$. For example choose a square root $q^{1/2}$ of $q$ and set
$\lambda = q^{-1}$ and $b=q^{1/2}$. Then we have a torsion free module
$L=\spa{\setC}{v_i|i\in \setZ}$ with action given by:
\begin{align*}
  F v_i =& v_{i+1}
  \\
  K v_i =& q^{- 2i-2} v_i
  \\
  E v_i =& \frac{(q^{1/2+i}-q^{-1/2-i})(q^{-1/2-i} -
    q^{i+1/2})}{(q-q\inv)^2} v_{i-1}
  \\
  =& \frac{q(q^{-i-1}-q^i)^2}{(q-q\inv)^2} v_{i-1}.
\end{align*}
In this paper we only focus on quantized enveloping algebras over
$\setC$ but note that we can define, for a general field $\mathbb{F}$
with $q\in \mathbb{F}\backslash\{0\}$ a non-root of unity, a simple
torsion free $U_{\mathbb{F}}(\mathfrak{sl}_2)$-module with integral
weights by the above formulas (here
$U_{\mathbb{F}}(\mathfrak{sl}_2)=U_A\tensor_A \mathbb{F}$ where
$\mathbb{F}$ is considered an $A$-algebra by sending $v$ to $q$).

\section{An example for $U_q(\mathfrak{sl}_3)$}
\label{sec:an-example-u_qm}
In this section we will show how we can construct a specific torsion
free simple module for $U_q(\mathfrak{sl}_3)$. In
Section~\ref{sec:type-a_n-calc} we classify all torsion free
$U_q(\mathfrak{sl}_n)$-modules with $n\geq 3$ so this example is of
course included there. If you are only interested in the general
classification you can skip this section but the calculations in this
section gives a taste of the calculations needed in the general case
in Section~\ref{sec:type-a_n-calc} and they show a phenomona that does
not happen in the classical case.

Let $\alpha_1$ and $\alpha_2$ be the two simple roots of the root
system. We will consider the set of commuting roots $\Sigma = \{
\beta_1,\beta_2\}$ where $\beta_1 = \alpha_1$ and $\beta_2 =
\alpha_1+\alpha_2$. Set $F_{\beta_1}:=F_{\alpha_1}$ and $F_{\beta_2}:=
T_{s_1}(F_{\alpha_2})=F_{\alpha_2}F_{\alpha_1}-q
F_{\alpha_1}F_{\alpha_2}=[F_{\alpha_2},F_{\alpha_1}]_q$. We have
$(\beta_1|\beta_2)=1$ and $0=[F_{\beta_2},F_{\beta_1}]_q =
F_{\beta_2}F_{\beta_1}-q^{-1} F_{\beta_1}F_{\beta_2}$ or equivalently
$F_{\beta_1}F_{\beta_2} = q F_{\beta_2}F_{\beta_1}$. Let $\lambda\in
X$ be determined by $\lambda(K_{\alpha_1})=q^{-1}$ and
$\lambda(K_{\alpha_2})=1$. Then $M(s_{\alpha_1}.\lambda)$ is a
submodule of $M(\lambda)$ and
$L(\lambda)=M(\lambda)/M(s_{\alpha_2}.\lambda)=M(\lambda)/M(q^{-\alpha_2}\lambda)$
is admissible of degree $1$. Let $\xi=e^{2\pi i/3}$. We will show that
$\phi_{F_\Sigma,(\xi,\xi)}.L(\lambda)_{F_\Sigma}$ is a torsion free
module. We have here a phenomona that does not happen in the classical
case: $\wt L(\lambda)_{F_\Sigma}=\wt
\phi_{F_\Sigma,(\xi,\xi)}.L(\lambda)_{F_\Sigma}$ but
$L(\lambda)_{F_\Sigma} \not \iso
\phi_{F_\Sigma,(\xi,\xi)}.L(\lambda)_{F_\Sigma}$ as $U_q$-modules
since one is simple and torsion free and the other isn't (compare
to~\cite[Section~10]{Mathieu} where Mathieu classifies the torsion
free simple modules by determining for a coherent family $\mathcal{M}$
for which cosets $t\in \mathfrak{h}^*/Q$, $\mathcal{M}[t]$ is torsion
free).

We will show that $E_{\alpha_1}$ and $E_{\alpha_2}$ act injectively on
the module $\phi_{F_\Sigma,(\xi,\xi)}.L(\lambda)_{F_\Sigma}$. So we
need to calculate $\phi_{F_\Sigma,(\xi,\xi)}(E_{\alpha_1})$ and
$\phi_{F_\Sigma,(\xi,\xi)}(E_{\alpha_2})$. $\phi_{F_\Sigma,(\xi,\xi)}=\phi_{F_{\beta_1},\xi}\circ
\phi_{F_{\beta_2},\xi}$. We have
\begin{align*}
  [E_{\alpha_1},F_{\beta_2}] =&
  F_{\alpha_2}[E_{\alpha_1},F_{\alpha_1}] - q
  [E_{\alpha_1},F_{\alpha_1}] F_{\alpha_2}
  \\
  =& F_{\alpha_2} \frac{K_{\alpha_1} - K_{\alpha_1}\inv}{q-q\inv} - q
  F_{\alpha_2} \frac{q K_{\alpha_1} - q\inv K_{\alpha_1}\inv}{q-q\inv}
  \\
  =& F_{\alpha_2} \frac{K_{\alpha_1} - q^2 K_{\alpha_1}}{q-q\inv}
  \\
  =& -F_{\alpha_2} q \frac{q-q\inv}{q-q\inv}K_{\alpha_1}
  \\
  =& -q F_{\alpha_2} K_{\alpha_1}.
\end{align*}
We can show by induction that
\begin{align*}
  [E_{\alpha_1}, F_{\beta_2}^j] =& - q^{2-j} [j] F_{\beta_2}^{j-1}
  F_{\alpha_2}K_{\alpha_1}
\end{align*}
for any $j\in \setN$. Using that $\phi_{F_{\beta_2},b}(E_{\alpha_1})$
is Laurent polynomial and equal to $F_{\beta_2}^{-j} E_{\alpha_1}
F_{\beta_2}^j$ for $b=q^j$ we get
\begin{align*}
  \phi_{F_{\beta_2},b}(E_{\alpha_1}) =& E_{\alpha_1} - q^2 b\inv
  \frac{b - b\inv }{q-q\inv} F_{\beta_2}\inv F_{\alpha_2} K_{\alpha_1}.
\end{align*}
We have $F_{\beta_2}F_{\beta_1}=q\inv F_{\beta_1}F_{\beta_2}$ so
$F_{\beta_1}^{-i}F_{\beta_2}F_{\beta_1}^i=q^{-i}F_{\beta_2}$ thus
$\phi_{F_{\beta_1},b}(F_{\beta_2}^{-1}) = b F_{\beta_2}^{- 1}$. We
have
\begin{align*}
  \phi_{F_{\alpha_1},b}(F_{\alpha_2}) =& b F_{\alpha_2} - \frac{b -
    b\inv}{q - q\inv} F_{\alpha_1}\inv ( q F_{\alpha_1} F_{\alpha_2} -
  F_{\alpha_2} F_{\alpha_1}) \\=& b F_{\alpha_2} +
  \frac{b-b\inv}{q-q\inv} F_{\alpha_1}\inv F_{\beta_2}
\end{align*}
and
\begin{align*}
  \phi_{F_{\beta_1},b_1}&( \phi_{F_{\beta_2},b_2}(E_{\alpha_1}))
  \\
  =& \phi_{F_{\alpha_1},b_1}\left( E_{\alpha_1} - q^2b_2\inv
    \frac{b_2- b_2\inv }{q-q\inv} F_{\beta_2}\inv F_{\alpha_2}
    K_{\alpha_1} \right)
  \\
  =& E_{\alpha_1} + F_{\alpha_1}\inv \frac{(b_1-b_1\inv)(qb_1\inv
    K_{\alpha_1} - q\inv b_1 K_{\alpha_1}\inv)}{(q-q\inv)^2}
  \\
  &- q^2 b_2\inv \frac{b_2 - b_2\inv}{q-q\inv} b_1 F_{\beta_2}\inv
  \left( b_1 F_{\alpha_2} + \frac{b_1-b_1\inv}{q-q\inv}
    F_{\alpha_1}\inv F_{\beta_2} \right) b_1^{-2} K_{\alpha_1}
  \\
  =& E_{\alpha_1} + F_{\alpha_1}\inv \frac{(b_1-b_1\inv)(qb_1\inv
    K_{\alpha_1} - q\inv b_1 K_{\alpha_1}\inv)}{(q-q\inv)^2}
  \\
  &- q^2 b_2\inv \frac{b_2 - b_2\inv }{q-q\inv} F_{\beta_2}\inv
  F_{\alpha_2} K_{\alpha_1}
  \\
  & - q b_1\inv b_2\inv \frac{(b_2 - b_2\inv
    )(b_1-b_1\inv)}{(q-q\inv)^2} F_{\alpha_1}\inv K_{\alpha_1}
  \\
  =& E_{\alpha_1} + b_2\inv F_{\alpha_1}\inv \frac{(b_1-b_1\inv)(q
    b_1\inv b_2^{-1}K_{\alpha_1} - q\inv b_1b_2
    K_{\alpha_1}\inv)}{(q-q\inv)^2}
  \\
  &- q^2 b_2\inv \frac{b_2 - b_2\inv }{q-q\inv} F_{\beta_2}\inv
  F_{\alpha_2} K_{\alpha_1}.
\end{align*}

Let $v_\lambda'$ be a highest weight vector in $L(\lambda)$ and set
$v_\lambda = 1\tensor v_\lambda'\in L(\lambda)_{F_\Sigma}$. We have
$F_{\alpha_2}v_\lambda = 0$ by construction so we have
\begin{align*}
  \phi_{F_\Sigma,(b_1,b_2)}(E_{\alpha_1})v_\lambda =& b_2\inv
  \frac{(b_1-b_1\inv)(b_1\inv b_2\inv - b_1 b_2)}{(q-q\inv)^2}
  F_{\alpha_1}\inv v_\lambda.
\end{align*}

$\phi_{F_\Sigma,(c_1,c_2)}.L(\lambda)_{F_\Sigma}$ is spanned by the
elements $F_{\beta_1}^i F_{\beta_2}^j
\phi_{F_\Sigma,(c_1,c_2)}.v_\lambda$, $i,j\in \setZ$ because every
weight space is one-dimensional and $F_{\beta_1}^iF_{\beta_2}^j$ acts
injectively.  Since
\begin{align*}
  F_{\beta_2}^{-j} F_{\beta_1}^{-i} E_{\alpha_1} F_{\beta_1}^i
  F_{\beta_2}^j =& F_{\beta_1}^{-i} F_{\beta_2}^{-j} E_{\alpha_1}
  F_{\beta_2}^j F_{\beta_1}^{i}
  \\
  =&\phi_{F_{\beta_1},q^i}(\phi_{F_{\beta_2},q^j}(E_{\alpha_1}))
  \\
  =& \phi_{F_\Sigma,(q^i,q^j)}(E_{\alpha_1})
\end{align*}
we have
\begin{align*}
  E_{\alpha_1} F_{\beta_1}^i
  F_{\beta_2}^j&\phi_{F_\Sigma,(c_1,c_2)}.v_\lambda
  \\
  =& F_{\beta_1}^i F_{\beta_2}^j
  \phi_{F_\Sigma,(q^i,q^j)}(E_{\alpha_1})
  \phi_{F_\Sigma,(c_1,c_2)}.v_\lambda
  \\
  =& F_{\beta_1}^i F_{\beta_2}^j
  \phi_{F_\Sigma,(c_1,c_2)}. \phi_{F_\Sigma,(q^i c_1,q^j
    c_2)}(E_{\alpha_1}) v_\lambda
  \\
  =& q^{-j}c_2\inv \frac{(q^ic_1-q^{-i}c_1\inv)(q^{-i-j}c_1\inv
    c_2\inv - q^{i+j}c_1 c_2)}{(q-q\inv)^2}F_{\beta_1}^i F_{\beta_2}^j
  \phi_{F_\Sigma,(c_1,c_2)}. F_{\alpha_1}\inv v_\lambda
  \\
  =& \frac{(q^ic_1-q^{-i}c_1\inv)(q^{-i-j}c_1\inv c_2\inv - q^{i+j}c_1
    c_2)}{(q-q\inv)^2}F_{\beta_1}^{i-1} F_{\beta_2}^j
  \phi_{F_\Sigma,(c_1,c_2)}. v_\lambda.
\end{align*}
This is only zero when $c_1 = \pm q^{-i}$ or $c_1c_2 = \pm
q^{-i-j}$. Set $c_1=c_2=e^{2\pi i/3}=:\xi$. Then we have shown that
$E_{\alpha_1}$ acts injectively on
$\phi_{F_\Sigma,(\xi,\xi)}.L(\lambda)_{F_\Sigma}$.

Now we will show that $E_{\alpha_2}$ acts injectively on
$F_{\Sigma,(\xi,\xi)}.L(\lambda)_{F_\Sigma}$.  We can show by
induction that
\begin{align*}
  [E_{\alpha_2},F_{\beta_2}^j] =& [j]
  F_{\alpha_1}F_{\beta_2}^{j-1}K_{\alpha_2}\inv
\end{align*}
so $\phi_{F_{\beta_2},b}(E_{\alpha_2}) = E_{\alpha_2} + b
\frac{b-b\inv}{q-q\inv} F_{\alpha_1} F_{\beta_2}\inv K_{\alpha_2}\inv$
and
\begin{align*}
  \phi_{F_\Sigma,(b_1,b_2)}(E_{\alpha_2}) =&
  \phi_{F_{\beta_1},b_1}(\phi_{F_{\beta_2},b_2}(E_{\alpha_2}))
  \\
  =& E_{\alpha_2} + b_2 \frac{b_2-b_2\inv}{q-q\inv} F_{\alpha_1}
  F_{\beta_2}\inv K_{\alpha_2}\inv.
\end{align*}
Thus
\begin{align*}
  E_{\alpha_2} F_{\beta_1}^i F_{\beta_2}^j
  \phi_{F_\Sigma,(c_1,c_2)}.v_\lambda =& F_{\beta_1}^i F_{\beta_2}^j
  \phi_{F_\Sigma,(q^i,q^j)}(E_{\alpha_2})
  \phi_{F_\Sigma,(c_1,c_2)}.v_\lambda
  \\
  =& F_{\beta_1}^i F_{\beta_2}^j
  \phi_{F_\Sigma,(c_1,c_2)}. \phi_{F_\Sigma,(q^i c_1,q^j
    c_2)}(E_{\alpha_2})v_\lambda
  \\
  =&F_{\beta_1}^i F_{\beta_2}^j \phi_{F_\Sigma,(c_1,c_2)}.  c_2
  \frac{q^jc_2-q^{-j}c_2\inv}{q-q\inv} F_{\alpha_1}F_{\beta_2}\inv
  K_{\alpha_2}\inv v_\lambda
  \\
  =& q^{-j-1} c_2 \frac{q^jc_2-q^{-j}c_2\inv}{q-q\inv}
  F_{\beta_1}^{i+1} F_{\beta_2}^{j-1} \phi_{F_\Sigma,(c_1,c_2)}.
  v_\lambda.
\end{align*}
We see that this is nonzero only if $c_2 = \pm q^{-j}$ so again
setting $c_1=c_2=\xi$ ensures that this is nonzero.

We have shown that the $U_q$-module
$\phi_{F_\Sigma,(\xi,\xi)}.L(\lambda)_{F_\Sigma}$ has a basis
$F_{\beta_1}^i F_{\beta_2}^j \phi_{F_\Sigma,(\xi,\xi)}.v_\lambda$,
$i,j\in \setZ$ and we have
\begin{align*}
  F_{\beta_1} F_{\beta_1}^i F_{\beta_2}^j
  \phi_{F_\Sigma,(\xi,\xi)}.v_\lambda =& F_{\beta_1}^{i+1}
  F_{\beta_2}^{j-1} \phi_{F_\Sigma,(\xi,\xi)}.v_\lambda
  \\
  F_{\beta_2} F_{\beta_1}^i F_{\beta_2}^j
  \phi_{F_\Sigma,(\xi,\xi)}.v_\lambda =& q^{-j}F_{\beta_1}^{i}
  F_{\beta_2}^{j+1} \phi_{F_\Sigma,(\xi,\xi)}.v_\lambda
  \\
  E_{\alpha_1} F_{\beta_1}^i F_{\beta_2}^j
  \phi_{F_\Sigma,(\xi,\xi)}.v_\lambda =& C_1 F_{\beta_1}^{i-1}
  F_{\beta_2}^{j} \phi_{F_\Sigma,(\xi,\xi)}.v_\lambda
  \\
  E_{\alpha_1} E_{\alpha_2} F_{\beta_1}^i F_{\beta_2}^j
  \phi_{F_\Sigma,(\xi,\xi)}.v_\lambda =& C_2 F_{\beta_1}^{i}
  F_{\beta_2}^{j-1} \phi_{F_\Sigma,(\xi,\xi)}.v_\lambda
\end{align*}
for some nonzero constants $C_1,C_2\in \setC^*$. We see that any of
the basis vectors $F_{\beta_1}^i F_{\beta_2}^j
\phi_{F_\Sigma,(\xi,\xi)}.v_\lambda$ can be mapped injectively to any
other basis vector $F_{\beta_1}^{i'} F_{\beta_2}^{j'}
\phi_{F_\Sigma,(\xi,\xi)}.v_\lambda$ by elements of $U_q$ so
$\phi_{F_\Sigma,(\xi,\xi)}.L(\lambda)_{F_\Sigma}$ is a simple
module. The module is torsion free by Proposition~\ref{prop:3}.

\section{Classification of admissible simple highest weight modules}

\subsection{Preliminaries}
\label{sec:class-admiss-simple}
In this section we prove some preliminary results with the goal to
classify all admissible simple highest weight modules.  We will only
focus on non-integral weights since we have the following theorem
from~\cite{CatO}:
\begin{thm}
  \label{thm:integral}
  Assume $q\in \setC\backslash\{0\}$ is transcendental. Let
  $\lambda:U_q^0 \to \setC$ be a weight such that $\lambda(K_\alpha)=
  q_\beta^i$ for some $i\in\setZ$ for every $\alpha\in \Pi$ -
  i.e. $\lambda \in q^Q$. Say $\lambda= q^\mu$, $\mu\in Q$. Let
  $L_{\setC}(\mu)$ denote the simple highest weight
  $\mathfrak{g}$-module of highest weight $\mu$. Then the character of
  $L(\lambda)$ and $L_\setC(\mu)$ are equal - i.e. for any $\nu\in Q$,
  $\dim L(\lambda)_{q^\nu\lambda} = \dim L_{\setC}(\mu)_{\nu+\mu}$.
\end{thm}
\begin{proof}
  \cite[Corollary~6.3]{CatO}.
\end{proof}
Extending to modules which are not of type~1 is done in the usual way
(cf. e.g.~\cite[Section~5.1--5.2]{Jantzen}).  The above theorem
implies that the integral admissible simple highest weight modules can
be classified from the classification of the classical admissible simple
highest weight modules when $q$ is transcendental. Hence we need only
to consider weights $\lambda\in X$ such that $\lambda(K_\alpha) \not
\in \pm q^{\setZ}$ for at least one $\alpha\in \Pi$ in this
case. \emph{So in the rest of the paper we will restrict our attention
  to the case when $q$ is transcendental}. If a similar theorem is
true for any non-root-of-unity $q$ then the results in this paper
extend to all non-root-of-unities but the author is not aware of any
such result. 

\begin{thm}
  \label{thm:Jantzen_filtration}
  Let $\lambda\in X$. Then there exists a filtration of $M(\lambda)$,
  $M(\lambda) \supset M_1 \supset \dots \supset M_r$ such that $M_1$
  is the unique maximal submodule of $M(\lambda)$ and
  \begin{equation*}
    \sum_{i=1}^r \ch M_i = \sum_{\substack{\beta\in \Phi^+ \\q^\rho\lambda(K_\beta)\in \pm q_\beta^{\setZ_{>0}}}} \ch M(s_{\beta}.\lambda)
  \end{equation*}
\end{thm}
The filtration is called the Jantzen filtration and the formula is
called the Jantzen sum formula.
\begin{proof}
  This is proved in~\cite[Section~4.1.2-4.1.3]{Joseph}. A proof using
  twisting functors can also be found
  in~\cite[Theorem~6.3]{DHP-twist}.
\end{proof}

\begin{defn}
  Let $\lambda\in X$.
  \begin{equation*}
    A(\lambda) = \{ \alpha \in \Pi | \lambda(K_\alpha) \not \in \pm q_{\alpha}^{\setN} \}.
  \end{equation*}
  Let $\gamma\in \Pi$.
  \begin{equation*}
    D(\gamma) = \{\beta \in \Phi^+|\beta=\sum_{\alpha\in\Pi} m_\alpha \alpha, \, m_\gamma>0\}.
  \end{equation*}
\end{defn}

\begin{lemma}
  \label{lemma:18}
  Let $\lambda\in X$. Let $\gamma\in \Pi$ be such that $\gamma\in
  A(\lambda)$. Then $-D(\gamma)\subset T_{L(\lambda)}$.
\end{lemma}
\begin{proof}
  Let $\beta=\sum_{\alpha\in \Pi}m_\alpha \alpha\in D(\gamma)$. We
  prove by induction over $\height \beta = \sum_{\alpha\in \Pi}
  m_\alpha$ that $-\beta \in T_{L(\lambda)}$.  If $\height \beta=1$
  then $\beta=\gamma$ and $-\gamma\in T_{L(\lambda)}$ by
  Proposition~\ref{prop:9}.

  Assume $\height \beta >1$. Then $\beta-\alpha \in \Phi^+$ for some
  $\alpha\in \Pi$. We have either $\alpha=\gamma$ or $\beta-\alpha\in
  D(\gamma)$. In either case we get $\beta = \beta'+\beta''$ for some
  $\beta',\beta''\in \Phi^+$ with $\beta'\in D(\gamma)$ and $\height
  \beta' < \height \beta$. By induction $-\beta'\in
  T_{L(\lambda)}$. If $-\beta \in F_{L(\lambda)}$ then $-\beta' =
  -\beta + \beta'' \in F_{L(\lambda)}$ since $\Phi^+\subset
  F_{L(\lambda)}$ and $F_{L(\lambda)}$ is closed
  (Proposition~\ref{prop:8}). A contradiction. So $-\beta\in
  T_{L(\lambda)}$.
\end{proof}

\begin{lemma}
  \label{lemma:38}
  Let $\gamma\in \Pi$. $D(\gamma)$ generates $Q$.
\end{lemma}
\begin{proof}
  Let $\left< D(\gamma)\right>$ be the subgroup of $Q$ generated by
  $D(\gamma)$. Assume $\Pi\cap \left<D(\gamma)\right>\neq \Pi$. Let
  $\alpha\not \in \left< D(\gamma) \right>$ be a simple root that is
  connected to an $\alpha'\in \left< D(\gamma) \right>$ (possible
  since the Dynkin diagram of a simple Lie algebra is connected). Then
  $\alpha+\alpha'\in \left<D(\gamma)\right>$. But then $\alpha =
  \alpha+\alpha'-\alpha' \in \left<D(\gamma)\right>$. A
  contradiction. So $\left<D(\gamma)\right>=Q$.
\end{proof}

\begin{lemma}
  \label{lemma:34}
  Let $\lambda\in X$ be a non-integral weight. Assume that
  $L(\lambda)$ is admissible. Then $A(\lambda)$ is connected and
  $|A(\lambda)|\leq 2$.
\end{lemma}
\begin{proof}
  Assume $|A(\lambda)|\geq 2$. Let $\alpha,\alpha'\in A(\lambda)$ be
  two distinct elements. We will show that $\alpha$ and $\alpha'$ are
  connected. So assume $(\alpha|\alpha')=0$ to reach a
  contradiction. $L(\lambda)$ is admissible of some degree $d$. By
  Lemma~\ref{lemma:30} and Proposition~\ref{prop:15}
  ${^{s_\alpha}}L(s_\alpha.\lambda)$ is admissible of the same degree
  $d$ ($L(s_\alpha.\lambda)$ is infinite dimensional since
  $s_\alpha.\lambda(K_{\alpha'})=\lambda(K_{\alpha'})\not \in \pm
  q_{\alpha'}^{\setN}$). Let $\Sigma$ be a set of commuting roots that
  is a basis of $Q$ such that $\alpha\in \Sigma$ and $-\Sigma\subset
  T_{L(\lambda)}$ (Lemma~\ref{lemma:26}). By Proposition~\ref{prop:9}
  ${^{s_\alpha}}L(s_\alpha.\lambda)$ is a subquotient of
  $L(\lambda)_{F_\Sigma}$. We claim that $\Suppess(L(\lambda))\cap
  \Suppess({^{s_\alpha}}L(s_\alpha.\lambda))\neq \emptyset$. If this
  is true then we have for $\nu \in \Suppess(L(\lambda))\cap
  \Suppess({^{s_\alpha}}L(s_\alpha.\lambda))$, $L(\lambda)_\nu \iso
  (L(\lambda)_{F_\Sigma})_\nu \iso
  ({^{s_\alpha}}L(s_\alpha.\lambda))_\nu$ as $(U_q)_0$-modules by
  Lemma~\ref{lemma:13}. But then by Theorem~\ref{thm:Lemire}
  $L(\lambda)\iso {^{s_{\alpha}}}L(s_\alpha.\lambda)$ which is clearly
  a contradiction by looking at the weights of the modules.  So we
  will prove the claim that $\Suppess(L(\lambda))\cap
  \Suppess({^{s_\alpha}}L(s_\alpha.\lambda))\neq \emptyset$:

  We have $-D(\alpha')\subset T_{L(\lambda)}$ and $-D(\alpha') \subset
  T_{{^{s_\alpha}}L(s_\alpha.\lambda)}=s_\alpha(T_{L(s_\alpha.\lambda)})$
  by Lemma~\ref{lemma:18} and the fact that $(\alpha|\alpha')=0$. So
  $-D(\alpha')\subset C(L(\lambda)) \cap
  C({^{s_\alpha}}L(s_\alpha.\lambda))$ thus $C(L(\lambda)) \cap
  C({^{s_\alpha}}L(s_\alpha.\lambda))$ generate $Q$ by
  Lemma~\ref{lemma:38}. This implies that
  $C(L(\lambda))-C({^{s_\alpha}}L(s_\alpha.\lambda))=Q$.  The weights
  of $L(\lambda)$ and ${^{s_\alpha}}L(s_\alpha.\lambda)$ are contained
  in $q^Q \lambda$ so a weight in the essential support of
  $L(\lambda)$ (resp. ${^{s_\alpha}}L(s_\alpha.\lambda)$) is of the
  form $q^{\mu_1}\lambda$ (resp. $q^{\mu_2}\lambda$) for some
  $\mu_1,\mu_2 \in Q$. By the above $q^{C(L(\lambda))+\mu_1}\lambda
  \cap q^{C({^{s_\alpha}}L(s_\alpha.\lambda))+\mu_2}\lambda\neq
  \emptyset$. Since $q^{C(L(\lambda))+\mu_1}\lambda \subset
  \Suppess(L(\lambda))$ and
  $q^{C({^{s_\alpha}}L(s_\alpha.\lambda))+\mu_2}\lambda \subset
  \Suppess({^{s_\alpha}}L(s_\alpha.\lambda))$ we have proved the
  claim.

  So we have proved that any two roots of $A(\lambda)$ are
  connected. Since there are no cycles in the Dynkin diagram of a
  simple Lie algebra we get $A(\lambda)=2$.
\end{proof}

\subsection{Rank $2$ calculations}
\label{sec:rank-2-calculations}
Following the procedure in~\cite[Section~7]{Mathieu} we classify
admissible simple highest weight modules in rank $2$ in order to
classify the modules in higher ranks. We only consider non-integral
weights because of Theorem~\ref{thm:integral}. We assume that $q$ is
transcendental over $\setQ$.

\begin{lemma}
  \label{lemma:8}
  Assume $\mathfrak{g}=\mathfrak{sl}_3$. Let $\lambda\in X$ be a
  non-integral weight. The module $L(\lambda)$ is admissible if and
  only if $q^\rho\lambda(K_{\beta})\in \pm q^{\setZ_{>0}}$ for at
  least one root $\beta\in \Phi^+$.
\end{lemma}
\begin{proof}
  It is easy to show that the Verma module $M(\lambda)$ is not
  admissible. So $q^\rho\lambda(K_{\beta})\in \pm q^{\setZ_{>0}}$ for
  at least one root $\beta\in \Phi^+$ by
  Theorem~\ref{thm:Jantzen_filtration}.  On the other hand suppose
  $q^\rho\lambda(K_{\beta})\in \pm q^{\setZ_{>0}}$ for at least one
  root $\beta\in \Phi^+$. If $q^\rho\lambda(K_{\alpha})\in \pm
  q^{\setZ_{>0}}$ for a simple root $\alpha\in \Pi$ then by easy
  calculations we see that $M(s_\alpha.\lambda)$ is a submodule of
  $M(\lambda)$. If $q^\rho\lambda(K_{\alpha})\not\in \pm
  q^{\setZ_{>0}}$ for both simple roots $\alpha\in \Pi$ then we get
  that $M(s_\beta.\lambda)$ is a submodule by
  Theorem~\ref{thm:Jantzen_filtration}. So in both cases we have a
  submodule $M(s_\beta.\lambda)$ of $M(\lambda)$. Since $L(\lambda)$
  is the unique simple quotient of $M(\lambda)$, $L(\lambda)$ is a
  subquotient of $M(\lambda)/M(s_\beta.\lambda)$. Since
  $M(\lambda)/M(s_\beta.\lambda)$ is admissible 
  we see that $L(\lambda)$ is admissible as well.
\end{proof}

\begin{lemma}
  \label{lemma:14}
  Assume $\mathfrak{g}$ is of type $C_2$
  (i.e. $\mathfrak{g}=\mathfrak{sp}(4)$). Let
  $\Pi=\{\alpha_1,\alpha_2\}$ where $\alpha_1$ is short and $\alpha_2$
  is long. Let $\lambda\in X$ be a non-integral weight. The module
  $L(\lambda)$ is infinite dimensional and admissible if and only if
  $q^\rho\lambda(K_{\alpha_1}),q^\rho\lambda(K_{\alpha_1+\alpha_2})\in
  \pm q^{\setZ_{>0}}$ and
  $\lambda(K_{\alpha_2}),\lambda(K_{2\alpha_1+\alpha_2})\in \pm
  q^{1+2\setZ} (= \pm q_{\alpha_2}^{1/2+\setZ}=\pm
  q_{2\alpha_1+\alpha_2}^{1/2+\setZ})$.
\end{lemma}
\begin{proof}
  Theorem~\ref{thm:Jantzen_filtration} implies that $q^\rho
  \lambda(K_{\beta}) \in q_{\beta}^{\setZ_{>0}}$ for at least two
  $\beta\in \Phi^+$ because otherwise $L(\lambda) =
  M(\lambda)/M(s_\beta.\lambda)$ for some $\beta\in \Phi^+$. But
  $M(\lambda)/M(s_\beta.\lambda)$ is not admissible. Since $\lambda$
  is not integral we know $q^\rho \lambda(K_\alpha)\not \in
  q_{\alpha}^{\setZ}$ for some $\alpha\in \Pi$. Suppose
  $\lambda(K_{\alpha_1})\not \in \pm q_{\alpha_1}^{\setZ}$. We split
  into cases and arrive at a contradiction in both cases: If
  $\lambda(K_{\alpha_2})\not \in \pm q_{\alpha_2}^{\setZ_{>0}}$ then
  by the above $q^\rho \lambda(K_{\alpha_1+\alpha_2})\in \pm
  q_{\alpha_1+\alpha_2}^{\setZ_{>0}}=\pm q^{\setZ_{>0}}$ and $q^\rho
  \lambda(K_{2\alpha_1+\alpha_2}) \in \pm
  q_{2\alpha_1+\alpha_2}^{\setZ_{>0}}=\pm q^{2\setZ_{>0}}$ which
  implies that $q^\rho \lambda(K_{\alpha_1}) = q^\rho \lambda(
  K_{2\alpha_1+\alpha_2} K_{\alpha_1+\alpha_2}\inv) \in \pm
  q^{\setZ}=\pm q_{\alpha_1}^{\setZ}$. A contradiction.

  The other case is $q^\rho \lambda(K_{\alpha_2}) \in \pm
  q_{\alpha_2}^{\setZ_{>0}}=\pm q^{2\setZ_{>0}}$: In this case we get
  $\lambda(K_{\alpha_1+\alpha_2}) \not \in \pm q^{\setZ} = \pm
  q_{\alpha_1+\alpha_2}^{\setZ}$ so the last root,
  $2\alpha_1+\alpha_2$, must satisfy that
  $q^\rho\lambda(K_{2\alpha_1+\alpha_2}) \in \pm
  q_{2\alpha_1+\alpha_2}^{\setZ_{>0}}=\pm q^{2\setZ_{>0}}$. But this
  implies that $\lambda(K_{\alpha_1})^2 =
  \lambda(K_{2\alpha_1+\alpha_2}K_{\alpha_2}\inv)\in \pm q^{2\setZ}$
  which implies that $\lambda(K_{\alpha_1}) \in \pm q^{\setZ}$. A
  contradiction.

  So $\lambda(K_{\alpha_1})\in \pm q^{\setZ}$. Since $\lambda$ is not
  integral we get $\lambda(K_{\alpha_2})\not \in \pm
  q_{\alpha_2}^{\setZ}=\pm q^{2\setZ}$. This implies that
  $\lambda(K_{2\alpha_1+\alpha_2})\not \in \pm q^{2\setZ} = \pm
  q_{2\alpha_1+\alpha_2}^{\setZ}$. Since $q^\rho \lambda(K_{\beta})
  \in \pm q_{\beta}^{\setZ_{>0}}$ for at least two $\beta\in \Phi^+$
  we get $q^{\rho}\lambda(K_{\alpha_1}) \in \pm q^{\setZ_{>0}}$ and
  $q^{\rho}\lambda(K_{\alpha_1+\alpha_2}) \in \pm
  q^{\setZ_{>0}}$. This in turn implies that $\lambda(K_{\alpha_2}) =
  \lambda(K_{\alpha_1+\alpha_2}K_{\alpha_1}\inv) \in \pm
  q^{\setZ}$. Since $\lambda(K_{\alpha_2})\not \in \pm q^{2\setZ}$ we
  get $\lambda(K_{\alpha_2}) \in \pm q^{1+2\setZ}$. Similarly
  $\lambda(K_{2\alpha_1+\alpha_2})=\lambda(K_{\alpha_1+\alpha_2}K_{\alpha_1})\in
  \pm q^{1+2\setZ}$.  So we have shown the only if part.

  Assume $\lambda$ is as required in the lemma. We will show that
  $L(\lambda)$ is admissible. By Theorem~\ref{thm:Jantzen_filtration}
  we see that the composition factors of $M(s_{\alpha_1}.\lambda)$ are
  $L(s_{\alpha_1}.\lambda)$ and
  $L(s_{\alpha_1+\alpha_2}s_{\alpha_1}.\lambda) = M(w_0.\lambda)$ and
  the composition factors of $M(s_{\alpha_1+\alpha_2})$ are
  $L(s_{\alpha_1+\alpha_2})$ and
  $L(s_{\alpha_1}s_{\alpha_1+\alpha_2})=M(w_0.\lambda)$. So
  \begin{equation*}
    \sum_{\substack{\beta\in \Phi^+ \\q^\rho\lambda(K_\beta)\in \pm q_\beta^{\setZ_{>0}}}} \ch M(s_{\beta}.\lambda) = \ch L(s_{\alpha_1}.\lambda)+\ch L(s_{\alpha_1+\alpha_2}.\lambda) + 2 \ch L(w_0.\lambda).
  \end{equation*}
  So the composition factors of the maximal submodule of $M(\lambda)$
  are $L(s_{\alpha_1}.\lambda)$, $L(s_{\alpha_1+\alpha_2}.\lambda)$ and
  $L(w_0.\lambda)$. The worst case scenario being multiplicity one. In
  this case the character of $L(\lambda)$ is
  \begin{align*}
    \ch M(\lambda)& - \ch L(s_{\alpha_1}.\lambda) - \ch
    L(s_{\alpha_1+\alpha_2}) - \ch L(w_0.\lambda)=
    \\
    =& \ch M(\lambda) - \ch M(s_{\alpha_1}.\lambda) - \ch
    M(s_{\alpha_1+\alpha_2}.\lambda) + \ch M(w_0.\lambda)
  \end{align*}
  The character of Verma modules are known and by an easy calculation
  it is seen that this would imply $L(\lambda)$ is admissible (cf. the
  proof of Lemma~7.2 in~\cite{Mathieu}).
\end{proof}

\subsection{Type A, D, E}
\label{sec:class-admiss-modul}
In this section we complete the classification of all simple
admissible highest weight modules when the Dynkin diagram of
$\mathfrak{g}$ is simply laced. In particular we show that
$\mathfrak{g}$ does not admit infinite dimensional simple admissible
modules when $\mathfrak{g}$ is of type D and E. In
Section~\ref{sec:class-admiss-modul-1} we show that the same is the
case when $\mathfrak{g}$ is of type $B$ or $F$. Combining this and
Section~\ref{sec:class-admiss-modul-1} we get that $\mathfrak{g}$
admits infinite dimensional simple admissible modules if and only if
$\mathfrak{g}$ is of type $A$ or $C$. Remember that we restrict our
attention to transcendental $q$ and to non-integral weights because of
Theorem~\ref{thm:integral}.

\begin{defn}
  Let $\lambda:U_q^0 \to \setC$ be a weight. In the Dynkin diagram of
  $\mathfrak{g}$ let any node corresponding to $\alpha \in \Pi\cap
  A(\lambda)$ be written as $\circ$ and every other as
  $\bullet$. e.g. if $\mathfrak{g}=\mathfrak{sl}_3$ and
  $|A(\lambda)|=1$ then the graph corresponding to $\lambda$ would
  look like this:
  \begin{equation*}
    \xymatrix{ \bullet \ar@{-}[r] & \circ}
  \end{equation*}
  
  We call this the colored Dynkin diagram corresponding to $\lambda$.
\end{defn}
In this way we get a 'coloring' of the Dynkin diagram for every
$\lambda$.

\begin{lemma}
  \label{lemma:39}
  Let $\lambda\in X$ be a non-integral weight such that $L(\lambda)$
  is admissible. If the colored Dynkin diagram of $\lambda$ contains
  \begin{equation*}
    \xymatrix{ \stackrel{\alpha'}{\circ} \ar@{-}[r] & \stackrel{\alpha}{\circ}}
  \end{equation*}
  as a subdiagram then $q^\rho\lambda(K_{\alpha+\alpha'})\in \pm
  q_{\alpha'+\alpha}^{\setZ_{>0}}$.
\end{lemma}
\begin{proof}
  Let $v_\lambda$ be a highest weight vector of $L(\lambda)$. Let
  $\mathfrak{s}$ be the Lie algebra $\mathfrak{sl}_3$ with $\alpha$
  and $\alpha'$ as simple roots. Let $U$ be the subalgebra of $U_q$
  generated by $F_{\alpha},F_{\alpha'},K_{\alpha}^{\pm
    1},K_{\alpha'}^{\pm 1},E_{\alpha},E_{\alpha'}$. Then $U \iso
  U_{q_{\alpha}}(\mathfrak{s})$ as algebras and $U v_\lambda$ contains
  the simple highest weight $U_{q_\alpha}(\mathfrak{s})$-module
  $L(\lambda,\mathfrak{s})$ of highest weight $\lambda$ (restricted to
  $U_{q_\alpha}^0(\mathfrak{s})$) as a subquotient. Since $L(\lambda)$
  is admissible so is $U v_\lambda$ hence $L(\lambda,\mathfrak{s})$ is
  admissible. Then Lemma~\ref{lemma:8} implies that $q^\rho
  \lambda(K_{\alpha+\alpha'})\in \pm q_\alpha^{\setZ_{>0}}$.
\end{proof}

\begin{lemma}
  \label{lemma:33}
  Let $\lambda\in X$ be a non-integral weight such that $L(\lambda)$
  is admissible. If the colored Dynkin diagram of $\lambda$ contains
  \begin{equation*}
    \xymatrix{ \stackrel{\alpha'}{\circ} \ar@{-}[r] & \stackrel{\alpha}{\circ} \ar@{-}[r] & \stackrel{\alpha''}{\bullet}}
  \end{equation*}
  as a subdiagram then $L(s_\alpha.\lambda)$ is admissible and the
  colored Dynkin diagram corresponding to $s_\alpha.\lambda$ contains
  \begin{equation*}
    \xymatrix{ \stackrel{\alpha'}{\bullet} \ar@{-}[r] & \stackrel{\alpha}{\circ} \ar@{-}[r] & \stackrel{\alpha''}{\circ}}
  \end{equation*}
  i.e. we can 'move' $\xymatrix{\circ \ar@{-}[r] &\circ}$ and still
  get an admissible module.
\end{lemma}
\begin{proof}
  $L(s_\alpha.\lambda)$ is admissible by Proposition~\ref{prop:9}. It is
  easy to see that $q^\rho s_\alpha.\lambda(K_\alpha)\not \in \pm
  q^{\setZ}$ (follows by Lemma~\ref{lemma:39} since $\lambda$ is
  non-integral), that $q^\rho s_{\alpha}.\lambda(K_{\alpha''})\not \in
  \pm q^{\setZ}$ and that $q^\rho s_\alpha.\lambda(K_{\alpha'}) \in
  \pm q^{\setZ_{>0}}$ (by Lemma~\ref{lemma:39})
\end{proof}

\begin{lemma}
  \label{lemma:41}
  Assume $\mathfrak{g}\neq \mathfrak{sl}_2$. Let $\lambda\in X$ be a
  non-integral weight such that $L(\lambda)$ is admissible.

  If $A(\lambda)=\{\alpha\}$ then $\alpha$ is only connected to one
  other simple root $\alpha'$, $L(s_\alpha.\lambda)$ is admissible and
  the corresponding colored Dynkin diagram of $s_\alpha.\lambda$
  contains
  \begin{equation*}
    \xymatrix{ \stackrel{\alpha}{\circ} \ar@{-}[r] &\stackrel{\alpha'}{\circ} }
  \end{equation*}
  as a subdiagram.

  On the other hand if the colored Dynkin diagram of $\lambda$
  contains
  \begin{equation*}
    \xymatrix{ \stackrel{\alpha}{\circ} \ar@{-}[r] &\stackrel{\alpha'}{\circ} }
  \end{equation*}
  and $\alpha'$ is the only root connected to $\alpha$ then the
  colored Dynkin diagram of $s_\alpha.\lambda$ contains
  \begin{equation*}
    \xymatrix{ \stackrel{\alpha}{\circ} \ar@{-}[r] &\stackrel{\alpha'}{\bullet} }
  \end{equation*}
  as a subdiagram.
\end{lemma}
\begin{proof}
  Since $\alpha\in A(\lambda)$, $L(s_\alpha.\lambda)$ is admissible by
  Proposition~\ref{prop:9}.  First assume $A(\lambda)=\{\alpha\}$. If
  $\alpha$ is connected to two distinct roots $\alpha'$ and $\alpha''$
  then it is easily seen that $\alpha',\alpha''\in
  A(s_\alpha.\lambda)$ contradicting the fact that
  $A(s_\alpha.\lambda)$ is connected (Lemma~\ref{lemma:34}). It is
  easily seen that $q^\rho s_\alpha.\lambda(K_{\alpha})\not \in \pm
  q^{\setZ_{>0}}$ (since $\lambda$ is non integral) and $q^\rho
  s_\alpha.\lambda(K_{\alpha'}) \not \in q^{\setZ_{>0}}$.
  
  On the other hand if $A(\lambda)=\{\alpha,\alpha'\}$ then $q^\rho
  s_\alpha.\lambda(K_{\alpha'}) = q^\rho \lambda(K_{\alpha+\alpha'})
  \in \pm q^{\setZ_{>0}}$ by Lemma~\ref{lemma:39}.
\end{proof}

Now we can eliminate the types that are not type $A$ by the following
theorem:

\begin{thm}
  \label{thm:simply_laced_exists_only_type_A}
  Assume $\mathfrak{g}$ is a simple Lie algebra of simply laced
  type. If there exists an infinite dimensional admissible simple
  module then $\mathfrak{g}$ is of type $A$.
\end{thm}
\begin{proof}
  Suppose there exists an infinite dimensional admissible simple
  module then by Theorem~\ref{thm:EXT_contains_highest_weight} there
  exists a $\lambda\in X$ such that $L(\lambda)$ is an infinite
  admissible simple highest weight module.  By
  Theorem~\ref{thm:integral} and the classification in~\cite{Mathieu}
  there exists no highest weight simple admissible modules with
  integral weights unless $\mathfrak{g}$ is of type $A$. We need to
  show the same for non-integral weights.

  If the Dynkin diagram is simply laced and not of type $A$ then the
  Dynkin diagram contains
  \begin{equation*}
    \xymatrix{ & \stackrel{\alpha}{\bullet} \ar@{-}[d] & \\ \stackrel{\alpha'}{\bullet} \ar@{-}[r] & \stackrel{\gamma}{\bullet}  \ar@{-}[r] &\stackrel{\alpha''}{\bullet}}
  \end{equation*}
  as a subdiagram.

  By Lemma~\ref{lemma:41} we can assume without loss of generality
  that $|A(\lambda)|=2$ and by Lemma~\ref{lemma:33} we can assume that
  the colored Dynkin diagram corresponding to $\lambda$ contains the
  following:
  \begin{equation*}
    \xymatrix{ & \stackrel{\alpha}{\bullet} \ar@{-}[d] & \\ \stackrel{\alpha'}{\circ} \ar@{-}[r] & \stackrel{\gamma}{\circ}  \ar@{-}[r] &\stackrel{\alpha''}{\bullet}}
  \end{equation*}
  But then $L(s_\gamma.\lambda)$ is admissible as well by
  Proposition~\ref{prop:9} and the colored Dynkin diagram for
  $s_\gamma.\lambda$ contains
  \begin{equation*}
    \xymatrix{ & \stackrel{\alpha}{\circ} \ar@{-}[d] & \\ \stackrel{\alpha'}{\bullet} \ar@{-}[r] & \stackrel{\gamma}{\circ}  \ar@{-}[r] &\stackrel{\alpha''}{\circ}}
  \end{equation*}
  contradicting the fact that $A(\lambda)$ is connected.
\end{proof}

Combining all the above results we get
\begin{thm}
  \label{thm:Classification_of_adm_modules_simply_laced}
  Let $\mathfrak{g}=\mathfrak{sl}_{n+1}$, $n\geq 2$ with simple roots
  $\alpha_1,\dots,\alpha_n$ such that $(\alpha_i|\alpha_{i+1})=-1$,
  $i=1,\dots,n$. Let $\lambda \in X$ be a non-integral weight.

  $L(\lambda)$ is admissible if and only if the colored Dynkin diagram
  of $\lambda$ is of one of the following types:
  
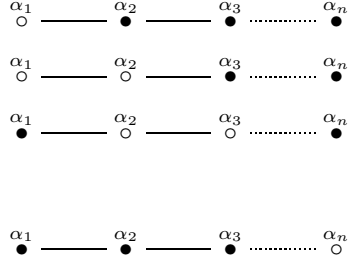
\begin{align*}
    \xymatrix{ \stackrel{\alpha_1}{\circ} \ar@{-}[r] &
      \stackrel{\alpha_2}{\bullet} \ar@{-}[r] &
      \stackrel{\alpha_3}{\bullet} \ar@{.}[r] &
      \stackrel{\alpha_n}{\bullet} }
    \\
    \xymatrix{ \stackrel{\alpha_1}{\circ} \ar@{-}[r] &
      \stackrel{\alpha_2}{\circ} \ar@{-}[r] &
      \stackrel{\alpha_3}{\bullet} \ar@{.}[r] &
      \stackrel{\alpha_n}{\bullet} }
    \\
    \xymatrix{ \stackrel{\alpha_1}{\bullet} \ar@{-}[r] &
      \stackrel{\alpha_2}{\circ} \ar@{-}[r] &
      \stackrel{\alpha_3}{\circ} \ar@{.}[r] &
      \stackrel{\alpha_n}{\bullet} }
    \\
    &\vdots
    \\
    \xymatrix{ \stackrel{\alpha_1}{\bullet} \ar@{-}[r] &
      \stackrel{\alpha_2}{\bullet} \ar@{-}[r] &
      \stackrel{\alpha_3}{\bullet} \ar@{.}[r] &
      \stackrel{\alpha_n}{\circ} }
  \end{align*}
\end{thm}
\begin{proof}
  By the above results these are the only possibilites. To show that
  $L(\lambda)$ is admissible when the colored Dynkin diagram is of the
  above form use the fact that by Lemma~\ref{lemma:33} and
  Lemma~\ref{lemma:41} we can assume $\lambda$ has colored Dynkin
  diagram as follows:
  \begin{equation*}
    \xymatrix{ \stackrel{\alpha_1}{\circ} \ar@{-}[r] &
      \stackrel{\alpha_2}{\bullet} \ar@{-}[r] &
      \stackrel{\alpha_3}{\bullet} \ar@{.}[r] &
      \stackrel{\alpha_n}{\bullet} }.
  \end{equation*}
  Let $\beta_i=\alpha_1+\alpha_2+\dots + \alpha_i$, $i=1,\dots,n$. We
  see easily that $T_{L(\lambda)}=-\{\beta_1,\beta_2,\dots,\beta_n\}$
  and $F_L = \Phi^+ \cup \Phi_{\{\alpha_2,\dots,\alpha_n\}}$. Let
  $\mathfrak{l}$, $\mathfrak{u}$, $\mathfrak{p}$ etc. be defined as in
  Section~2 of~\cite{DHP1}. By~\cite[Theorem~2.23]{DHP1}
  $N:=L(\lambda)^{\mathfrak{u}}$ is a simple finite dimensional
  $U_q(\mathfrak{l})$-module and $L(\lambda)$ is the unique simple
  quotient of $\mathcal{M}(N) = U_q \tensor_{U_q(\mathfrak{p})}
  N$. Since the vectors $\beta_1,\dots,\beta_n$ are linearly
  independent $\mathcal{M}(N)$ is admissible. This implies that
  $L(\lambda)$ is admissible since it is a quotient of
  $\mathcal{M}(N)$.
\end{proof}

We can now make Corollary~\ref{cor:1} more specific in type A:
\begin{cor}
  \label{cor:2}
  Let $\mathfrak{g}=\mathfrak{sl}_{n+1}$, $n\geq 2$ with simple roots
  $\alpha_1,\dots,\alpha_n$ such that $(\alpha_i|\alpha_{i+1})=-1$,
  $i=1,\dots,n$. Let $\beta_j=\alpha_1+\cdots+\alpha_j$, $j=1,\dots,n$
  and $\Sigma=\{\beta_1,\dots,\beta_n\}$. Let
  $F_{\beta_j}=T_{s_1}\cdots T_{s_{j-1}}(F_{\alpha_{j}})$ and let
  $F_{\Sigma}=\{q^a F_{\beta_1}^{a_1}\cdots F_{\beta_n}^{a_n}|a_i\in
  \setN,a\in \setZ\}$ be the corresponding Ore subset. Then $\Sigma$
  is a set of commuting roots that is a basis of $Q$ with
  corresponding Ore subset $F_\Sigma$.
  
  Let $\beta_j'=\alpha_n+\cdots+\alpha_{n-j}$, $j=1,\dots,n$ and
  $\Sigma=\{\beta_1',\dots,\beta_n'\}$. Let
  $F'_{\beta_j'}=T_{s_n}\cdots T_{s_{n-j+1}}(F_{\alpha_{n-j}})$ and
  let $F_{\Sigma'}=\{q^a (F'_{\beta_1'})^{a_1}\cdots
  (F'_{\beta_n'})^{a_n}|a_i\in \setN,a\in \setZ\}$ be the
  corresponding Ore subset. Then $\Sigma'$ is a set of commuting roots
  that is a basis of $Q$ with corresponding Ore subset $F_{\Sigma'}$.

  Let $L$ be a simple torsion free module then one of the two
  following claims hold
  \begin{itemize}
  \item There exists a $\lambda\in X$ with $\lambda(K_{\alpha_1})\not
    \in \pm q^{\setN}$, $\lambda(K_{\alpha_i})\in \pm q^{\setN}$,
    $i=2,\dots,n$ and $\mathbf{b}\in (\setC^*)^n$ such that
    \begin{align*}
      L \iso \phi_{F_\Sigma,\mathbf{b}}.L(\lambda)_{F_\Sigma}.
    \end{align*}
  \item There exists a $\lambda\in X$ with $\lambda(K_{\alpha_n})\not
    \in \pm q^{\setN}$, $\lambda(K_{\alpha_i})\in \pm q^{\setN}$,
    $i=1,\dots,n-1$ and $\mathbf{b}\in (\setC^*)^n$ such that
    \begin{align*}
      L \iso \phi_{F_{\Sigma'},\mathbf{b}}.L(\lambda)_{F_{\Sigma'}}.
    \end{align*}
  \end{itemize}
\end{cor}
\begin{proof}
  By Theorem~\ref{thm:EXT_contains_highest_weight}
  $\mathcal{EXT}(L)\iso \mathcal{EXT}(L(\lambda'))$ for some
  $\lambda'\in X$. If $\lambda'$ is non-integral then by
  Theorem~\ref{thm:Classification_of_adm_modules_simply_laced},
  Lemma~\ref{lemma:33}, Lemma~\ref{lemma:30} and
  Proposition~\ref{prop:15} there exists a $\lambda$ such that
  $\lambda(K_{\alpha_1})\not \in \pm q^{\setN}$,
  $\lambda(K_{\alpha_i})\in\pm q^{\setN}$, $i=2,\dots,n$ and such that
  $\mathcal{EXT}(L(\lambda'))\iso \mathcal{EXT}(L(\lambda))$. By
  Lemma~\ref{lemma:20} we can choose $\Sigma$ as the commuting set of
  roots that is used in the definition of $\mathcal{EXT}(L(\lambda))$.

  If $\lambda'$ is integral we see by Theorem~\ref{thm:integral},
  Lemma~\ref{lemma:30}, Proposition~\ref{prop:15} and the
  classification in~\cite[Section~8]{Mathieu} that
  $\mathcal{EXT}(L(\lambda'))\iso \mathcal{EXT}(L(\lambda))$ for a
  $\lambda$ such that $A(\lambda)=\{\alpha_1\}$ or
  $A(\lambda)=\{\alpha_n\}$
  (cf.~e.g.~\cite[Proposition~8.5]{Mathieu}).

  Now the result follows just like in the proof of
  Corollary~\ref{cor:1}.
\end{proof}

In Section~\ref{sec:type-a_n-calc} we determine all $\mathbf{b}\in
(\setC^*)^n$ such that
$\phi_{F_\Sigma,\mathbf{b}}.L(\lambda)_{F_\Sigma}$ is torsion free
with $\Sigma$ as above in Corollary~\ref{cor:2} and $\lambda$ such
that $\lambda(K_{\alpha_1})\not \in \pm q^{\setN}$,
$\lambda(K_{\alpha_i})\in\pm q^{\setN}$, $i=2,\dots,n$. By symmetry of
the Dynkin diagram and Corollary~\ref{cor:2} this classifies all
simple torsion free modules.

\subsection{Quantum Shale-Weil representation}
\label{sec:quantum-shale-weyl}

In this section we assume $\mathfrak{g}$ is of type $C_n$. Let
$\alpha_1,\dots,\alpha_n$ be the simple roots such that $\alpha_i$ is
connected to $\alpha_{i+1}$ and $\alpha_1$ is long. We will describe a
specific admissible module $V$ and show that $V=L(\omega^+)\oplus
L(\omega^-)$ for some weights $\omega^{\pm}$ with the purpose of
classifying the admissible simple highest weight modules, see
Theorem~\ref{thm:existence_adm_mod_type_C}. Let
$V=\setC[X_1,\dots,X_n]$. We describe an action of the simple root
vectors on $V$: For $i\in \{2,\dots,n\}$
\begin{align*}
  E_{\alpha_1} X_1^{a_1}X_2^{a_2}\cdots X_n^{a_n} =&
  -\frac{[a_1][a_1-1]}{[2]} X_1^{a_1-2}X_2^{a_2}\cdots X_n^{a_n}
  \\
  F_{\alpha_1} X_1^{a_1}X_2^{a_2}\cdots X_n^{a_n} =& \frac{1}{[2]}
  X_1^{a_1+2}X_2^{a_2}\cdots X_n^{a_n}
  \\
  E_{\alpha_i} X_1^{a_1}\cdots X_n^{a_n} =& [a_i] X_1^{a_1}\cdots
  X_{i-1}^{a_{i-1}+1}X_i^{a_i-1}\cdots X_n^{a_n}
  \\
  F_{\alpha_i} X_1^{a_1}\cdots X_n^{a_n} =& [a_{i-1}] X_1^{a_1}\cdots
  X_{i-1}^{a_{i-1}-1}X_i^{a_i+1}\cdots X_n^{a_n}
  \\
  K_{\alpha_1}^{\pm 1} X_1^{a_1}X_2^{a_2}\cdots X_n^{a_n} =& q^{\mp
    (2a_1+1)} X_1^{a_1}X_2^{a_2}\cdots X_n^{a_n}
  \\
  K_{\alpha_i}^{\pm 1} X_1^{a_1}X_2^{a_2}\cdots X_n^{a_n} =& q^{\pm
    (a_{i-1}-a_i)}X_1^{a_1}X_2^{a_2}\cdots X_n^{a_n}.
\end{align*}

We check that this is an action of $U_q$ by checking the generating
relations. These are tedious and kind of long calculations but just
direct calculations. We refer to the generating relations as (R1) to
(R6) like in~\cite[Section~4.3]{Jantzen}.

(R1) is clear. (R2) and (R3): Let $j\in\{1,\dots,n\}$
\begin{align*}
  K_{\alpha_j} E_{\alpha_1} X_1^{a_1}\cdots X_n^{a_n} =&
  \begin{cases}
    -q^{-2a_1+3} \frac{[a_1][a_1-1]}{[2]} X_1^{a_1-2}X_2^{a_2}\cdots
    X_n^{a_n} &\text{ if } j=1
    \\
    -q^{a_{1}-2-a_2} \frac{[a_1][a_1-1]}{[2]}
    X_1^{a_1-2}X_2^{a_2}\cdots X_n^{a_n} &\text{ if } j=2
    \\
    -q^{a_{j-1}-a_j} \frac{[a_1][a_1-1]}{[2]}
    X_1^{a_1-2}X_2^{a_2}\cdots X_n^{a_n} &\text{ if } j>2
  \end{cases}
  \\
  =& q^{(\alpha_1|\alpha_j)}E_{\alpha_1} K_{\alpha_j}
  X_1^{a_1}X_2^{a_2}\cdots X_n^{a_n}.
\end{align*}
Similar for $K_{\alpha_j}F_{\alpha_1}$. For $i\in\{2,\dots,n\}$
\begin{align*}
  K_{\alpha_j} E_{\alpha_i} X_1^{a_1}\cdots X_n^{a_n} =&
  \begin{cases}
    q^{a_{j-1}-a_j}[a_i] X_1^{a_1}\cdots
    X_{i-1}^{a_{i-1}+1}X_i^{a_i-1}\cdots X_n^{a_n} &\text{ if }
    |j-i|>1
    \\
    q^{a_{j-1}-a_j-1}[a_i] X_1^{a_1}\cdots
    X_{i-1}^{a_{i-1}+1}X_i^{a_i-1}\cdots X_n^{a_n} &\text{ if } j=i-1
    \\
    q^{a_{j-1}+1-a_j+1}[a_i] X_1^{a_1}\cdots
    X_{i-1}^{a_{i-1}+1}X_i^{a_i-1}\cdots X_n^{a_n} &\text{ if } j=i
    \\
    q^{a_{j-1}-1-a_j}[a_i] X_1^{a_1}\cdots
    X_{i-1}^{a_{i-1}+1}X_i^{a_i-1}\cdots X_n^{a_n} &\text{ if } j=i+1
  \end{cases}
  \\
  =& q^{(\alpha_i|\alpha_j)}E_{\alpha_1} K_{\alpha_j}
  X_1^{a_1}X_2^{a_2}\cdots X_n^{a_n}.
\end{align*}
Similarly for $K_{\alpha_j}F_{\alpha_i}$.

(R4):
\begin{align*}
  [E_{\alpha_1},F_{\alpha_1}] X_1^{a_1}X_2^{a_2}\cdots X_n^{a_n} =&
  E_{\alpha_1} \frac{1}{[2]} X_1^{a_1+2}X_2^{a_2}\cdots X_n^{a_n} +
  F_{\alpha_1} \frac{[a_1][a_1-1]}{[2]} X_1^{a_1-2}X_2^{a_2}\cdots
  X_n^{a_n}
  \\
  =& \left(-\frac{[a_1+2][a_1+1]}{[2][2]} +
    \frac{[a_1][a_1-1]}{[2][2]}\right) X_1^{a_1}\cdots X_n^{a_n}
  \\
  =& \frac{q^{-2a_1-1}-q^{2a_1+1}}{q^2-q^{-2}} X_1^{a_1}\cdots
  X_n^{a_n}
  \\
  =& \frac{K_{\alpha_1}-K_{\alpha_1}\inv}{q^2-q^{-2}} X_1^{a_1}\cdots
  X_n^{a_n}.
\end{align*}

\begin{align*}
  [E_{\alpha_1},F_{\alpha_2}] X_1^{a_1}\cdots X_n^{a_n} =&
  [a_1]E_{\alpha_1} X_1^{a_1-1}X_2^{a_2+1}\cdots X_n^{a_n} +
  \frac{[a_1][a_1-1]}{[2]} F_{\alpha_2} X_1^{a_1-2}X_2^{a_2}\cdots
  X_n^{a_n}
  \\
  =& -\frac{[a_1][a_1-1][a_1-2]}{[2]} X_1^{a_1-3}X_2^{a_2+1}\cdots
  X_n^{a_n}
  \\
  &+ \frac{[a_1][a_1-1][a_1-2]}{[2]} X_1^{a_1-3}X_2^{a_2+1}\cdots
  X_n^{a_n}
  \\
  =&0.
\end{align*}

For $i>2$ clearly $[E_{\alpha_1},F_{\alpha_i}]X_1^{a_1}\cdots
X_n^{a_n}=0$. For $i,j\in\{2,\dots,n\}$: If $|i-j|>1$ clearly
$[E_{\alpha_i},F_{\alpha_j}]X_1^{a_1}\cdots X_n^{a_n}=0$.
\begin{align*}
  [E_{\alpha_i},F_{\alpha_{i+1}}] X_1^{a_n}\cdots X_n^{a_n} =&
  [a_{i}]E_{\alpha_i} X_1^{a_1}\cdots X_{i}^{a_{i}-1}X_{i+1}^{a_{i+1}
    +1}\cdots X_n^{a_n}
  \\
  &- [a_i]F_{\alpha_{i+1}}X_1^{a_1}\cdots
  X_{i-1}^{a_{i-1}+1}X_i^{a_i-1}\cdots X_n^{a_n}
  \\
  =& [a_{i}][a_i-1] X_1^{a_1}\cdots
  X_{i-1}^{a_{i-1}+1}X_{i}^{a_{i}-2}X_{i+1}^{a_{i+1} +1}\cdots
  X_n^{a_n}
  \\
  &- [a_{i}][a_i-1]X_1^{a_1}\cdots
  X_{i-1}^{a_{i-1}+1}X_i^{a_i-2}X_{i+1}^{a_{i+1}+1}\cdots X_n^{a_n}
  \\
  =& 0.
\end{align*}

\begin{align*}
  [E_{\alpha_2},F_{\alpha_1}] X_1^{a_1}\cdots X_n^{a_n} =&
  E_{\alpha_2} \frac{1}{[2]} X_1^{a_1+2}X_2^{a_2}\cdots X_n^{a_n} -
  [a_2]F_{\alpha_1} X_1^{a_1+1}X_2^{a_2-1}\cdots X_n^{a_n}
  \\
  =& \frac{[a_2]}{[2]} X_1^{a_1+3}X_2^{a_2-1}\cdots X_n^{a_n} -
  \frac{[a_2]}{[2]}X_1^{a_1+3}X_2^{a_2-1}\cdots X_n^{a_n}
  \\
  =&0.
\end{align*}

For $i>2$:
\begin{align*}
  [E_{\alpha_i},F_{\alpha_{i-1}}] X_1^{a_n}\cdots X_n^{a_n} =&
  [a_{i-2}]E_{\alpha_i} X_1^{a_1}\cdots
  X_{i-2}^{a_{i-2}-1}X_{i-1}^{a_{i-1} +1}\cdots X_n^{a_n}
  \\
  &- [a_i]F_{\alpha_{i-1}}X_1^{a_1}\cdots
  X_{i-1}^{a_{i-1}+1}X_i^{a_i-1}\cdots X_n^{a_n}
  \\
  =&[a_{i-2}][a_i] X_1^{a_1}\cdots X_{i-2}^{a_{i-2}-1}X_{i-1}^{a_{i-1}
    +2}X_i^{a_i-1}\cdots X_n^{a_n}
  \\
  &- [a_i][a_{i-2}]F_{\alpha_{i-1}}X_1^{a_1}\cdots
  X_{i-2}^{a_{i-2}-1}X_{i-1}^{a_{i-1}+2}X_i^{a_i-1}\cdots X_n^{a_n}
  \\
  =& 0.
\end{align*}

For $i>1$:
\begin{align*}
  [E_{\alpha_i},F_{\alpha_i}]X_1^{a_1}\cdots X_n^{a_n} =&
  [a_{i-1}]E_{\alpha_i}X_1^{a_1}\cdots
  X_{i-1}^{a_{i-1}+1}X_i^{a_i-1}\cdots X_n^{a_n}
  \\
  &- [a_i]F_{\alpha_i}X_1^{a_1}\cdots
  X_{i-1}^{a_{i-1}-1}X_i^{a_{i}+1}\cdots X_n^{a_n}
  \\
  =& ([a_{i-1}][a_i-1]-[a_i][a_{i-1}-1])X_1^{a_1}\cdots
  X_{i-1}^{a_{i-1}}X_i^{a_i}\cdots X_n^{a_n}
  \\
  =& [a_{i-1}-a_i] X_1^{a_1}\cdots X_n^{a_n}
  \\
  =& \frac{K_{\alpha_i}-K_{\alpha_i}\inv}{q-q\inv}X_1^{a_1}\cdots
  X_n^{a_n}.
\end{align*}

Finally we have the relations (R5) and (R6): Clearly
$[E_{\alpha_i},E_{\alpha_j}]X_1^{a_1}\cdots X_n^{a_n}=0$ and
$[F_{\alpha_i},F_{\alpha_j}]X_1^{a_1}\cdots X_n^{a_n}=0$ when
$|j-i|>1$.

\begin{align*}
  (E_{\alpha_2}^3 E_{\alpha_1}& - [3]
  E_{\alpha_2}^2E_{\alpha_1}E_{\alpha_2} +
  [3]E_{\alpha_2}E_{\alpha_1}E_{\alpha_2}^2 -
  E_{\alpha_1}E_{\alpha_2}^3)X_1^{a_1}\cdots X_n^{a_n}
  \\
  =& \frac{1}{[2]}\Big(-[a_1][a_1-1][a_2][a_2-1][a_2-2]
  \\
  &+[3][a_1+1][a_1][a_2][a_2-1][a_2-2]
  \\
  &-[3][a_1+2][a_1+1][a_2][a_2-1][a_2-2]
  \\
  &+[a_1+3][a_1+2][a_2][a_2-1][a_2-2]\Big)X_1^{a_1+1}X_2^{a_2-3}\cdots
  X_n^{a_n}
  \\
  =& \frac{[a_2][a_2-1][a_2-2]}{[2]}\Big(
  -[a_1][a_1-1]+[3][a_1+1][a_1]
  \\
  &-[3][a_1+2][a_1+1]+[a_1+3][a_1+2]\Big) X_1^{a_1+1}X_2^{a_2-3}\cdots
  X_n^{a_n}
  \\
  =&0.
\end{align*}

\begin{align*}
  (E_{\alpha_1}^2 E_{\alpha_2}-[2]_{\alpha_1}
  &E_{\alpha_1}E_{\alpha_2}E_{\alpha_1} +
  E_{\alpha_2}E_{\alpha_1}^2)X_1^{a_1}\cdots X_n^{a_n}
  \\
  =& \frac{[a_2]}{[2][2]}\big( [a_1+1][a_1][a_1-1][a_1-2]
  \\
  &- [2]_{\alpha_1} [a_1][a_1-1][a_1-1][a_1-2]
  \\
  &+ [a_1][a_1-1][a_1-2][a_1-3]\Big) X_1^{a_1+3}X_2^{a_2-1}\cdots
  X_n^{a_n}
  \\
  =& \frac{[a_2][a_1][a_1-1][a_1-2]}{[2][2]} \Big(
  [a_1+1]-[2]_{\alpha_1}[a_1-1]
  \\
  &+[a_1-3]\Big) X_1^{a_1+3}X_2^{a_2-1}\cdots X_n^{a_n}
  \\
  =& 0.
\end{align*}

For $i>1$:
\begin{align*}
  (E_{\alpha_i}^2
  E_{\alpha_{i+1}}-&[2]E_{\alpha_i}E_{\alpha_{i+1}}E_{\alpha_i}+E_{\alpha_{i+1}}E_{\alpha_i}^2)X_1^{a_1}\cdots
  X_n^{a_n}
  \\
  =& [a_{i+1}][a_i]([a_i+1]-[2][a_i]+[a_i-1])X_1^{a_1}\cdots
  X_{i-1}^{a_{i-1}+2}X_i^{a_i-1}X_{i+1}^{a_{i+1}-1}\cdots X_n^{a_n}
  \\
  =& 0.
\end{align*}

\begin{align*}
  (E_{\alpha_{i+1}}^2
  E_{\alpha_{i}}-&[2]E_{\alpha_{i+1}}E_{\alpha_{i}}E_{\alpha_{i+1}}+E_{\alpha_{i}}E_{\alpha_{i+1}}^2)X_1^{a_1}\cdots
  X_n^{a_n}
  \\
  =& [a_{i+1}][a_{i+1}-1]([a_i]-[2][a_i+1]+[a_i+2])X_1^{a_1}\cdots
  X_{i-1}^{a_{i-1}+1}X_i^{a_i+1}X_{i+1}^{a_{i+1}-2}\cdots X_n^{a_n}
  \\
  =& 0.
\end{align*}

\begin{align*}
  (F_{\alpha_1}^2 F_{\alpha_2}-[2]_{\alpha_1}
  &F_{\alpha_1}F_{\alpha_2}F_{\alpha_1} +
  F_{\alpha_2}F_{\alpha_1}^2)X_1^{a_1}\cdots X_n^{a_n}
  \\
  =&
  \frac{1}{[2][2]}([a_1]-[2]_{\alpha_1}[a_1+2]+[a+4])X_1^{a_1+3}X_2^{a_2+1}\cdots
  X_n^{a_n}
  \\
  =&0.
\end{align*}

\begin{align*}
  (F_{\alpha_2}^3 F_{\alpha_1}& - [3]
  F_{\alpha_2}^2F_{\alpha_1}F_{\alpha_2} +
  [3]F_{\alpha_2}F_{\alpha_1}F_{\alpha_2}^2 -
  F_{\alpha_1}F_{\alpha_2}^3)X_1^{a_1}\cdots X_n^{a_n}
  \\
  =& \frac{1}{[2]}\Big( [a_1+2][a_1+1][a_1] - [3][a_1][a_1+1][a_1]
  \\
  &+ [3][a_1][a_1-1][a_1]
  \\
  &-[a_1][a_1-1][a_1-2] \Big) X_1^{a_1-1}X_2^{a_2+3}\cdots X_n^{a_n}
  \\
  =& \frac{[a_1]}{[2]}\Big( [a_1+2][a_1+1] - [3][a_1+1][a_1]
  \\
  &+ [3][a_1][a_1-1]
  \\
  &-[a_1-1][a_1-2] \Big) X_1^{a_1-1}X_2^{a_2+3}\cdots X_n^{a_n}
  \\
  =& 0.
\end{align*}

For $i>1$:
\begin{align*}
  (F_{\alpha_i}^2
  F_{\alpha_{i+1}}-&[2]F_{\alpha_i}F_{\alpha_{i+1}}F_{\alpha_i}+F_{\alpha_{i+1}}F_{\alpha_i}^2)X_1^{a_1}\cdots
  X_n^{a_n}
  \\
  =& [a_{i-1}][a_{i-1}-1]([a_i]-[2][a_i+1]+[a_i+2])X_1^{a_1}\cdots
  X_{i-1}^{a_{i-1}-2}X_i^{a_i+1}X_{i+1}^{a_{i+1}+1}\cdots X_n^{a_n}
  \\
  =& 0.
\end{align*}

\begin{align*}
  (F_{\alpha_{i+1}}^2
  F_{\alpha_{i}}-&[2]F_{\alpha_{i+1}}F_{\alpha_{i}}F_{\alpha_{i+1}}+F_{\alpha_{i}}F_{\alpha_{i+1}}^2)X_1^{a_1}\cdots
  X_n^{a_n}
  \\
  =& [a_{i-1}][a_{i}]([a_i+1]-[2][a_i]+[a_i-1])X_1^{a_1}\cdots
  X_{i-1}^{a_{i-1}-1}X_i^{a_i-1}X_{i+1}^{a_{i+1}+2}\cdots X_n^{a_n}
  \\
  =& 0.
\end{align*}

So we have shown that $V$ is a $U_q(\mathfrak{g})$-module. Note that
$V$ is admissible of degree $1$ and $V=V^{even}\oplus V^{odd}$ where
$V^{even}$ are even degree polynomials and $V^{odd}$ are odd degree
polynomials. Furthermore we see that $V^{even}=L(\omega^+)$ and
$V^{odd}=L(\omega^-)$ where $\omega^{\pm}$ are the weights defined by
$\omega^+(K_{\alpha_1})=q\inv$, $\omega^+(K_{\alpha_i})=1$, $i>1$ and
$\omega^-(K_{\alpha_1})=q^{-3}$, $\omega^-(K_{\alpha_2})=q\inv$,
$\omega^-(K_{\alpha_i})=1$, $i>2$. $V^{even}$ is generated by $1$ and
$V^{odd}$ is generated by $X_1$. We will use the fact that
$L(\omega^+)$ is admissible in
Theorem~\ref{thm:existence_adm_mod_type_C} in the next section.

\subsection{Type B, C, F}
\label{sec:class-admiss-modul-1}
In this section we classify the simple highest weight admissible
modules when $\mathfrak{g}$ is of type $B$, $C$ or $F$. Remember that
we have assumed that $q$ is transcendental.

\begin{thm}
  \label{thm:clas_of_adm_modules_type_B_C_F}
  Let $\mathfrak{g}$ be a simple Lie algebra not of type
  $G_2$. Suppose there exists an infinite dimensional admissible
  simple $U_q(\mathfrak{g})$-module. Then $\mathfrak{g}$ is of type
  $A$ or $C$.
\end{thm}
\begin{proof}
  If $\mathfrak{g}$ is simply laced then
  Theorem~\ref{thm:Classification_of_adm_modules_simply_laced} gives
  that $\mathfrak{g}$ is of type $A$.  So assume $\mathfrak{g}$ is not
  of simply laced type. Theorem~\ref{thm:integral} and the
  classification in the classical case tells us that no admissible
  infinite dimensional simple highest weight modules exists with
  integral weights when $\mathfrak{g}$ is not simply laced
  (cf.~\cite[Lemma~9.1]{Mathieu}).
  
  We have assumed that $\mathfrak{g}$ is not of type $G_2$ so the
  remaining non-simply laced types are $B$, $C$ or $F$. We will show
  that the Dynkin diagram of $\mathfrak{g}$ can't contain the
  subdiagram
  \begin{equation*}
    \xymatrix{ \stackrel{\alpha_1}{\bullet} \ar@{<=}[r] & \stackrel{\alpha_2}{\bullet} \ar@{-}[r] & \stackrel{\alpha_3}{\bullet}}.
  \end{equation*}

  Assume the Dynkin diagram contains the above as a subdiagram. If
  there exists a simple admissible infinite dimensional module $L$
  then there exists a non-integral $\lambda\in X$ such that
  $L(\lambda)$ is infinite dimensional and admissible
  (Theorem~\ref{thm:EXT_contains_highest_weight}). Let $\lambda\in X$
  be a non-integral weight such that $L(\lambda)$ is admissible. Then
  by Lemma~\ref{lemma:14}, $q^\rho\lambda(K_{\alpha_1}) \in \pm
  q_{\alpha_1}^{\setZ}= \pm q^{\setZ}$. By Lemma~\ref{lemma:33} and
  Lemma~\ref{lemma:41} we can assume without loss of generality that
  the colored Dynkin diagram of $\lambda$ is of the form
  \begin{equation*}
    \xymatrix{ \stackrel{\alpha_1}{\bullet} \ar@{<=}[r] & \stackrel{\alpha_2}{\circ} \ar@{-}[r] & \stackrel{\alpha_3}{\circ}}.
  \end{equation*}
  Let $\mathfrak{s}$ be the simple rank $3$ Lie algebra of type
  $B_3$. Let $U$ be the subalgebra of $U_q$ generated by
  $E_{\alpha_i},F_{\alpha_i},K_{\alpha_i}^{\pm 1}$, $i=1,2,3$. Then
  $U\iso U_q(\mathfrak{s})$. Let $Q_{\mathfrak{s}}:= \setZ
  \{\alpha_1,\alpha_2,\alpha_3\}\subset Q$. Let $v_\lambda$ be a
  highest weight vector of $L(\lambda)$. Then $Uv_\lambda$ contains
  the simple highest weight $U_q(\mathfrak{s})$-module
  $L(\lambda,\mathfrak{s})$ of highest weight $\lambda$ (restricted to
  $U_q^0(\mathfrak{s})$) as a subquotient. Since $L(\lambda)$ is
  admissible so is $L(\lambda,\mathfrak{s})$.
  
  Like in the proof of Lemma~\ref{lemma:34} we get a contradiction if
  we can show that $T_{L(\lambda,\mathfrak{s})}\cap
  T_{{^{s_{\alpha_2}}}L(s_{\alpha_2}.\lambda,\mathfrak{s})}$ generates
  $Q_{\mathfrak{s}}$. It is easily seen that
  $\{-\alpha_1-\alpha_2,-\alpha_3,-2\alpha_1-\alpha_2\} \subset
  T_{L(\lambda,\mathfrak{s})}\cap
  T_{{^{s_{\alpha_2}}}L(s_{\alpha_2}.\lambda,\mathfrak{s})}$, so
  $T_{L(\lambda,\mathfrak{s})}\cap
  T_{{^{s_{\alpha_2}}}L(s_{\alpha_2}.\lambda,\mathfrak{s})}$ generates
  $Q_{\mathfrak{s}}$. So $C(L(\lambda,\mathfrak{s})) \cap
  C({^{s_{\alpha_2}}}L(s_{\alpha_2}.\lambda,\mathfrak{s}))$ generates
  $Q_{\mathfrak{s}}$. Therefore $C(L(\lambda,\mathfrak{s}))-
  C({^{s_{\alpha_2}}}L(s_{\alpha_2}.\lambda,\mathfrak{s}))=Q_{\mathfrak{s}}$.
  The weights of $L(\lambda,\mathfrak{s})$ and
  ${^{s_{\alpha_2}}}L(s_{\alpha_2}.\lambda,\mathfrak{s})$ are
  contained in $q^{Q_{\mathfrak{s}}} \lambda$ so a weight in the
  essential support of $L(\lambda,\mathfrak{s})$
  (resp. ${^{s_{\alpha_2}}}L(s_{\alpha_2}.\lambda,\mathfrak{s})$) is
  of the form $q^{\mu_1}\lambda$ (resp. $q^{\mu_2}\lambda$) for some
  $\mu_1,\mu_2 \in Q_{\mathfrak{s}}$. By the above
  $q^{C(L(\lambda,\mathfrak{s}))+\mu_1}\lambda \cap
  q^{C({^{s_{\alpha_2}}}L(s_{\alpha_2}.\lambda,\mathfrak{s}))+\mu_2}\lambda\neq
  \emptyset$. Since $q^{C(L(\lambda,\mathfrak{s}))+\mu_1}\lambda
  \subset \Suppess(L(\lambda))$ and
  $q^{C({^{s_{\alpha_2}}}L(s_{\alpha_2}.\lambda,\mathfrak{s}))+\mu_2}\lambda
  \subset
  \Suppess({^{s_{\alpha_2}}}L(s_{\alpha_2}.\lambda,\mathfrak{s}))$ we
  have proved that
  $\Suppess({^{s_{\alpha_2}}}L(s_{\alpha_2}.\lambda,\mathfrak{s}))\cap
  \Suppess(L(\lambda,\mathfrak{s}))\neq \emptyset$. By
  Proposition~\ref{prop:9} $L(\lambda,\mathfrak{s})$ and
  ${^{s_{\alpha_2}}}L(s_{\alpha_2}.\lambda,\mathfrak{s})$ are
  subquotients of $L(\lambda,\mathfrak{s})_{F_{\alpha_2}}$. Let
  $\nu\in
  \Suppess({^{s_{\alpha_2}}}L(s_{\alpha_2}.\lambda,\mathfrak{s}))\cap
  \Suppess(L(\lambda,\mathfrak{s}))$. Then by Lemma~\ref{lemma:13}
  $L(\lambda,\mathfrak{s})_{\nu} \iso
  (L(\lambda,\mathfrak{s})_{F_{\alpha_2}})_\nu \iso
  ({^{s_{\alpha_2}}}L(s_{\alpha_2}.\lambda,\mathfrak{s}))_\nu$ so by
  Theorem~\ref{thm:Lemire} $L(\lambda,\mathfrak{s})\iso
  {^{s_{\alpha_2}}}L(s_{\alpha_2}.\lambda,\mathfrak{s})$. This is a
  contradiction by looking at weights of the modules.
\end{proof}

\begin{thm}
  \label{thm:existence_adm_mod_type_C}
  Let $\mathfrak{g}$ be a simple Lie algebra of type $C_n$
  (i.e. $\mathfrak{g}=\mathfrak{sp}(2n)$). Let
  $\alpha_1,\dots,\alpha_n$ be the simple roots such that $\alpha_i$
  is connected to $\alpha_{i+1}$ and $\alpha_1$ is long -- i.e. the
  Dynkin diagram of $C_n$ is
  \begin{equation*}
    \xymatrix{ \stackrel{\alpha_1}{\bullet} \ar@{=>}[r] & \stackrel{\alpha_2}{\bullet} \ar@{.}[r] & \stackrel{\alpha_{n-1}}{\bullet} \ar@{-}[r] & \stackrel{\alpha_n}{\bullet}}.
  \end{equation*}
  Let $\lambda\in X$. $L(\lambda)$ is infinite dimensional and
  admissible if and only if
  \begin{itemize}
  \item $\lambda(K_{\alpha_i})\in \pm q^{\setN}$ for $1< i \leq n$
  \item $\lambda(K_{\alpha_1})\in \pm q_{\alpha_1}^{1/2 + \setZ}=\pm
    q^{1+2\setZ}$
  \item $\lambda(K_{\alpha_{1}+\alpha_2})\in \pm q^{\setZ_{\geq -2}}$
  \end{itemize}
  or equivalently $q^\rho\lambda(K_\beta)\in \pm q^{\setZ_{>0}}$ for
  every short root $\beta\in \Phi^+$ and $\lambda(K_{\beta'})\in \pm
  q^{1+2\setZ}$ for every long root $\beta'\in \Phi^+$.
\end{thm}
\begin{proof}
  Assume $\lambda(K_{\alpha_i})\not \in \pm q^{\setN}$ for some
  $i>1$. Then by Lemma~\ref{lemma:33} there exists a $\lambda'$ such
  that $L(\lambda')$ is admissible and such that
  $\lambda'(K_{\alpha_{2}})\not \in q^{\setN}$. Let $\mathfrak{s}$ be
  the Lie algebra $\mathfrak{sp}(4)$ with simple roots $\alpha_{2}$
  and $\alpha_1$. Let $U$ be the subalgebra of $U_q$ generated by
  $F_{\alpha_{1}},F_{\alpha_{2}},K_{\alpha_1},K_{\alpha_{2}},E_{\alpha_1},E_{\alpha_{2}}$. Then
  $U \iso U_{q}(\mathfrak{s})$ as algebras and $U v_{\lambda'}$
  contains the simple highest weight $U_{q}(\mathfrak{s})$-module
  $L(\lambda',\mathfrak{s})$ of highest weight $\lambda'$ (restricted
  to $U_{q_\alpha}^0(\mathfrak{s})$) as a subquotient. Since
  $L(\lambda')$ is admissible so is $U v_{\lambda'}$ hence
  $L(\lambda',\mathfrak{s})$ is admissible. So
  $\lambda'(K_{\alpha_{2}})\in \pm q^{\setN}$ by
  Lemma~\ref{lemma:14}. A contradiction. So we have proven that
  $\lambda(K_{\alpha_i})\in \pm q^{\setN}$ for $1< i \leq n$ is a
  neccesary condition. We get also from Lemma~\ref{lemma:14} that
  $\lambda(K_{\alpha_1})\in q^{1+2\setZ}$ and
  $q^3\lambda(K_{\alpha_{1}+\alpha_2})=
  q^\rho\lambda(K_{\alpha_{1}+\alpha_2}) \in \pm q^{\setZ_{>0}}$ which
  shows that the two other conditions are neccesary.

  Now assume we have a weight $\lambda\in X$ that satisfies the
  above. So $\lambda(K_{\alpha_1})=q^{-1+r}$ for some $r\in
  2\setZ$. We can assume $r\in \setN$ by Lemma~\ref{lemma:30} and
  Proposition~\ref{prop:15} (if $r<0$ replace $\lambda$ with
  $s_1.\lambda$, $L(\lambda)$ is admissible if and only if
  $L(s_1.\lambda)$ is).  We have $\lambda= \omega^+ \lambda_0$ for
  some dominant integral weight $\lambda_0$ and $L(\lambda)$ is a
  subquotient of $L(\omega^+)\tensor L(\lambda_0)$. Since
  $L(\omega^+)$ is admissible and $L(\lambda_0)$ is finite dimensional
  $L(\omega^+)\tensor L(\lambda_0)$ is admissible and since
  $L(\lambda)$ is a subquotient of $L(\omega^+)\tensor L(\lambda_0)$,
  $L(\lambda)$ is admissible as well.
\end{proof}

\begin{cor}
  \label{cor:3}
  Let $\mathfrak{g}$ be a simple Lie algebra of type $C_n$
  (i.e. $\mathfrak{g}=\mathfrak{sp}(2n)$). Let
  $\alpha_1,\dots,\alpha_n$ be the simple roots such that $\alpha_i$
  is connected to $\alpha_{i+1}$ and $\alpha_1$ is long.
  
  Let $\beta_j=\alpha_1+\cdots+\alpha_j$, $j=1,\dots,n$ and
  $\Sigma=\{\beta_1,\dots,\beta_n\}$. Let $F_{\beta_j}=T_{s_1}\cdots
  T_{s_{j-1}}(F_{\alpha_{j}})$ and let $F_{\Sigma}=\{q^a
  F_{\beta_1}^{a_1}\cdots F_{\beta_n}^{a_n}|a_i\in \setN,a\in \setZ\}$
  be the corresponding Ore subset. Then $\Sigma$ is a set of commuting
  roots that is a basis of $Q$ with corresponding Ore subset
  $F_\Sigma$.
  
  Let $L$ be a simple torsion free module. Then there exists a
  $\lambda\in X$ with $\lambda(K_{\beta})\in \pm q^\setN$ for all
  short $\beta\in \Phi^+$ and $\lambda(K_{\gamma}) \in \pm
  q^{1+2\setZ}$ for all long $\gamma \in \Phi^+$ and a $\mathbf{b}\in
  (\setC^*)^n$ such that
  \begin{align*}
    L\iso \phi_{F_\Sigma,\mathbf{b}}.L(\lambda)_{F_\Sigma}
  \end{align*}
\end{cor}
\begin{proof}
  By Theorem~\ref{thm:EXT_contains_highest_weight} there exists a
  $\lambda \in X$ such that $\mathcal{EXT}(L)\iso
  \mathcal{EXT}(L(\lambda))$. By Proposition~\ref{prop:15}
  $L(\lambda)$ is admissible and by
  Theorem~\ref{thm:existence_adm_mod_type_C} $\lambda$ is as described
  in the statement of the corollary. Now the result follows just like
  in the proof of Corollary~\ref{cor:1}.
\end{proof}

In Section~\ref{sec:type-c_n-calc} we determine all $\mathbf{b}\in
(\setC^*)^n$ such that
$\phi_{F_\Sigma,\mathbf{b}}.L(\lambda)_{F_\Sigma}$ is torsion free
(with $\Sigma$ and $\lambda$ as above in Corollary~\ref{cor:3}). By
the corollary this classifies all simple torsion free modules for type
$C$.

\section{Classification of simple torsion free modules. Type A.}
\label{sec:type-a_n-calc}
In this section we assume $\mathfrak{g}=\mathfrak{sl}_{n+1}$ with
$n\geq 2$. Let $\Pi=\{\alpha_1,\dots,\alpha_n\}$ denote the simple
roots such that $(\alpha_i|\alpha_{i+1})=-1$, $i=1,\dots,n-1$. Set
$\beta_j = s_1\cdots s_{j-1}(\alpha_j)=\alpha_1+\dots+\alpha_j$, then
$\Sigma=\{\beta_1,\dots,\beta_n\}$ is a set of commuting roots with
corresponding root vectors $F_{\beta_j} = T_{s_1}\cdots
T_{s_{j-1}}(F_{\alpha_j})$. We will show some commutation formulas and
use these to calculate $\phi_{F_\Sigma,\mathbf{b}}$ on all simple root
vectors. This will allow us to determine exactly for which
$\mathbf{b}\in (\setC^*)^n$,
$\phi_{F_\Sigma,\mathbf{b}}.L(\lambda)_{F_\Sigma}$ is torsion free,
see Theorem~\ref{thm:clas_of_b_such_that_twist_is_torsion_free}.

Choose a reduced expression of $w_0$ starting with $s_1\cdots s_n$ and
define roots $\gamma_1,\dots,\gamma_N$ and root vectors
$F_{\gamma_1},\dots,F_{\gamma_N}$ from this expression. Note that
$F_{\beta_i}=F_{\gamma_i}$ for $i=1,\dots, n$.

\begin{prop}
  \label{prop:26}
  Let $i\in \{2,\dots,n\}$ and $j\in \{1,\dots,n\}$.
  \begin{equation*}
    [F_{\alpha_i},F_{\beta_j}]_q =
    \begin{cases}
      F_{\beta_i}, &\text{ if } j = i-1
      \\
      0, &\text{ otherwise}
    \end{cases}
  \end{equation*}
  and
  \begin{equation*}
    [E_{\alpha_i},F_{\beta_j}] =
    \begin{cases}
      F_{\beta_{i-1}} K_{\alpha_i}\inv, &\text{ if } j = i
      \\
      0, &\text{ otherwise.}
    \end{cases}
  \end{equation*}
\end{prop}
\begin{proof}
  We will show the proposition for the $F$'s first and then for the
  $E$'s.
  
  Assume first that $j<i-1$. Then clearly
  $[F_{\alpha_i},F_{\beta_j}]_q = [F_{\alpha_i},F_{\beta_j}]=0$ since
  $\alpha_i$ is not connected to any of the simple roots
  $\alpha_1,\dots,\alpha_j$ appearing in $\beta_j$.
  
  Then assume $j\geq i$. We must have $\alpha_i = \gamma_k$ for some
  $k>n$ since $\{\gamma_1,\dots,\gamma_N\}=\Phi^+$. By
  Theorem~\ref{thm:DP} $[F_{\alpha_i},F_{\beta_j}]_q$ is a linear
  combination of monomials of the form
  $F_{\gamma_{j+1}}^{a_{j+1}}\cdots F_{\gamma_{k-1}}^{a_{k-1}}$. For a
  monomial of this form to appear with nonzero coefficient we must
  have
  \begin{equation*}
    \sum_{h=j+1}^{k-1} a_h \gamma_h = \alpha_i + \beta_j = \alpha_1+\dots + \alpha_{i-1}+2\alpha_i +\alpha_{i+1}+\dots \alpha_j.
  \end{equation*}
  For this to be possible one of the positive roots $\gamma_s$,
  $j<s<k$ must be equal to $\alpha_1+\alpha_2+\dots+\alpha_m$ for some
  $m\leq j$ but $\alpha_1+\alpha_2+\dots+\alpha_m=\gamma_m$ by
  construction and $m\leq j<s$ so $m\neq s$. We conclude that this is
  not possible.

  Finally we investigate the case when $j=i-1$. We have
  \begin{align*}
    [F_{\alpha_i},F_{\beta_{i-1}}]_q =& [T_{s_1}\cdots
    T_{s_{i-2}}(F_{\alpha_i}),T_{s_1}\cdots
    T_{s_{i-2}}(F_{\alpha_{i-1}})]_q
    \\
    =& T_{s_1}\cdots T_{s_{i-2}}
    \left([F_{\alpha_i},F_{\alpha_{i-1}}]_q\right)
    \\
    =& T_{s_1}\cdots T_{s_{i-2}} T_{s_{i-1}}(F_{\alpha_i})
    \\
    =& F_{\beta_i}.
  \end{align*}

  For the $E$'s: Assume first $j<i$: Since $F_{\beta_j}$ is a
  polynomial in $F_{\alpha_1},\dots,F_{\alpha_j}$, $E_{\alpha_i}$
  commutes with $F_{\beta_j}$ when $j<i$.

  Assume then $j=i$: We have by the above
  \begin{equation*}
    F_{\beta_i} = [F_{\alpha_i},F_{\beta_{i-1}}]_q
  \end{equation*}
  so
  \begin{align*}
    [E_{\alpha_i},F_{\beta_i}] =&
    [E_{\alpha_i},(F_{\alpha_i}F_{\beta_{i-1}}-q^{-(\beta_{i-1}|\alpha_i)}F_{\beta_{i-1}}F_{\alpha_i})]
    \\
    =& [E_{\alpha_i},F_{\alpha_i}] F_{\beta_{i-1}} - q F_{\beta_{i-1}}
    [E_{\alpha_i},F_{\alpha_i}]
    \\
    =& \frac{ K_{\alpha_i}-K_{\alpha_i}\inv}{q-q\inv} F_{\beta_{i-1}}
    - q F_{\beta_{i-1}} \frac{K_{\alpha_i}-K_{\alpha_i}\inv}{q-q\inv}
    \\
    =& F_{\beta_{i-1}} \frac{ q K_{\alpha_i} - q\inv K_{\alpha_i}\inv
      - q K_{\alpha_i} + q K_{\alpha_i}\inv}{q-q\inv}
    \\
    =& F_{\beta_{i-1}}K_{\alpha_{i}}\inv.
  \end{align*}
  
  Finally assume $j>i$: Observe first that we have
  \begin{equation*}
    T_{s_{i+1}}\cdots T_{s_{j-1}}F_{\alpha_j} = \sum_{s=1}^m u_s F_{\alpha_{i+1}} u_s'
  \end{equation*}
  for some $m\in \setN$ and some $u_s,u_s'$ that are polynomials in
  $F_{\alpha_{i+2}},\dots F_{\alpha_{j}}$. Note that
  $T_{s_i}(u_s)=u_s$ and $T_{s_i}(u_s')=u_s'$ for all $s$ since
  $\alpha_i$ is not connected to any of the simple roots
  $\alpha_{i+2},\dots \alpha_j$. So
  \begin{align*}
    T_{s_i}T_{s_{i+1}}\cdots T_{s_{j-1}}F_{\alpha_j} =& T_{s_i}\left(
      \sum_{s=1}^m u_s F_{\alpha_{i+1}} u_s'\right)
    \\
    =& \sum_{s=1}^m u_s T_{s_i}(F_{\alpha_{i+1}}) u_s'
    \\
    =& \sum_{s=1}^m u_s
    (F_{\alpha_{i+1}}F_{\alpha_i}-qF_{\alpha_i}F_{\alpha_{i+1}}) u_s'
    \\
    =& \sum_{s=1}^m u_s F_{\alpha_{i+1}} u_s' F_{\alpha_i} - q
    F_{\alpha_i}\sum_{s=1}^m u_s F_{\alpha_{i+1}} u_s'
    \\
    =& T_{s_{i+1}}\cdots T_{s_{j-1}}(F_{\alpha_j}) F_{\alpha_i} - q
    F_{\alpha_i} T_{s_{i+1}}\cdots T_{s_{j-1}}(F_{\alpha_j}).
  \end{align*}
  Thus we see that
  \begin{align*}
    F_{\beta_j} =& T_{s_1}\dots T_{s_i}\cdots
    T_{s_{j-1}}(F_{\alpha_j})
    \\
    =&T_{s_{i+1}}\cdots T_{s_{j-1}}(F_{\alpha_j}) T_{s_{1}}\cdots
    T_{s_{i-1}}(F_{\alpha_i}) - qT_{s_{1}}\cdots T_{s_{i-1}}(
    F_{\alpha_i}) T_{s_{i+1}}\cdots T_{s_{j-1}}(F_{\alpha_j})
    \\
    =& T_{s_{i+1}}\cdots T_{s_{j-1}}(F_{\alpha_j}) F_{\beta_i} - q
    F_{\beta_i} T_{s_{i+1}}\cdots T_{s_{j-1}}(F_{\alpha_j})
  \end{align*}
  and therefore
  \begin{align*}
    [E_{\alpha_i},F_{\beta_j}] =& T_{s_{i+1}}\cdots
    T_{s_{j-1}}(F_{\alpha_j}) [E_{\alpha_i},F_{\beta_i}] - q
    [E_{\alpha_i},F_{\beta_i}]T_{s_{i+1}}\cdots
    T_{s_{j-1}}(F_{\alpha_j})
    \\
    =& T_{s_{i+1}}\cdots T_{s_{j-1}}(F_{\alpha_j})
    F_{\beta_{i-1}}K_{\alpha_i}\inv- q F_{\beta_{i-1}}K_{\alpha_i}\inv
    T_{s_{i+1}}\cdots T_{s_{j-1}}(F_{\alpha_j}) 
    \\
    =& F_{\beta_{i-1}} T_{s_{i+1}}\cdots T_{s_{j-1}}(F_{\alpha_j})
    K_{\alpha_i}\inv - F_{\beta_{i-1}} T_{s_{i+1}}\cdots
    T_{s_{j-1}}(F_{\alpha_j}) K_{\alpha_i}\inv
    \\
    =& 0.
  \end{align*}
\end{proof}

\begin{prop}
  \label{prop:27}
  Let $i\in \{2,\dots,n\}$. Let $a\in \setZ_{>0}$. Then
  \begin{equation*}
    [F_{\alpha_i},F_{\beta_{i-1}}^a]_q = [a] F_{\beta_{i-1}}^{a-1}F_{\beta_i}
  \end{equation*}
  and for $b\in \setC^*$
  \begin{equation*}
    \phi_{F_{\beta_{i-1}},b}(F_{\alpha_i}) = b F_{\alpha_i}+ \frac{b-b\inv}{q-q\inv} F_{\beta_{i-1}}\inv F_{\beta_i}.
  \end{equation*}
\end{prop}
\begin{proof}
  The first claim is proved by induction over $a$. $a=1$ is shown
  in Proposition~\ref{prop:26}. The induction step:
  \begin{align*}
    F_{\alpha_i}F_{\beta_{i-1}}^{a+1} =& \left( q^a F_{\beta_{i-1}}^a
      F_{\alpha_i} + [a] F_{\beta_{i-1}}^{a-1}F_{\beta_i}\right)
    F_{\beta_{i-1}}
    \\
    =& q^{a+1}F_{\beta_{i-1}}^{a+1}F_{\alpha_i} +q^a F_{\beta_{i-1}}^a
    F_{\beta_i} + q\inv [a] F_{\beta_{i-1}}^a F_{\beta_i}
    \\
    =& q^{a+1}F_{\beta_{i-1}}^{a+1}F_{\alpha_i} +
    [a+1]F_{\beta_{i-1}}^a F_{\beta_i}.
  \end{align*}
  So we have proved the first claim. We get then for $a\in \setZ_{>0}$
  \begin{equation*}
    \phi_{F_{\beta_{i-1}},q^a}(F_{\alpha_i}) = F_{\beta_{i-1}}^{-a} F_{\alpha_i} F_{\beta_{i-1}}^a = q^a F_{\alpha_i} + \frac{q^a-q^{-a}}{q-q\inv} F_{\beta_{i-1}}\inv F_{\beta_i}.
  \end{equation*}
  Using the fact that $\phi_{F_{\beta_{i-1}},b}(F_{\alpha_i})$ is
  Laurent polynomial in $b$ we get the second claim of the
  proposition.
\end{proof}

\begin{prop}
  \label{prop:28}
  Let $i\in \{2,\dots,n\}$. Let $a\in \setZ_{>0}$. Then
  \begin{equation*}
    [E_{\alpha_i},F_{\beta_i}^a] = q^{a-1}[a] F_{\beta_i}^{a-1}F_{\beta_{i-1}}K_{\alpha_i}\inv
  \end{equation*}
  and for $b\in \setC^*$
  \begin{equation*}
    \phi_{F_{\beta_{i}},b}(E_{\alpha_i}) = E_{\alpha_i}+ q\inv b \frac{b-b\inv}{q-q\inv} F_{\beta_i}\inv F_{\beta_{i-1}} K_{\alpha_i}\inv.
  \end{equation*}
\end{prop}
\begin{proof}
  The first claim is proved by induction over $a$. $a=1$ is shown
  in Proposition~\ref{prop:26}. The induction step:
  \begin{align*}
    E_{\alpha_i} F_{\beta_i}^{a+1} =& \left( F_{\beta_i}^a
      E_{\alpha_i}+ q^{a-1} [a] F_{\beta_i}^{a-1}F_{\beta_{i-1}}
      K_{\alpha_i}\inv \right) F_{\beta_i}
    \\
    =& F_{\beta_i}^{a+1} E_{\alpha_i} + F_{\beta_i}^a F_{\beta_{i-1}}
    K_{\alpha_i}\inv + q^{a+1} [a] F_{\beta_i}^a
    F_{\beta_{i-1}}K_{\alpha_i}\inv
    \\
    =& F_{\beta_i}^{a+1} E_{\alpha_i} + q^{a} (q^{-a} +
    q[a])F_{\beta_i}^a F_{\beta_{i-1}}K_{\alpha_i}\inv
    \\
    =& F_{\beta_i}^{a+1} E_{\alpha_i} + q^{a} [a+1]F_{\beta_i}^a
    F_{\beta_{i-1}}K_{\alpha_i}\inv.
  \end{align*}
  This proves the first claim. We get then for $a\in \setZ_{>0}$
  \begin{equation*}
    \phi_{F_{\beta_i},q^a}(E_{\alpha_i}) = F_{\beta_i}^{-a} E_{\alpha_i} F_{\beta_i}^a = E_{\alpha_i} + q\inv q^a  \frac{q^a - q^{-a}}{q-q\inv} F_{\beta_i}\inv F_{\beta_{i-1}}K_{\alpha_i}\inv.
  \end{equation*}
  Using the fact that $\phi_{F_{\beta_{i}},b}(E_{\alpha_i})$ is
  Laurent polynomial in $b$ we get the second claim of the
  proposition.
\end{proof}

In our classification we don't need to calculate
$\phi_{F_\Sigma,\mathbf{b}}(E_{\alpha_1})$ but for completeness we
show what it is in this case in Proposition~\ref{prop:30}. To do this
we need the following proposition:
\begin{prop}
  \label{prop:29}
  Let $j\in \{2,\dots,n\}$. Then
  \begin{equation*}
    [E_{\alpha_1},F_{\beta_j}] = -q T_{s_2}\cdots T_{s_{j-1}}(F_{\alpha_j})K_{\alpha_1},
  \end{equation*}
  for $a\in \setZ_{>0}$:
  \begin{equation*}
    [E_{\alpha_1},F_{\beta_j}^a] = -q^{2-a}[a]F_{\beta_j}^{a-1} T_{s_2}\cdots T_{s_{j-1}}(F_{\alpha_j})K_{\alpha_1}
  \end{equation*}
  and for $b\in \setC^*$:
  \begin{equation*}
    \phi_{F_{\beta_j},b}(E_{\alpha_1}) = E_{\alpha_1} - q^2 b \frac{b-b\inv}{q-q\inv} F_{\beta_j}\inv T_{s_2}\cdots T_{s_{j-1}}(F_{\alpha_j}) K_{\alpha_1}.
  \end{equation*}
\end{prop}
\begin{proof}
  Like in the proof of Proposition~\ref{prop:26} we see that
  \begin{equation*}
    T_{s_{2}}\cdots T_{s_{j-1}}F_{\alpha_j} = \sum_{s=1}^m u_s F_{\alpha_{2}} u_s'
  \end{equation*}
  for some $m\in \setN$ and some $u_s,u_s'$ that are polynomials in
  $F_{\alpha_{3}},\dots F_{\alpha_{j}}$. Note that $T_{s_1}(u_s)=u_s$
  and $T_{s_1}(u_s')=u_s'$ for all $s$ since $\alpha_1$ is not
  connected to any of the simple roots $\alpha_{3},\dots \alpha_j$. So
  \begin{align*}
    T_{s_1}T_{s_{2}}\cdots T_{s_{j-1}}F_{\alpha_j} =& T_{s_1}\left(
      \sum_{s=1}^m u_s F_{\alpha_{2}} u_s'\right)
    \\
    =& \sum_{s=1}^m u_s T_{s_1}(F_{\alpha_{2}}) u_s'
    \\
    =& \sum_{s=1}^m u_s (F_{\alpha_{2}}F_{\alpha_1}-q
    F_{\alpha_1}F_{\alpha_{2}}) u_s'
    \\
    =& \sum_{s=1}^m u_s F_{\alpha_{2}} u_s' F_{\alpha_1} - q
    F_{\alpha_1}\sum_{s=1}^m u_s F_{\alpha_{2}} u_s'
    \\
    =& T_{s_{2}}\cdots T_{s_{j-1}}(F_{\alpha_j}) F_{\alpha_1} - q
    F_{\alpha_1} T_{s_{2}}\cdots T_{s_{j-1}}(F_{\alpha_j}).
  \end{align*}
  Thus
  \begin{align*}
    [E_{\alpha_1},F_{\beta_j}] =& T_{s_2}\cdots
    T_{s_{j-1}}(F_{\alpha_j}) [E_{\alpha_1},F_{\alpha_1}] - q
    [E_{\alpha_1},F_{\alpha_1}] T_{s_2}\cdots
    T_{s_{j-1}}(F_{\alpha_j})
    \\
    =& T_{s_2}\cdots T_{s_{j-1}}(F_{\alpha_j})
    \frac{K_{\alpha_1}-K_{\alpha_1}\inv}{q-q\inv} - q
    \frac{K_{\alpha_1}-K_{\alpha_1}\inv}{q-q\inv} T_{s_2}\cdots
    T_{s_{j-1}}(F_{\alpha_j})
    \\
    =& T_{s_2}\cdots T_{s_{j-1}}(F_{\alpha_j}) \frac{
      K_{\alpha_1}-K_{\alpha_1}\inv - q^2 K_{\alpha_1} +
      K_{\alpha_1}\inv}{q-q\inv}
    \\
    =& -q T_{s_2}\cdots T_{s_{j-1}}(F_{\alpha_j}) K_{\alpha_1}.
  \end{align*}
  Note that $T_{s_2}\cdots T_{s_{j-1}}(F_{\alpha_j})$ is a polynomial
  in $F_{\alpha_2},\dots,F_{\alpha_j}$. By Proposition~\ref{prop:26}
  $[F_{\alpha_i},F_{\beta_j}]_q = [F_{\alpha_i},F_{\beta_j}] = 0$ for
  $1<i<j$ and $[F_{\alpha_j},F_{\beta_j}]_q =
  F_{\alpha_j}F_{\beta_j}-q\inv F_{\beta_j}F_{\alpha_j}=0$ so
  \begin{align*}
    T_{s_2}\cdots T_{s_{j-1}}(F_{\alpha_j}) F_{\beta_j} - q\inv
    F_{\beta_j}T_{s_2}\cdots T_{s_{j-1}}(F_{\alpha_j}) =&
    [T_{s_2}\cdots T_{s_{j-1}}(F_{\alpha_j}),F_{\beta_j}]_q
    \\
    =& 0.
  \end{align*}
  
  The second claim is by induction on $a$:
  \begin{align*}
    E_{\alpha_1} F_{\beta_j}^{a+1} =& \left( F_{\beta_j}^a
      E_{\alpha_1} - q^{2-a}[a] F_{\beta_j}^{a-1} T_{s_2}\cdots
      T_{s_{j-1}}(F_{\alpha_j}) K_{\alpha_1} \right) F_{\beta_j}
    \\
    =& F_{\beta_j}^{a+1} E_{\alpha_1} - q F_{\beta_j}^a T_{s_2}\cdots
    T_{s_{j-1}}(F_{\alpha_j}) K_{\alpha_1}
    \\
    &- q^{-a} [a] F_{\beta_j}^a T_{s_2}\cdots
    T_{s_{j-1}}(F_{\alpha_j}) K_{\alpha_1}
    \\
    =& F_{\beta_j}^{a+1} E_{\alpha_1} - q^{1-a}\left( q^{a} + q\inv
      [a]\right)F_{\beta_j}^a T_{s_2}\cdots T_{s_{j-1}}(F_{\alpha_j})
    K_{\alpha_1}
    \\
    =& F_{\beta_j}^{a+1} E_{\alpha_1} - q^{1-a}[a+1] F_{\beta_j}^a
    T_{s_2}\cdots T_{s_{j-1}}(F_{\alpha_j}) K_{\alpha_1}.
  \end{align*}

  So we get for $a\in \setZ_{>0}$:
  \begin{equation*}
    \phi_{F_{\beta_j},q^a}(E_{\alpha_1}) = F_{\beta_j}^{-a} E_{\alpha_1} F_{\beta_j}^a = E_{\alpha_1} - q^2 q^{-a} \frac{q^a - q^{-a}}{q-q\inv} T_{s_2}\cdots
    T_{s_{j-1}}(F_{\alpha_j}) K_{\alpha_1}.
  \end{equation*}
  Using the fact that $\phi_{F_{\beta_{j}},b}(E_{\alpha_1})$ is
  Laurent polynomial in $b$ we get the third claim of the proposition.
\end{proof}

We can combine the above propositions in the following proposition
\begin{prop}
  \label{prop:30}
  Let $i\in \{2,\dots,n\}$. For $\mathbf{b}=(b_1,\dots,b_n)\in
  (\setC^*)^n$
  \begin{align*}
    \phi_{F_\Sigma,\mathbf{b}}(F_{\alpha_{i}}) =& b_i^{-1}b_{i+1}\inv
    \cdots b_n\inv \phi_{F_{\beta_{i-1},b_{i-1}}}(F_{\alpha_i})
    \\
    =& b_i^{-1}b_{i+1}\inv \cdots b_n\inv (b_{i-1}F_{\alpha_i}+
    \frac{b_{i-1}-b_{i-1}\inv}{q-q\inv} F_{\beta_{i-1}}\inv
    F_{\beta_i})
    \\
    \phi_{F_\Sigma,\mathbf{b}}(E_{\alpha_i}) =&
    \phi_{F_{\beta_{i},b_{i}}}(E_{\alpha_i}) = E_{\alpha_i}+ q\inv b_i
    \frac{b_i-b_i\inv}{q-q\inv} F_{\beta_i}\inv F_{\beta_{i-1}}
    K_{\alpha_i}\inv.
  \end{align*}
  Furthermore
  \begin{equation*}
    \phi_{F_\Sigma,\mathbf{b}}(F_{\alpha_1}) = b_2\cdots b_n
    F_{\alpha_1}
  \end{equation*}
  and
  \begin{align*}
    \phi_{F_{\Sigma,\mathbf{b}}}(E_{\alpha_1}) =& E_{\alpha_1} - q^2
    \sum_{j=2}^n b_jb_{j+1}\inv\cdots b_n\inv
    \frac{b_j-b_j\inv}{q-q\inv}F_{\beta_j}\inv T_{s_2}\cdots
    T_{s_{j-1}}(F_{\alpha_j}) K_{\alpha_1}
    \\
    &+ b_2\inv \cdots b_n\inv F_{\beta_1}\inv
    \frac{(b_1-b_1\inv)(qb_1\inv\cdots b_n\inv K_{\alpha_1} - q\inv
      b_1\cdots b_n K_{\alpha_1}\inv)}{(q-q\inv)^2}.
  \end{align*}
\end{prop}
\begin{proof}
  The first two equations follow from Proposition~\ref{prop:26},
  Proposition~\ref{prop:27} and Proposition~\ref{prop:28}. The third
  follows because $F_{\alpha_1}=F_{\beta_1}$ q-commutes with all the
  other root vectors $F_{\beta_2},\dots,F_{\beta_n}$ (see also the
  discussion before Definition~\ref{def:twist_by_weight}). For the
  last equation we use Proposition~\ref{prop:29}:
  \begin{align*}
    \phi_{F_\Sigma,\mathbf{b}}(E_{\alpha_1})=&
    \phi_{F_{\beta_n},b_n}\circ \cdots \circ
    \phi_{F_{\beta_1},b_1}(E_{\alpha_1})
    \\
    =& \phi_{F_{\beta_n},b_n}\circ \cdots \circ
    \phi_{F_{\beta_2},b_2}\left( E_{\alpha_1} - F_{\beta_1}\inv
      \frac{(b_1-b_1\inv)(qb_1\inv K_{\alpha_1}-q\inv b_1
        K_{\alpha_1}\inv)}{(q-q\inv)^2}\right)
    \\
    =& \phi_{F_{\beta_n},b_n}\circ \cdots \circ
    \phi_{F_{\beta_3},b_3}( E_{\alpha_1} - q^2
    b_2\frac{b_2-b_2\inv}{q-q\inv} F_{\beta_2}\inv F_{\alpha_2}
    K_{\alpha_1}
    \\
    &- b_2\inv F_{\beta_1}\inv \frac{(b_1-b_1\inv)(qb_1\inv b_2\inv
      K_{\alpha_1}-q\inv b_1b_2 K_{\alpha_1}\inv)}{(q-q\inv)^2})
    \\
    &\vdots
    \\
    =& E_{\alpha_1} - q^2 \sum_{j=2}^n b_jb_{j+1}\inv\cdots b_n\inv
    \frac{b_j-b_j\inv}{q-q\inv}F_{\beta_j}\inv T_{s_2}\cdots
    T_{s_{j-1}}(F_{\alpha_j}) K_{\alpha_1}
    \\
    &- b_2\inv \cdots b_n\inv F_{\beta_1}\inv
    \frac{(b_1-b_1\inv)(qb_1\inv\cdots b_n\inv K_{\alpha_1} - q\inv
      b_1\cdots b_n K_{\alpha_1}\inv)}{(q-q\inv)^2}
  \end{align*}
\end{proof}

\begin{prop}
  \label{prop:4}
  Let $\lambda$ be a weight such that $\lambda(K_{\alpha_i})\in \pm
  q^\setN$ for $i=2,\dots,n$ and $\lambda(K_{\alpha_1})\not \in \pm
  q^\setN$. Let $\mathbf{b}=(b_1,\dots,b_n)\in (\setC^*)^n$. Let
  $i\in\{2,\dots,n\}$. Then $E_{\alpha_i}$ acts injectively on the
  $U_q$-module $\phi_{F_\Sigma,\mathbf{b}}.L(\lambda)_{F_\Sigma}$ if and
  only if $b_i \not \in \pm q^{\setZ}$ and $F_{\alpha_i}$ acts
  injectively on $\phi_{F_\Sigma,\mathbf{b}}.L(\lambda)_{F_\Sigma}$ if
  and only if $b_{i-1}\not \in \pm q^{\setZ}$.
\end{prop}
\begin{proof}
  By Proposition~\ref{prop:25} and Corollary~\ref{cor:6} a root vector
  acts injectively on the $U_q$-module
  \begin{equation*}
    \phi_{F_\Sigma,(b_1,\dots,b_n)}.L(\lambda)_{F_\Sigma}
  \end{equation*}
  if and only if it acts injectively on
  \begin{equation*}
    \phi_{F_\Sigma,(\epsilon_1
      q^{i_1}b_1,\dots,\epsilon_n q^{i_n}b_n)}.L(\lambda)_{F_\Sigma}
  \end{equation*}
  for any $i_1,\dots,i_n\in\setZ$ and $\epsilon_1,\dots,\epsilon_n\in
  \{\pm 1\}$.
  
  Assume there exists a $0\neq v \in
  \phi_{F_\Sigma,\mathbf{b}}.L(\lambda)_{F_\Sigma}$ such that
  $E_{\alpha_i}v=0$. We have $v=F_{\beta_1}^{a_1}\cdots
  F_{\beta_n}^{a_n} \tensor v'$ for some $a_1,\dots,a_n\in \setZ_{\leq 0}$ and
  some $v'\in L(\lambda)$. So $E_{\alpha_i}v=0$ implies
  \begin{equation*}
    0=\phi_{F_\Sigma,\mathbf{b}}(E_{\alpha_i}) F_{\beta_1}^{a_1}\cdots F_{\beta_n}^{a_n}\tensor v' =  F_{\beta_1}^{a_1}\cdots F_{\beta_n}^{a_n}\tensor \phi_{F_\Sigma,\mathbf{c}}(E_{\alpha_i})v'
  \end{equation*}
  where $\mathbf{c}=(q^{a_1}b_1,\dots,q^{a_n}b_n)$. So there exists a
  $v'\in L(\lambda)$ such that
  $\phi_{F_\Sigma,\mathbf{c}}(E_{\alpha_i})v'=0$. That is
  \begin{equation*}
    \left(E_{\alpha_i} + q\inv c_i \frac{c_i-c_i\inv}{q-q\inv} F_{\beta_i}\inv F_{\beta_{i-1}} K_{\alpha_i}\inv\right) v' = 0
  \end{equation*}
  or equivalently
  \begin{equation*}
    F_{\beta_i}E_{\alpha_i} v' = q\inv c_i \frac{c_i\inv - c_i}{q-q\inv} F_{\beta_{i-1}}K_{\alpha_i}\inv v'.
  \end{equation*}
  Since $L(\lambda)$ is a highest weight module we have some $r\in
  \setN$ such that $E_{\alpha_i}^{r}v' \neq 0$ and $E_{\alpha_i}^{r+1}
  v' =0$. Fix this $r$. We get
  \begin{align*}
    E_{\alpha_i}^{(r)}F_{\beta_i}E_{\alpha_i} v' =
    E_{\alpha_i}^{(r)}q\inv c_i \frac{c_i\inv - c_i}{q-q\inv}
    F_{\beta_{i-1}}K_{\alpha_i}\inv v'
  \end{align*}
  and calculating the right hand side and left hand side we get
  \begin{align*}
    q^{r-1}[r]F_{\beta_{i-1}}K_{\alpha_i}\inv E_{\alpha_i}^{(r)}v' =
    q^{-1+2r} c_i \frac{c_i\inv - c_i}{q-q\inv}
    F_{\beta_{i-1}}K_{\alpha_i}\inv E_{\alpha_i}^{(r)} v'.
  \end{align*}
  So we must have
  \begin{align*}
    q^{r-1}[r] = q^{-1+2r} c_i \frac{c_i\inv - c_i}{q-q\inv}
  \end{align*}
  or equivalently $c_i = \pm q^{-r}$. Since $c_i\in q^{\setZ} b_i$ we
  have proved the first claim.

  The other claim is shown similarly (see~e.g. the calculations done
  in the proof of Proposition~\ref{prop:7}. The calculations will be
  the same in this case).
\end{proof}

\begin{prop}
  \label{prop:5}
  Let $M$ be a weight $U_q$-module of finite Jordan-Hölder length with
  finite dimensional weight spaces. Let $\alpha\in \Pi$. If $E_\alpha$
  and $F_\alpha$ both act injectively on $M$ then $E_\alpha$ and
  $F_\alpha$ act injectively on every composition factor of $M$.
\end{prop}
\begin{proof}
  Let $V$ be a simple $U_q$-submodule of $M$. Let $\mu$ be a weight of
  $V$. Then $V_\mu$ is a simple $(U_q)_0$-module by
  Theorem~\ref{thm:Lemire} and $E_{\alpha}F_{\alpha}$ and
  $F_{\alpha}E_\alpha$ act injectively on $V_\mu$ by assumption. Since
  $\dim M_\mu<\infty$ this implies that $F_\alpha E_\alpha$ and
  $E_\alpha F_\alpha$ act injectively on the $(U_q)_0$ module
  $(M/V)_\mu \iso M_\mu /V_\mu$. Since $M/V$ is the sum of its weight
  spaces this implies that $E_\alpha F_\alpha$ and $F_\alpha E_\alpha$
  act injectively on $M/V$. This in turn implies that $E_\alpha$ and
  $F_\alpha$ act injectively on $M/V$. Doing induction on the
  Jordan-Hölder length of $M$ finishes the proof.
\end{proof}
The above proposition is true for a general simple Lie algebra
$\mathfrak{g}$ and we will use it in the next section as well.

\begin{thm}
  \label{thm:clas_of_b_such_that_twist_is_torsion_free}
  Let $\lambda$ be a weight such that $\lambda(K_{\alpha_i})\in \pm
  q^\setN$ for $i=2,\dots,n$ and $\lambda(K_{\alpha_1})\not \in \pm
  q^\setN$. Let $\mathbf{b}=(b_1,\dots,b_n)\in (\setC^*)^n$.  Then
  $\phi_{F_\Sigma,\mathbf{b}}.L(\lambda)_{F_\Sigma}$ is simple and
  torsion free if and only if $b_i\not \in \pm q^\setZ$, $i=1,\dots, n$
  and $\lambda(K_{\alpha_1})\inv b_1\cdots b_n \not \in \pm q^\setZ$.
\end{thm}
\begin{proof}
  By Proposition~\ref{prop:9} $L(\lambda)$ is a
  subquotient of
  \begin{equation*}
    {^{\bar{s_1}}}\left( \phi_{F_\Sigma,(\lambda(K_{\alpha_1}),1,\dots,1)}.L(\lambda)_{F_\Sigma}\right).
  \end{equation*}
  So by Lemma~\ref{lemma:1} we get (using that
  $L(\lambda)={^{s_1}}\left({^{\bar{s_1}}}L(\lambda)\right)$) for any
  $\mathbf{c}=(c_1,\dots,c_n) \in (\setC^*)^n$
  \begin{equation*}
    \left( \phi_{F_\Sigma,\mathbf{c}}.L(\lambda)_{F_\Sigma} \right)^{ss} \iso {^{\bar{s_1}}}\left( \phi_{F_\Sigma,(\lambda(K_{\alpha_1})c_1\inv \cdots c_n\inv,c_2,\dots,c_n)}.L(\lambda)_{F_\Sigma}\right)^{ss}.
  \end{equation*}
  We have $\lambda(K_{\alpha_2})=\epsilon q^{r}$ for some $r\in \setN$
  and some $\epsilon \in \{\pm 1\}$. We see in the proof of
  Lemma~\ref{lemma:30} that $L(\lambda)$ is a subqoutient of
  \begin{equation*}
    {^{\bar{s_2}}}\left( \phi_{F_\Sigma,(\epsilon,\epsilon,1,\dots,1)}.L(\lambda)_{F_\Sigma}\right).
  \end{equation*}
  We get by Lemma~\ref{lemma:1} (using that
  $L(\lambda)={^{s_2}}\left({^{\bar{s_2}}}L(\lambda)\right)$) for any
  $\mathbf{c}=(c_1,\dots,c_n) \in (\setC^*)^n$
  \begin{equation*}
    \left( \phi_{F_\Sigma,\mathbf{c}}.L(\lambda)_{F_\Sigma}\right)^{ss} \iso {^{\bar{s_2}}}\left( \phi_{F_\Sigma,(\epsilon c_2,\epsilon c_1,c_3,\dots,c_n)}.L(\lambda)_{F_\Sigma}\right)^{ss}.
  \end{equation*}
  Combining the above we get
  \begin{align*}
    \left(
      \phi_{F_\Sigma,\mathbf{b}}.L(\lambda)_{F_\Sigma}\right)^{ss}
    \iso& {^{\bar{s_2}}}\left( \phi_{F_\Sigma,(\epsilon b_2,\epsilon
        b_1,b_3,\dots,b_n)}.L(\lambda)_{F_\Sigma}\right)^{ss}
    \\
    \iso& {^{\bar{s_2}}}\left({^{\bar{s_1}}}\left(
        \phi_{F_\Sigma,(\lambda(K_{\alpha_1})b_1\inv\cdots b_n\inv,
          \epsilon
          b_1,b_3,\dots,b_n)}.L(\lambda)_{F_\Sigma}\right)\right)^{ss}
    \\
    \iso& {^{\bar{s_1s_2}}}\left(
      \phi_{F_\Sigma,(\lambda(K_{\alpha_1})b_1\inv\cdots b_n\inv,
        \epsilon b_1,b_3,\dots,b_n)}.L(\lambda)_{F_\Sigma}\right)^{ss}.
  \end{align*}
  Since $T_{s_1}\inv T_{s_2}\inv(E_{\alpha_1})=E_{\alpha_2}$ and
  $T_{s_1}\inv T_{s_2}\inv(F_{\alpha_1})=F_{\alpha_2}$ we get by
  Proposition~\ref{prop:4} that $E_{\alpha_1}$ acts injectively on
  ${^{\bar{s_1s_2}}}\left(
    \phi_{F_{\Sigma},(\lambda(K_{\alpha_1})b_1\inv \cdots
      b_n\inv,\epsilon
      b_1,b_3,\dots,b_n)}.L(\lambda)_{F_{\Sigma}}\right)$ if and only
  if $b_1 \not \in \pm q^{\setZ}$ and $F_{\alpha_1}$ acts injectively
  on ${^{\bar{s_1s_2}}}\left(
    \phi_{F_{\Sigma},(\lambda(K_{\alpha_1})b_1\inv \cdots
      b_n\inv,\epsilon
      b_1,b_3,\dots,b_n)}.L(\lambda)_{F_{\Sigma}}\right)$ if and only
  if $\lambda(K_{\alpha_1})\inv b_1\cdots b_n \not \in \pm q^{\setZ}$.

  Assume $\phi_{F_\Sigma,\mathbf{b}}.L(\lambda)_{F_\Sigma}$ is torsion
  free. Then all root vectors act injectively on
  $\phi_{F_\Sigma,\mathbf{b}}.L(\lambda)_{F_\Sigma}$. We claim
  $\phi_{F_\Sigma,\mathbf{b}}.L(\lambda)_{F_\Sigma}$ is simple: Let
  $V\subset \phi_{F_\Sigma,\mathbf{b}}.L(\lambda)_{F_\Sigma}$ be a
  simple module. Then $V$ is admissible of the same degree $d$ as
  $L(\lambda)$ by Proposition~\ref{prop:15} and because all root
  vectors act injectively $\dim V_{q^\mu\lambda}=d$ for all $\mu\in
  Q$. So $V= \phi_{F_\Sigma,\mathbf{b}}.L(\lambda)_{F_\Sigma}$. Thus
  $\left(\phi_{F_\Sigma,\mathbf{b}}.L(\lambda)_{F_\Sigma}\right)^{ss}=\phi_{F_\Sigma,\mathbf{b}}.L(\lambda)_{F_\Sigma}$. Then
  by the above
  \begin{equation*}
    \phi_{F_\Sigma,\mathbf{b}}.L(\lambda)_{F_\Sigma} \iso {^{\bar{s_1s_2}}}\left( \phi_{F_{\Sigma},(\lambda(K_{\alpha_1})b_1\inv \cdots b_n\inv,\epsilon b_1,b_3,\dots,b_n)}.L(\lambda)_{F_{\Sigma}}\right).
  \end{equation*}
  This shows that when
  $\phi_{F_\Sigma,\mathbf{b}}.L(\lambda)_{F_{\Sigma}}$ is torsion free
  we must have $\lambda(K_{\alpha_1})\inv b_1\cdots b_n \not \in \pm
  q^{\setZ}$.  By Proposition~\ref{prop:4} $b_i\not \in \pm q^\setZ$,
  $i=1,\dots, n$.
  
  Assume on the other hand that $b_i \not \in \pm q^{\setZ}$ for $i\in
  \{1,\dots,n\}$ and $\lambda(K_{\alpha_1})\inv b_1\cdots b_n\not \in
  \pm q^{\setZ}$. By Proposition~\ref{prop:4} we get that the simple
  root vectors $E_{\alpha_2},\dots,E_{\alpha_n}$ and
  $F_{\alpha_1},\dots,F_{\alpha_n}$ all act injectively on
  $\phi_{F_\Sigma,\mathbf{b}}.L(\lambda)_{F_\Sigma}$. We need to show
  that $E_{\alpha_1}$ acts injectively on the module.  By the above
  \begin{equation*}
    \left(\phi_{F_\Sigma,\mathbf{b}}.L(\lambda)_{F_\Sigma}\right)^{ss} \iso {^{\bar{s_1s_2}}}\left( \phi_{F_{\Sigma},(\lambda(K_{\alpha_1})b_1\inv \cdots b_n\inv,\epsilon b_1,b_3,\dots,b_n)}.L(\lambda)_{F_{\Sigma}}\right)^{ss}
  \end{equation*}
  and the root vectors $E_{\alpha_1},F_{\alpha_1}$ act injectively on
  \begin{equation*}
    {^{\bar{s_1s_2}}}\left(
      \phi_{F_{\Sigma},(\lambda(K_{\alpha_1})b_1\inv \cdots
        b_n\inv,\epsilon b_1,b_3,\dots,b_n)}.L(\lambda)_{F_{\Sigma}}\right).
  \end{equation*}
  Then by Proposition~\ref{prop:5} $E_{\alpha_1}$ act injectively on
  all composition factors of
  $\phi_{F_\Sigma,\mathbf{b}}.L(\lambda)_{F_\Sigma}$.

  Let $V$ be a simple $U_q$-submodule of
  $\phi_{F_\Sigma,\mathbf{b}}.L(\lambda)_{F_\Sigma}$. By the above all
  simple root vectors act injectively on $V$ and then like above this
  implies $V=\phi_{F_\Sigma,\mathbf{b}}.L(\lambda)_{F_\Sigma}$
  i.e. $\phi_{F_\Sigma,\mathbf{b}}.L(\lambda)_{F_\Sigma}$ is simple
  and torsion free.
\end{proof}

By the comments after Corollary~\ref{cor:2} the above Theorem
completes the classification of simple torsion free modules in type A.

\section{Classification of simple torsion free modules. Type C.}
\label{sec:type-c_n-calc}
In this section we assume $\mathfrak{g}$ is of type $C_n$
(i.e. $\mathfrak{g}=\mathfrak{sp}_{2n}$) with $n\geq 2$. Let
$\Pi=\{\alpha_1,\dots,\alpha_n\}$ denote the simple roots such that
$(\alpha_i|\alpha_{i+1})=-1$, $i=2,\dots,n-1$,
$\left<\alpha_2,\alpha_1^\vee\right>=-1$ and
$\left<\alpha_1,\alpha_2^\vee\right>=-2$ i.e. $\alpha_1$ is long and
$\alpha_2,\dots,\alpha_n$ are short.

Set $\beta_j = s_1\cdots s_{j-1}(\alpha_j)=\alpha_1+\dots+\alpha_j$,
then $\Sigma=\{\beta_1,\dots,\beta_n\}$ is a set of commuting roots
with corresponding root vectors $F_{\beta_j} = T_{s_1}\cdots
T_{s_{j-1}}(F_{\alpha_j})$. We will show some commutation formulas and
use these to calculate $\phi_{F_\Sigma,\mathbf{b}}$ on most of the
simple root vectors.

Choose a reduced expression of $w_0$ starting with $s_1\cdots s_n
s_1\cdots s_{n-1}$ and define root vectors
$F_{\gamma_1},\dots,F_{\gamma_N}$ from this expression. Note that
$F_{\beta_i}=F_{\gamma_i}$ for $i=1,\dots, n$. Note for use in the
proposition below that for $j\in \{1,\dots,n-1\}$,
\begin{align*}
  \gamma_{n+j} = s_1\cdots s_n s_1\cdots
  s_{j-1}(\alpha_j)=\alpha_1+2\alpha_2+\alpha_3+\cdots \alpha_{j+1}
\end{align*}
and
\begin{align*}
  F_{\gamma_{n+j}}=&T_{s_1}\cdots T_{s_n} T_{s_1}\cdots
  T_{s_{j-1}}(F_{\alpha_j})
  \\
  =& T_{s_1}\cdots T_{s_{j+1}} T_{s_1}\cdots
  T_{s_{j-1}}(F_{\alpha_j}).
\end{align*}
In particular $F_{\alpha_1+2\alpha_2}=T_{s_1}T_{s_2}(F_{\alpha_1})$.

\begin{prop}
  \label{prop:32}
  Let $i\in \{2,\dots,n\}$ and $j\in \{1,\dots,n\}$
  \begin{equation*}
    [F_{\alpha_i},F_{\beta_j}]_q =
    \begin{cases}
      [2] F_{\alpha_1+2\alpha_2}, &\text{ if } j=i=2
      \\
      F_{\alpha_1+2\alpha_2+\alpha_3+\cdots+\alpha_{j}}, &\text{ if }
      i=2 \text{ and } j>2
      \\
      F_{\beta_i}, &\text{ if } j = i-1
      \\
      0, &\text{ otherwise}
    \end{cases}
  \end{equation*}
  and
  \begin{equation*}
    [E_{\alpha_i},F_{\beta_j}] =
    \begin{cases}
      [2] F_{\beta_{1}} K_{\alpha_2}\inv, &\text{ if } j = 2 = i
      \\
      F_{\beta_{i-1}} K_{\alpha_i}\inv, &\text{ if } j = i > 2
      \\
      0, &\text{ otherwise.}
    \end{cases}
  \end{equation*}
\end{prop}
\begin{proof}
  We will show the proposition for the $F$'s first and then for the
  $E$'s.  Assume first that $j<i-1$. Then clearly
  $[F_{\alpha_i},F_{\beta_j}]_q = [F_{\alpha_i},F_{\beta_j}]=0$ since
  $\alpha_i$ is not connected to any of the simple roots
  $\alpha_1,\dots,\alpha_j$ appearing in $\beta_j$.
 
  Then assume $j\geq i>2$. We must have $\alpha_i = \gamma_k$ for some
  $k>n$ since $\{\gamma_1,\dots,\gamma_N\}=\Phi^+$. By
  Theorem~\ref{thm:DP} $[F_{\alpha_i},F_{\beta_j}]_q$ is a linear
  combination of monomials of the form
  $F_{\gamma_{j+1}}^{a_{j+1}}\cdots F_{\gamma_{k-1}}^{a_{k-1}}$. For a
  monomial of this form to appear with nonzero coefficient we must
  have
  \begin{equation*}
    \sum_{h=j+1}^{k-1} a_h \gamma_h = \alpha_i + \beta_j = \alpha_1+\dots + \alpha_{i-1}+2\alpha_i +\alpha_{i+1}+\dots \alpha_j.
  \end{equation*}
  For this to be possible one of the positive roots $\gamma_s$,
  $j<s<k$ must be equal to $\alpha_1+\alpha_2+\dots+\alpha_m$ for some
  $m\leq j$ but $\alpha_1+\alpha_2+\dots+\alpha_m=\gamma_m$ by
  construction and $m\leq j<s$ so $m\neq s$. We conclude that this is
  not possible.
  
  Assume $j=i-1$. We have
  \begin{align*}
    [F_{\alpha_i},F_{\beta_{i-1}}]_q =& [T_{s_1}\cdots
    T_{s_{i-2}}(F_{\alpha_i}),T_{s_1}\cdots
    T_{s_{i-2}}(F_{\alpha_{i-1}})]_q
    \\
    =& T_{s_1}\cdots T_{s_{i-2}}
    \left([F_{\alpha_i},F_{\alpha_{i-1}}]_q\right)
    \\
    =& T_{s_1}\cdots T_{s_{i-2}} T_{s_{i-1}}(F_{\alpha_i})
    \\
    =& F_{\beta_i}.
  \end{align*}
  
  Assume $j=2=i$. Then
  \begin{align*}
    [F_{\alpha_2},F_{\beta_2}]_q =&
    F_{\alpha_2}F_{\beta_2}-F_{\beta_2}F_{\alpha_2}
    \\
    =&
    F_{\alpha_2}(F_{\alpha_2}F_{\alpha_1}-q^2F_{\alpha_1}F_{\alpha_2})-(F_{\alpha_2}F_{\alpha_1}-q^2
    F_{\alpha_1}F_{\alpha_2})F_{\alpha_2}
    \\
    =& (q^2 F_{\alpha_1}F_{\alpha_2}^2 - q [2]
    F_{\alpha_2}F_{\alpha_1}F_{\alpha_2} + F_{\alpha_2}^2
    F_{\alpha_1})
    \\
    =& [2] T_{s_2}\inv (F_{\alpha_1})
    \\
    =& [2] T_{s_2}\inv T_{s_2}T_{s_1}T_{s_2}(F_{\alpha_1})
    \\
    =& [2] T_{s_1}T_{s_2}(F_{\alpha_1})
    \\
    =& [2] F_{\alpha_1+2\alpha_2}.
  \end{align*}
  
  Assume $i=2$ and $j=3$. Then
  \begin{align*}
    F_{\alpha_1+2\alpha_2+\alpha_3} =& F_{\gamma_{n+2}}
    \\
    =& T_{s_1} T_{s_{2}}T_{s_{3}}T_{s_{1}}(F_{\alpha_{2}})
    \\
    =& T_{s_1} T_{s_2} T_{s_1} T_{s_3}(F_{\alpha_2})
    \\
    =& T_{s_1} T_{s_2} T_{s_1}
    (F_{\alpha_2}F_{\alpha_3}-qF_{\alpha_3}F_{\alpha_2})
    \\
    =& F_{\alpha_2}F_{\beta_3}-qF_{\beta_3}F_{\alpha_2}.
  \end{align*}
  
  Finally assume $i=2$ and $j>3$. We have
  \begin{align*}
    F_{\alpha_1+2\alpha_2+\alpha_3+\cdots+\alpha_j} =&
    F_{\gamma_{n+j-1}}
    \\
    =& T_{s_1}\cdots T_{s_{j-2}}T_{s_{j-1}}T_{s_{j}}T_{s_{1}}\cdots
    T_{s_{j-3}}T_{s_{j-2}}(F_{\alpha_{j-1}})
    \\
    =& T_{s_1}\cdots T_{s_{j-2}}T_{s_{1}}\cdots
    T_{s_{j-3}}T_{s_{j-1}}T_{s_{j-2}}T_{s_{j}}(F_{\alpha_{j-1}})
    \\
    =& T_{s_1}\cdots T_{s_{j-2}}T_{s_{1}}\cdots
    T_{s_{j-3}}T_{s_{j-1}}T_{s_{j-2}}(F_{\alpha_{j-1}}F_{\alpha_{j}}-qF_{\alpha_j}F_{\alpha_{j-1}})
    \\
    =& T_{s_1}\cdots T_{s_{j-2}}T_{s_{1}}\cdots
    T_{s_{j-3}}(F_{\alpha_{j-2}}T_{s_{j-1}}(F_{\alpha_{j}})-qT_{s_{j-1}}(F_{\alpha_j})F_{\alpha_{j-2}})
    \\
    =& F_{\alpha_2} F_{\beta_j} - q F_{\beta_j} F_{\alpha_2}
    \\
    =& [F_{\alpha_2},F_{\beta_j}]_q
  \end{align*}
  using the facts that
  $T_{s_{j-1}}T_{s_{j-2}}(F_{\alpha_{j-1}})=F_{\alpha_{j-2}}$ and
  $T_{s_1}\cdots T_{s_{j-2}}T_{s_1}\cdots
  T_{s_{j-3}}(F_{\alpha_{j-2}})=F_{\alpha_2}$ by Proposition~8.20
  in~\cite{Jantzen} (The proposition is about the $E$ root vectors but
  the proposition is true for the $F$'s as well).

  For the $E$'s: Assume first $j<i$: Since $F_{\beta_j}$ is a
  polynomial in $F_{\alpha_1},\dots,F_{\alpha_j}$, $E_{\alpha_i}$
  commutes with $F_{\beta_j}$ when $j<i$.

  Assume then $j=i$: We have by the above
  \begin{equation*}
    F_{\beta_i} = [F_{\alpha_i},F_{\beta_{i-1}}]_q
  \end{equation*}
  so
  \begin{align*}
    [E_{\alpha_i},F_{\beta_i}] =&
    [E_{\alpha_i},(F_{\alpha_i}F_{\beta_{i-1}}-q^{-(\beta_{i-1}|\alpha_i)}F_{\beta_{i-1}}F_{\alpha_i})]
    \\
    =& [E_{\alpha_i},F_{\alpha_i}] F_{\beta_{i-1}} - q_{\alpha_{i-1}}
    F_{\beta_{i-1}} [E_{\alpha_i},F_{\alpha_i}]
    \\
    =& \frac{ K_{\alpha_i}-K_{\alpha_i}\inv}{q-q\inv} F_{\beta_{i-1}}
    - q_{\alpha_{i-1}} F_{\beta_{i-1}}
    \frac{K_{\alpha_i}-K_{\alpha_i}\inv}{q-q\inv}
    \\
    =& F_{\beta_{i-1}} \frac{ q_{\alpha_{i-1}} K_{\alpha_i} -
      q_{\alpha_{i-1}}\inv K_{\alpha_i}\inv - q_{\alpha_{i-1}}
      K_{\alpha_i} + q_{\alpha_{i-1}} K_{\alpha_i}\inv}{q-q\inv}
    \\
    =& \frac{q_{\alpha_{i-1}}-q_{\alpha_{i-1}}\inv}{q-q\inv}
    F_{\beta_{i-1}}K_{\alpha_i}\inv
    \\
    =&
    \begin{cases}
      [2] F_{\beta_{i-1}}K_{\alpha_{i}}\inv, &\text{ if } i=2
      \\
      F_{\beta_{i-1}}K_{\alpha_{i}}\inv, &\text{ otherwise. }
    \end{cases}
  \end{align*}
  
  Finally assume $j>i$: Observe first that we have
  \begin{equation*}
    T_{s_{i+1}}\cdots T_{s_{j-1}}F_{\alpha_j} = \sum_{s=1}^m u_s F_{\alpha_{i+1}} u_s'
  \end{equation*}
  for some $m\in \setN$ and some $u_s,u_s'$ that are polynomials in
  $F_{\alpha_{i+2}},\dots F_{\alpha_{j}}$. Note that
  $T_{s_i}(u_s)=u_s$ and $T_{s_i}(u_s')=u_s'$ for all $s$ since
  $\alpha_i$ is not connected to any of the simple roots
  $\alpha_{i+2},\dots \alpha_j$. So
  \begin{align*}
    T_{s_i}T_{s_{i+1}}\cdots T_{s_{j-1}}F_{\alpha_j} =& T_{s_i}\left(
      \sum_{s=1}^m u_s F_{\alpha_{i+1}} u_s'\right)
    \\
    =& \sum_{s=1}^m u_s T_{s_i}(F_{\alpha_{i+1}}) u_s'
    \\
    =& \sum_{s=1}^m u_s
    (F_{\alpha_{i+1}}F_{\alpha_i}-qF_{\alpha_i}F_{\alpha_{i+1}}) u_s'
    \\
    =& \sum_{s=1}^m u_s F_{\alpha_{i+1}} u_s' F_{\alpha_i} - q
    F_{\alpha_i}\sum_{s=1}^m u_s F_{\alpha_{i+1}} u_s'
    \\
    =& T_{s_{i+1}}\cdots T_{s_{j-1}}(F_{\alpha_j}) F_{\alpha_i} - q
    F_{\alpha_i} T_{s_{i+1}}\cdots T_{s_{j-1}}(F_{\alpha_j}).
  \end{align*}
  Thus we see that
  \begin{align*}
    F_{\beta_j} =& T_{s_1}\dots T_{s_i}\cdots
    T_{s_{j-1}}(F_{\alpha_j})
    \\
    =&T_{s_{i+1}}\cdots T_{s_{j-1}}(F_{\alpha_j}) T_{s_{1}}\cdots
    T_{s_{i-1}}(F_{\alpha_i}) - qF_{\alpha_i} T_{s_{i+1}}\cdots
    T_{s_{j-1}}(F_{\alpha_j})
    \\
    =& T_{s_{i+1}}\cdots T_{s_{j-1}}(F_{\alpha_j}) F_{\beta_i} - q
    F_{\beta_i} T_{s_{i+1}}\cdots T_{s_{j-1}}(F_{\alpha_j})
  \end{align*}
  and therefore
  \begin{align*}
    [E_{\alpha_i},F_{\beta_j}] =& T_{s_{i+1}}\cdots
    T_{s_{j-1}}(F_{\alpha_j}) [E_{\alpha_i},F_{\beta_i}] - q
    [E_{\alpha_i},F_{\beta_i}]T_{s_{i+1}}\cdots
    T_{s_{j-1}}(F_{\alpha_j})
    \\
    =& [r_i] (T_{s_{i+1}}\cdots T_{s_{j-1}}(F_{\alpha_j})
    F_{\beta_{i-1}}K_{\alpha_i}\inv- q F_{\beta_{i-1}}K_{\alpha_i}\inv
    T_{s_{i+1}}\cdots T_{s_{j-1}}(F_{\alpha_j}) )
    \\
    =& [r_i] (F_{\beta_{i-1}} T_{s_{i+1}}\cdots
    T_{s_{j-1}}(F_{\alpha_j}) K_{\alpha_i}\inv - F_{\beta_{i-1}}
    T_{s_{i+1}}\cdots T_{s_{j-1}}(F_{\alpha_j}) K_{\alpha_i}\inv )
    \\
    =& 0
  \end{align*}
  where
  \begin{equation*}
    r_i = 
    \begin{cases}
      2 &\text{ if } i =2
      \\
      1 &\text{ otherwise. }
    \end{cases}
  \end{equation*}

\end{proof}

\begin{prop}
  \label{prop:33}
  Let $i\in \{2,\dots,n\}$. Let $a\in \setZ_{>0}$. Then
  \begin{equation*}
    [F_{\alpha_i},F_{\beta_{i-1}}^a]_q = [a]_{\beta_{i-1}} F_{\beta_{i-1}}^{a-1}F_{\beta_i}
  \end{equation*}
  and for $b\in \setC^*$
  \begin{equation*}
    \phi_{F_{\beta_{i-1}},b}(F_{\alpha_i}) =
    \begin{cases}
      b^2 F_{\alpha_2}+ \frac{b^2-b^{-2}}{q^2-q^{-2}}
      F_{\beta_{1}}\inv F_{\beta_2}, &\text{ if } i=2
      \\
      b F_{\alpha_i}+ \frac{b-b\inv}{q-q\inv} F_{\beta_{i-1}}\inv
      F_{\beta_i}, &\text{ otherwise. }
    \end{cases}
  \end{equation*}
\end{prop}
\begin{proof}
  The first claim is proved by induction over $a$. $a=1$ is shown in
  Proposition~\ref{prop:26}. The induction step:
  \begin{align*}
    F_{\alpha_i}F_{\beta_{i-1}}^{a+1} =& \left( q_{\beta_{i-1}}^a
      F_{\beta_{i-1}}^a F_{\alpha_i} + [a]_{\beta_{i-1}}
      F_{\beta_{i-1}}^{a-1}F_{\beta_i}\right) F_{\beta_{i-1}}
    \\
    =& q_{\beta_{i-1}}^{a+1}F_{\beta_{i-1}}^{a+1}F_{\alpha_i}
    +q_{\beta_{i-1}}^a F_{\beta_{i-1}}^a F_{\beta_i} +
    q_{\beta_{i-1}}\inv [a]_{\beta_{i-1}} F_{\beta_{i-1}}^a F_{\beta_i}
    \\
    =& q_{\beta_{i-1}}^{a+1}F_{\beta_{i-1}}^{a+1}F_{\alpha_i} +
    [a+1]_{\beta_{i-1}} F_{\beta_{i-1}}^a F_{\beta_i}.
  \end{align*}
  So we have proved the first claim. We get then for $a\in \setZ_{>0}$:
  \begin{equation*}
    \phi_{F_{\beta_{i-1}},q^a}(F_{\alpha_i}) = F_{\beta_{i-1}}^{-a} F_{\alpha_i} F_{\beta_{i-1}}^a = q_{\beta_{i-1}}^a F_{\alpha_i} + \frac{q_{\beta_{i-1}}^a-q_{\beta_{i-1}}^{-a}}{q_{\beta_{i-1}}-q_{\beta_{i-1}}\inv} F_{\beta_{i-1}}\inv F_{\beta_i}.
  \end{equation*}
  Using the fact that $\phi_{F_{\beta_{i-1}},b}(F_{\alpha_i})$ is
  Laurent polynomial in $b$ we get the second claim of the
  proposition.
\end{proof}

\begin{prop}
  \label{prop:34}
  Let $i\in \{2,\dots,n\}$. Let $a\in \setZ_{>0}$. Then
  \begin{equation*}
    [E_{\alpha_i},F_{\beta_i}^a] =
    \begin{cases}
      q^{a-1}[2][a] F_{\beta_2}^{a-1}F_{\beta_{1}}K_{\alpha_2}\inv,
      &\text{ if } i=2
      \\
      q^{a-1}[a] F_{\beta_i}^{a-1}F_{\beta_{i-1}}K_{\alpha_i}\inv,
      &\text{ otherwise. }
    \end{cases}
  \end{equation*}
  and for $b\in \setC^*$
  \begin{equation*}
    \phi_{F_{\beta_{i}},b}(E_{\alpha_i}) =
    \begin{cases}
      E_{\alpha_2}+ q^{-1} [2]b \frac{b-b^{-1}}{q-q\inv}
      F_{\beta_2}\inv F_{\beta_{1}} K_{\alpha_2}\inv, &\text{ if } i=2
      \\
      E_{\alpha_i}+ q\inv b \frac{b-b\inv}{q-q\inv} F_{\beta_i}\inv
      F_{\beta_{i-1}} K_{\alpha_i}\inv, &\text{ otherwise. }
    \end{cases}
  \end{equation*}
\end{prop}
\begin{proof}
  The first claim is proved by induction over $a$. $a=1$ is shown in
  Proposition~\ref{prop:26}. The induction step: For $i>2$:
  \begin{align*}
    E_{\alpha_i} F_{\beta_i}^{a+1} =& \left( F_{\beta_i}^a
      E_{\alpha_i}+ q^{a-1} [a] F_{\beta_i}^{a-1}F_{\beta_{i-1}}
      K_{\alpha_i}\inv \right) F_{\beta_i}
    \\
    =& F_{\beta_i}^{a+1} E_{\alpha_i} + F_{\beta_i}^a F_{\beta_{i-1}}
    K_{\alpha_i}\inv + q^{a+1} [a] F_{\beta_i}^a
    F_{\beta_{i-1}}K_{\alpha_i}\inv
    \\
    =& F_{\beta_i}^{a+1} E_{\alpha_i} + q^{a} (q^{-a} +
    q[a])F_{\beta_i}^a F_{\beta_{i-1}}K_{\alpha_i}\inv
    \\
    =& F_{\beta_i}^{a+1} E_{\alpha_i} +
    q^{a}[a+1]_{\alpha_{i-1}}F_{\beta_i}^a
    F_{\beta_{i-1}}K_{\alpha_i}\inv.
  \end{align*}
  
  For $i=2$:
  \begin{align*}
    E_{\alpha_2} F_{\beta_2}^{a+1} =& \left( F_{\beta_2}^a
      E_{\alpha_2}+ q^{a-1}[2][a] F_{\beta_2}^{a-1}F_{\beta_{1}}
      K_{\alpha_2}\inv \right) F_{\beta_2}
    \\
    =& F_{\beta_2}^{a+1} E_{\alpha_2} + [2] F_{\beta_2}^a
    F_{\beta_{1}} K_{\alpha_2}\inv + q^{a+1} [2][a] F_{\beta_2}^a
    F_{\beta_{1}}K_{\alpha_2}\inv
    \\
    =& F_{\beta_2}^{a+1} E_{\alpha_2} + q^{a} [2](q^{-a} + q
    [a])F_{\beta_2}^a F_{\beta_{1}}K_{\alpha_2}\inv
    \\
    =& F_{\beta_2}^{a+1} E_{\alpha_2} + q^{a} [2][a+1] F_{\beta_2}^a
    F_{\beta_{1}}K_{\alpha_2}\inv.
  \end{align*}

  This proves the first claim. We get then for $a\in \setZ_{>0}$
  \begin{equation*}
    \phi_{F_{\beta_i},q^a}(E_{\alpha_i}) =F_{\beta_i}^{-a} E_{\alpha_i} F_{\beta_i}^a =
    \begin{cases}
      E_{\alpha_2} + q^{-2} q^{2a} \frac{q^{2a} - q^{-2a}}{q-q\inv}
      F_{\beta_2}\inv F_{\beta_{1}}K_{\alpha_2}\inv, &\text{ if } i=2
      \\
      E_{\alpha_i} + q\inv q^a \frac{q^a - q^{-a}}{q-q\inv}
      F_{\beta_i}\inv F_{\beta_{i-1}}K_{\alpha_i}\inv, &\text{
        otherwise. }
    \end{cases}
  \end{equation*}
  Using the fact that $\phi_{F_{\beta_{i}},b}(E_{\alpha_i})$ is
  Laurent polynomial in $b$ we get the second claim of the
  proposition.
\end{proof}

We combine the above propositions in the following proposition
\begin{prop}
  \label{prop:37}
  Let $i\in \{3,\dots,n\}$. For $\mathbf{b}=(b_1,\dots,b_n)\in
  (\setC^*)^n$
  \begin{align*}
    \phi_{F_\Sigma,\mathbf{b}}(F_{\alpha_{i}}) =&
    \phi_{F_{\beta_{i-1},b_{i-1}}}(F_{\alpha_i})
    \\
    =& b_{i-1}F_{\alpha_i}+ \frac{b_{i-1}-b_{i-1}\inv}{q-q\inv}
    F_{\beta_{i-1}}\inv F_{\beta_i}
    \\
    \phi_{F_\Sigma,\mathbf{b}}(E_{\alpha_i}) =&
    \phi_{F_{\beta_{i},b_{i}}}(E_{\alpha_i}) = E_{\alpha_i}+ q\inv b_i
    \frac{b_i-b_i\inv}{q-q\inv} F_{\beta_i}\inv F_{\beta_{i-1}}
    K_{\alpha_i}\inv.
  \end{align*}
  Furthermore
  \begin{equation*}
    \phi_{F_\Sigma,\mathbf{b}}(E_{\alpha_2})= E_{\alpha_2} + q^{-1}[2] b_2 \frac{b_2-b_2\inv}{q-q\inv} F_{\beta_2}\inv F_{\beta_1}K_{\alpha_2}\inv
  \end{equation*}
  
  and
  \begin{equation*}
    \phi_{F_\Sigma,\mathbf{b}}(F_{\alpha_1}) = b_2\cdots b_n
    F_{\alpha_1}.
  \end{equation*}
\end{prop}

With similar proof as the proof of Proposition~\ref{prop:4} we can
show
\begin{prop}
  \label{prop:7}
  Let $\lambda$ be a weight such that $\lambda(K_{\beta})\in \pm
  q^\setN$ for all short $\beta\in \Phi^+$ and $\lambda(K_{\gamma})
  \in \pm q^{1+2\setZ}$ for all long $\gamma \in \Phi^+$.  Let
  $\mathbf{b}=(b_1,\dots,b_n)\in (\setC^*)^n$. $E_{\alpha_2}$ acts
  injectively on the $U_q$-module
  $\phi_{F_\Sigma,\mathbf{b}}.L(\lambda)_{F_\Sigma}$ if and only if
  $b_2 \not \in \pm q^{\setZ}$.  Let $i\in\{3,\dots,n\}$. Then
  $E_{\alpha_i}$ acts injectively on the module
  $\phi_{F_\Sigma,\mathbf{b}}.L(\lambda)_{F_\Sigma}$ if and only if
  $b_i \not \in \pm q^{\setZ}$ and $F_{\alpha_i}$ acts injectively on
  $\phi_{F_\Sigma,\mathbf{b}}.L(\lambda)_{F_\Sigma}$ if and only if
  $b_{i-1}\not \in \pm q^{\setZ}$.
\end{prop}
\begin{proof}
  By Proposition~\ref{prop:25} and Corollary~\ref{cor:6} a root vector
  acts injectively on the $U_q$-module
  \begin{equation*}
    \phi_{F_\Sigma,(b_1,\dots,b_n)}.L(\lambda)_{F_\Sigma}
  \end{equation*}
  if and only if it acts injectively on
  \begin{equation*}
    \phi_{F_\Sigma,(\epsilon_1
      q^{i_1}b_1,\dots,\epsilon_n q^{i_n}b_n)}.L(\lambda)_{F_\Sigma}
  \end{equation*}
  for any $i_1,\dots,i_n\in\setZ$ and $\epsilon_1,\dots,\epsilon_n\in
  \{\pm 1\}$.
  
  Assume there exists a $0\neq v \in
  \phi_{F_\Sigma,\mathbf{b}}.L(\lambda)_{F_\Sigma}$ such that
  $F_{\alpha_i}v=0$. We have $v=F_{\beta_1}^{a_1}\cdots
  F_{\beta_n}^{a_n} \tensor v'$ for some $a_1,\dots,a_n\in \setZ_{\leq 0}$ and
  some $v'\in L(\lambda)$. $F_{\alpha_i}v=0$ implies
  \begin{equation*}
    0=\phi_{F_\Sigma,\mathbf{b}}(F_{\alpha_i}) F_{\beta_1}^{a_1}\cdots F_{\beta_n}^{a_n}\tensor v' =  F_{\beta_1}^{a_1}\cdots F_{\beta_n}^{a_n}\tensor \phi_{F_\Sigma,\mathbf{c}}(F_{\alpha_i})v'
  \end{equation*}
  where $\mathbf{c}=(q^{a_1}b_1,\dots,q^{a_n}b_n)$. So there exists a
  $v'\in L(\lambda)$ such that
  $\phi_{F_\Sigma,\mathbf{c}}(F_{\alpha_i})v'=0$. That is
  \begin{equation*}
    \left(c_{i-1}F_{\alpha_i} +  \frac{c_{i-1}-c_{i-1}\inv}{q-q\inv} F_{\beta_{i-1}}\inv F_{\beta_{i}}\right) v' = 0
  \end{equation*}
  or equivalently
  \begin{equation*}
    \left(F_{\beta_{i-1}}F_{\alpha_i} +  c_{i-1}\inv \frac{c_{i-1} - c_{i-1}\inv}{q-q\inv} F_{\beta_i}\right) v' = 0.
  \end{equation*}
  Let $r\in \setN$ be such that $F_{\alpha_i}^{(r)}v'\neq 0$ and
  $F_{\alpha_i}^{(r+1)}v'=0$ (possible since $\lambda(K_{\alpha_i})\in
  \pm q^{\setN}$ so $-\alpha_i \in F_{L(\lambda)}$). So the above
  being equal to zero implies
  \begin{align*}
    0=&F_{\alpha_i}^{(r)}\left(F_{\beta_{i-1}}F_{\alpha_i} +
      c_{i-1}\inv \frac{c_{i-1} - c_{i-1}\inv}{q-q\inv}
      F_{\beta_i}\right) v'
    \\
    =& \left( [r] F_{\beta_i}F_{\alpha_i}^{(r)} + q^{-r}
      \frac{1-c_{i-1}^{-2}}{q-q\inv} F_{\beta_i}
      F_{\alpha_i}^{(r)}\right)v'
    \\
    =& \left( [r] + q^{-r} \frac{1-c_{i-1}^{-2}}{q-q\inv}\right)
    F_{\beta_i}F_{\alpha_i}^{(r)}v'.
  \end{align*}
  Since $F_{\beta_i}F_{\alpha_i}^{(r)}v'\neq 0$ this is equivalent to
  \begin{align*}
    0=q^r - q^{-r} +q^{-r} - q^{-r}c_{i-1}^{-2}= q^r -
    q^{-r}c_{i-1}^{-2}
  \end{align*}
  or equivalently $c_{i-1} = \pm q^{-r}$.

  The other claims are shown similarly.
\end{proof}

\begin{prop}
  \label{prop:6}
  Let $\lambda$ be a weight such that $\lambda(K_{\beta})\in \pm
  q^\setN$ for all short $\beta\in \Phi^+$ and $\lambda(K_{\gamma})
  \in \pm q^{1+2\setZ}$ for all long $\gamma \in \Phi^+$. Let
  $\mathbf{b}=(b_1,\dots,b_n)\in (\setC^*)^n$.  Then
  $F_{\alpha_1+2\alpha_2}$ acts injectively on the $U_q$-module
  $\phi_{F_\Sigma,\mathbf{b}}.L(\lambda)_{F_\Sigma}$.
\end{prop}
\begin{proof}
  We can show similarly to the above calculations in this section that
  \begin{equation*}
    \phi_{F_\Sigma,\mathbf{b}}(F_{\alpha_1+2\alpha_2}) = b_2^2 F_{\alpha_1+2\alpha_2} + (1-q^2) b_2^2 b_1^{-2} \frac{b_1^2-b_1^{-2}}{q^2-q^{-2}} F_{\beta_1}\inv F_{\beta_2}^{(2)}.
  \end{equation*}
  
  By Proposition~\ref{prop:25} and Corollary~\ref{cor:6} a root vector
  acts injectively on the $U_q$-module
  \begin{equation*}
    \phi_{F_\Sigma,(b_1,\dots,b_n)}.L(\lambda)_{F_\Sigma}
  \end{equation*}
  if and only if it acts injectively on
  \begin{equation*}
    \phi_{F_\Sigma,(\epsilon_1
      q^{i_1}b_1,\dots,\epsilon_n q^{i_n}b_n)}.L(\lambda)_{F_\Sigma}
  \end{equation*}
  for any $i_1,\dots,i_n\in\setZ$ and $\epsilon_1,\dots,\epsilon_n\in
  \{\pm 1\}$.
  
  Assume there exists a $0\neq v \in
  \phi_{F_\Sigma,\mathbf{b}}.L(\lambda)_{F_\Sigma}$ such that
  $F_{\alpha_1+2\alpha_2}v=0$. We have $v=F_{\beta_1}^{a_1}\cdots
  F_{\beta_n}^{a_n} \tensor v'$ for some $a_1,\dots,a_n\in \setZ$ and
  some $v'\in L(\lambda)$. So $F_{\alpha_1+2\alpha_2}v=0$ implies
  \begin{equation*}
    0=\phi_{F_\Sigma,\mathbf{b}}(F_{\alpha_1+2\alpha_2}) F_{\beta_1}^{a_1}\cdots F_{\beta_n}^{a_n}\tensor v' =  F_{\beta_1}^{a_1}\cdots F_{\beta_n}^{a_n}\tensor \phi_{F_\Sigma,\mathbf{c}}(F_{\alpha_1+2\alpha_2})v'
  \end{equation*}
  where $\mathbf{c}=(q^{a_1}b_1,\dots,q^{a_n}b_n)$. So there exists a
  $v'\in L(\lambda)$ and $a_1,\dots,a_n\in \setZ$ such that for
  $\mathbf{c}=(q^{a_1}b_1,\dots,q^{a_n}b_n)$,
  $\phi_{F_\Sigma,\mathbf{c}}(F_{\alpha_1+2\alpha_2})v'=0$. That is
  \begin{equation*}
    \left(c_2^2 F_{\alpha_1+2\alpha_2} + (1-q^2) c_2^2 c_1^{-2} \frac{c_1^2-c_1^{-2}}{q^2-q^{-2}} F_{\beta_1}\inv F_{\beta_2}^{(2)}\right) v' = 0
  \end{equation*}
  or equivalently
  \begin{equation*}
    F_{\beta_1}F_{\alpha_1+2\alpha_2}v' + (1-q^2) c_1^{-2} \frac{c_1^2-c_1^{-2}}{q^2-q^{-2}} F_{\beta_2}^{(2)}v'=0.
  \end{equation*}
  
  So to prove our claim it is enough to prove that
  \begin{equation*}
    \left(F_{\beta_1}F_{\alpha_1+2\alpha_2} + (1-q^2) c_1^{-2}
      \frac{c_1^2-c_1^{-2}}{q^2-q^{-2}} F_{\beta_2}^{(2)}\right)v'\neq
    0
  \end{equation*}
  for any $v'\in L(\lambda)$ and any $c_1 \in \setC^*$.
  
  So let $v'\in L(\lambda)$ and let $c_1\in \setC^*$. Let $r\in \setN$
  be such that $E_{\alpha_2}^{(r)}v'\neq 0$ and
  $E_{\alpha_2}^{(r+1)}v'=0$ (possible since $L(\lambda)$ is a highest
  weight module).  It is straightforward to show that for $a\in
  \setN$:
  \begin{equation*}
    [E_{\alpha_2}^{(a)},F_{\alpha_1+2\alpha_2}] = q^{-a+1}[2]F_{\beta_2}E_{\alpha_2}^{(a-1)}K_{\alpha_2}\inv + q^{4-2a} F_{\beta_1} E_{\alpha_2}^{(a-2)}K_{\alpha_2}^{-2}
  \end{equation*}
  and
  \begin{equation*}
    [E_{\alpha_2}^{(a)},F_{\beta_2}^{(2)}] = q^{2-a}[2] F_{\beta_2}F_{\beta_1} E_{\alpha_2}^{(a-1)}K_{\alpha_2}\inv + q^{3-2a}[2] F_{\beta_1}^2 E_{\alpha_2}^{(a-2)}K_{\alpha_2}^{-2}.
  \end{equation*}
  Using this we get
  \begin{align*}
    E_{\alpha_2}^{(r+2)}&\left(F_{\beta_1}F_{\alpha_1+2\alpha_2} +
      (1-q^2) c_1^{-2} \frac{c_1^2-c_1^{-2}}{q^2-q^{-2}}
      F_{\beta_2}^{(2)}\right)v'
    \\
    =& \left( q^{-2r} + q^{-1-2r}[2]
      (1-q^2)c_1^{-2}\frac{c_1^2-c_1^{-2}}{q^2-q^{-2}}\right)
    F_{\beta_1}^2 E_{\alpha_2}^{(r)}K_{\alpha_2}^{-2}v'
    \\
    =& q^{-2r}c_1^{-4} F_{\beta_1}^2
    E_{\alpha_2}^{(r)}K_{\alpha_2}^{-2}v'
    \\
    \neq& 0
  \end{align*}
  since $F_{\beta_1}$ acts injectively on $L(\lambda)$. Thus
  \begin{equation*}
    \left(F_{\beta_1}F_{\alpha_1+2\alpha_2} +
      (1-q^2) c_1^{-2} \frac{c_1^2-c_1^{-2}}{q^2-q^{-2}}
      F_{\beta_2}^{(2)}\right)v'\neq 0.
  \end{equation*}
\end{proof}

\begin{thm}
  \label{thm:clas_C}
  Let $\lambda$ be a weight such that $\lambda(K_{\beta})\in \pm
  q^\setN$ for all short $\beta\in \Phi$ and $\lambda(K_{\gamma}) \in
  \pm q^{1+2\setZ}$ for all long $\gamma \in \Phi$. Let
  $\mathbf{b}=(b_1,\dots,b_n)\in (\setC^*)^n$.  Then the $U_q$-module
  $\phi_{F_\Sigma,\mathbf{b}}.L(\lambda)_{F_\Sigma}$ is simple and
  torsion free if and only if $b_i\not \in \pm q^\setZ$, $i=2,\dots,
  n$ and $b_1^2 b_2\cdots b_n \not \in \pm q^\setZ$.
\end{thm}
\begin{proof}
  Let $i\in \{2,\dots,n\}$. By Proposition~\ref{prop:7},
  $E_{\alpha_i}$ acts injectively on
  $\phi_{F_{\Sigma},\mathbf{b}}.L(\lambda)_{F_\Sigma}$ if and only if
  $b_i\not \in \pm q^{\setZ}$. If
  $\phi_{F_\Sigma,\mathbf{b}}.L(\lambda)_{F_\Sigma}$ is torsion free
  then every root vector acts injectively. So
  $\phi_{F_\Sigma,\mathbf{b}}.L(\lambda)_{F_\Sigma}$ being torsion
  free implies $b_i \not \in \pm q^{\setZ}$.

  Let $\Sigma'=\{\beta_1',\dots,\beta_n'\}$ denote the set of
  commuting roots with $\beta_1' = \alpha_1+\alpha_2$,
  $\beta_2'=\alpha_1+2\alpha_2$,
  $\beta_j'=\alpha_1+2\alpha_2+\alpha_3+\cdots+\alpha_j$,
  $j=3,\dots,n$. Let $F'_{\beta_1'} :=
  T_{s_1}(F_{\alpha_2})=F_{\beta_2}$,
  $F'_{\beta_2'}:=T_{s_1}T_{s_2}(F_{\alpha_1})=F_{\alpha_1+2\alpha_1}$,
  $F'_{\beta_j'} := T_{s_1}\cdots T_{s_n}T_{s_1}\cdots
  T_{s_{j-2}}(F_{\alpha_{j-1}})=T_{s_2}(F_{\beta_{j}})=F_{\alpha_1+2\alpha_2+\alpha_3+\cdots+\alpha_j}$,
  $j=3,\dots,n$ (in this case we actually have
  $F'_{\beta_j'}=F_{\beta_j'}$) and $F_{\Sigma'}$ the Ore subset
  generated by $F'_{\beta_1'},\dots,F'_{\beta_n'}$. Similarly to the
  above calculations in this section we can show that for
  $\mathbf{c}\in (\setC^*)^n$
  \begin{align*}
    \phi_{F_{\Sigma'},\mathbf{c}}(F_{\alpha_2}) = c_n\inv \cdots
    c_3\inv c_2^{-2} \left( F_{\alpha_2} + q [2] c_1\inv
      \frac{c_1-c_1\inv}{q-q\inv} (F'_{\beta_1'})\inv
      F'_{\beta_2'}\right).
  \end{align*}
  Let $v\in L(\lambda)$ and let $r\in \setN$ be such that
  $F_{\alpha_2}^{(r)}v \neq 0$ and $F_{\alpha_2}^{(r+1)}v =0$
  (possible since $\lambda(K_{\alpha_2})\in\pm q^{\setN}$). Then we
  see like in the proof of Proposition~\ref{prop:7} that
  $\phi_{F_{\Sigma'},\mathbf{c}}(F_{\alpha_2})v =0$ if and only if
  $c_1=\pm q^{-r}$ thus
  $\phi_{F_{\Sigma'},\mathbf{c}}.L(\lambda)_{F_{\Sigma'}}$ is not
  torsion free whenever $c_1 \in \pm q^\setZ$ by
  Proposition~\ref{prop:25} and Corollary~\ref{cor:6}.

  Set $f(\mathbf{b})=(b_1^2 b_2\cdots b_n,b_1\inv b_3\inv\cdots
  b_n,b_3,\dots,b_n)$. Then by Lemma~\ref{lemma:1}
  \begin{align*}
    \left(
      \phi_{F_\Sigma,\mathbf{b}}.L(\lambda)_{F_\Sigma}\right)^{ss}
    \iso \left(
      \phi_{F_{\Sigma'},f(\mathbf{b})}.L(\lambda)_{F_{\Sigma'}}\right)^{ss}.
  \end{align*}
  If $\phi_{F_\Sigma,\mathbf{b}}.L(\lambda)_{F_\Sigma}$ is torsion
  free then it is simple so
  \begin{align*}
    \phi_{F_\Sigma,\mathbf{b}}.L(\lambda)_{F_\Sigma} \iso& \left(
      \phi_{F_\Sigma,\mathbf{b}}.L(\lambda)_{F_\Sigma}\right)^{ss}
    \\
    \iso& \left(
      \phi_{F_{\Sigma'},f(\mathbf{b})}.L(\lambda)_{F_{\Sigma'}}\right)^{ss}
    \\
    \iso& \phi_{F_{\Sigma'},f(\mathbf{b})}.L(\lambda)_{F_{\Sigma'}}.
  \end{align*}
  We see that $\phi_{F_\Sigma,\mathbf{b}}.L(\lambda)_{F_\Sigma}$ being
  torsion free implies $b_1^2 b_2\cdots b_n \not \in \pm q^{\setZ}$.

  Now assume $b_i\not \in \pm q^\setZ$, $i=2,\dots, n$ and $b_1^2
  b_2\cdots b_n \not \in \pm q^\setZ$. By Proposition~\ref{prop:7} and
  Proposition~\ref{prop:5} $E_{\alpha_i}$ and $F_{\alpha_i}$,
  $i=3,\dots,n$ act injectively on all composition factors of
  $\phi_{F_\Sigma,\mathbf{b}}.L(\lambda)_{F_\Sigma}$.
  
  Let $L_1$ be a simple submodule of
  $\phi_{F_\Sigma,\mathbf{b}}.L(\lambda)_{F_\Sigma}$ and let $L_2$ be
  a simple submodule of
  $\phi_{F_{\Sigma'},f(\mathbf{b})}.L(\lambda)_{F_{\Sigma'}}$. By
  Proposition~\ref{prop:6}, $F_{\alpha_1+2\alpha_2}$ acts injectively
  on $\phi_{F_\Sigma,\mathbf{b}}.L(\lambda)_{F_\Sigma}$. Now clearly
  $\{-\alpha_1-\alpha_2,-\alpha_1-2\alpha_2,\alpha_3,\dots,\alpha_n\}\subset
  T_{L_1}\cap T_{L_2}$ so $C(L_1)\cap C(L_2)$ generates $Q$. This
  implies that $C(L_1)-C(L_2)=Q$. Since $\left(
    \phi_{F_\Sigma,\mathbf{b}}.L(\lambda)_{F_\Sigma}\right)^{ss} \iso
  \left(
    \phi_{F_{\Sigma'},f(\mathbf{b})}.L(\lambda)_{F_{\Sigma'}}\right)^{ss}$
  we have $\wt L_k \subset q^Q (\mathbf{b}\inv)^{\Sigma}\lambda$,
  $k=1,2$. Choose $\mu_1,\mu_2\in Q$ such that
  $q^{\mu_1}(\mathbf{b}\inv)^\Sigma \lambda \in \Suppess(L_1)$ and
  $q^{\mu_2}(\mathbf{b}\inv)^\Sigma \lambda \in \Suppess(L_2)$. Then
  obviously $q^{C(L_1)+\mu_1}(\mathbf{b}\inv)^\Sigma \lambda\subset
  \Suppess(L_1)$ and $q^{C(L_2)+\mu_2}(\mathbf{b}\inv)^\Sigma
  \lambda\subset \Suppess(L_2)$. By the above
  $q^{C(L_1)+\mu_1}(\mathbf{b}\inv)^\Sigma \lambda \cap
  q^{C(L_2)+\mu_2}(\mathbf{b}\inv)^\Sigma \lambda \neq \emptyset$ so
  $\Suppess(L_1)\cap \Suppess(L_2)\neq \emptyset$. Let $\nu \in
  \Suppess(L_1)\cap \Suppess(L_2)$. By Proposition~\ref{prop:15},
  $L_1$ and $L_2$ are admissible of the same degree as
  $L(\lambda)$. So we have as $(U_q)_0$-modules (using that
  $(L_1)_\nu$ and $(L_2)_\nu$ are simple $(U_q)_0$-modules by
  Theorem~\ref{thm:Lemire})
  \begin{align*}
    (L_1)_\nu &= \left(
      \phi_{F_\Sigma,\mathbf{b}}.L(\lambda)_{F_\Sigma}\right)_\nu \iso
    \left(
      \left(\phi_{F_\Sigma,\mathbf{b}}.L(\lambda)_{F_\Sigma}\right)_\nu\right)^{ss}
    \\
    &\iso \left(
      \left(\phi_{F_{\Sigma'},f(\mathbf{b})}.L(\lambda)_{F_{\Sigma'}}\right)_\nu\right)^{ss}
    \iso
    \left(\phi_{F_{\Sigma'},f(\mathbf{b})}.L(\lambda)_{F_{\Sigma'}}\right)_\nu
    = (L_2)_\nu.
  \end{align*}
  By Theorem~\ref{thm:Lemire} this implies $L_1\iso L_2$.

  Let $\Sigma''=\{ \beta_1'',\dots,\beta_n''\}$ denote the set of
  commuting roots with $\beta_1'' = \alpha_1+2\alpha_2$,
  $\beta_2''=\alpha_2$,
  $\beta_j''=\alpha_1+2\alpha_2+\alpha_3+\cdots+\alpha_j$,
  $j=3,\dots,n$. Let $F''_{\beta_1''} :=
  T_{s_1}T_{s_2}(F_{\alpha_1})$, $F''_{\beta_2''}:=F_{\alpha_2}$,
  $F''_{\beta_j''} := T_{s_2}T_{s_1}T_{s_2}T_{s_3}\cdots
  T_{s_{j-1}}(F_{\alpha_{j}})=T_{s_1}T_{s_2}(F_{\beta_j})$,
  $j=3,\dots,n$ and $F_{\Sigma''}$ the Ore subset generated by
  $F''_{\beta_1''},\dots,F''_{\beta_n''}$. Note that
  $F''_{\beta_j''}=T_{s_1}T_{s_2}(F_{\beta_j})$ for all $j\in
  \{1,\dots,n\}$. The root vectors
  $F''_{\beta_1''},\dots,F''_{\beta_n''}$ act injectively on
  ${^{\bar{s_2 s_1}}}L(\lambda)$. By
  Theorem~\ref{thm:EXT_contains_highest_weight} and
  Proposition~\ref{prop:19} $L(\lambda)$ is a submodule of
  $\left(\phi_{F_{\Sigma''},\mathbf{d}}.({^{\bar{s_2
          s_1}}}L(\lambda))_{F_{\Sigma''}}\right)^{ss}$ for some
  $\mathbf{d}\in (\setC^*)^n$.  Then by Lemma~\ref{lemma:1}
  \begin{equation*}
    \left(\phi_{F_\Sigma,\mathbf{b}}.L(\lambda)_{F_\Sigma}\right)^{ss} \iso \left(\phi_{F_{\Sigma''},g(\mathbf{b})\mathbf{d}}.({^{\bar{s_2s_1}}}L(\lambda))_{F_{\Sigma''}}\right)^{ss}  
  \end{equation*}
  for some $g(\mathbf{b})\in (\setC^*)^n$.

  Observe that for $a_1,\dots,a_n\in \setN$:
  \begin{align*}
    \phi_{F_{\Sigma''},(q^{a_1},\dots,q^{a_n})}&(-K_{\alpha_1}E_{\alpha_1})
    \\
    =&
    \phi_{F_{\Sigma''},(q^{a_1},\dots,q^{a_n})}(T_{s_1}T_{s_2}(F_{\alpha_1+2\alpha_2}))
    \\
    =& \left(F''_{\beta_1''}\right)^{-a_1}\cdots
    \left(F''_{\beta_n''}\right)^{-a_n}
    T_{s_1}T_{s_2}(F_{\alpha_1+2\alpha_2})
    \left(F''_{\beta_n''}\right)^{a_n} \cdots
    \left(F''_{\beta_1''}\right)^{a_1}
    \\
    =& T_{s_1}T_{s_2}\left( F_{\beta_1}^{-a_n}\cdots
      F_{\beta_n}^{-a_n} F_{\alpha_1+2\alpha_2}
      F_{\beta_n}^{a_n}\cdots F_{\beta_1}^{a_1}\right)
    \\
    =& T_{s_1}T_{s_2}\left(
      \phi_{F_\Sigma,(q^{a_1},\dots,q^{a_n})}(F_{\alpha_1+2\alpha_2})\right).
  \end{align*}
  Since $\phi_{F_{\Sigma''},\mathbf{c}}(-K_{\alpha_1}E_{\alpha_1})$
  and
  $T_{s_1}T_{s_2}\left(\phi_{F_\Sigma,\mathbf{c}}(F_{\alpha_1+2\alpha_2})\right)$
  are both Laurent polynomial in $\mathbf{c}$ we get by
  Lemma~\ref{lemma:37} that
  $\phi_{F_{\Sigma''},\mathbf{c}}(-K_{\alpha_1}E_{\alpha_1})
  =T_{s_1}T_{s_2}\left(\phi_{F_\Sigma,\mathbf{c}}(F_{\alpha_1+2\alpha_2})\right)$
  for any $\mathbf{c}\in (\setC^*)^n$.
  $T_{s_1}T_{s_2}\left(\phi_{F_\Sigma,\mathbf{c}}(F_{\alpha_1+2\alpha_2})\right)$
  acts injectively on ${^{\bar{s_2s_1}}}L(\lambda)$ for any
  $\mathbf{c}\in (\setC^*)^n$ by Proposition~\ref{prop:6}. This
  implies that $-K_{\alpha_1}E_{\alpha_1}$ acts injectively on
  $\phi_{F_{\Sigma''},g(\mathbf{b})\mathbf{d}}.({^{\bar{s_2s_1}}}L(\lambda))_{F_{\Sigma''}}$
  and this implies that $E_{\alpha_1}$ acts injectively.

  Let $L_3$ be a simple submodule of
  $\phi_{F_{\Sigma''},g(\mathbf{b})\mathbf{d}}.({^{\bar{s_2s_1}}}L(\lambda))_{F_{\Sigma''}}$.
  We see that
  $\{-\alpha_2,-\alpha_1-2\alpha_2,\alpha_3,\dots,\alpha_n\}\subset
  T_{L_3}\cap T_{L_2}$ so $C(L_2)\cap C(L_3)$ generates $Q$
  ($\{\alpha_3,\dots,\alpha_n\}\subset T_{L_3}$ because of
  Proposition~\ref{prop:5} and the fact that $L_3$ is a composition
  factor of
  $\phi_{F_\Sigma,\mathbf{b}}.L(\lambda)_{F_\Sigma}$). Arguing as
  above this implies that $L_2\iso L_3$. We have shown that $L_1\iso
  L_2 \iso L_3$. Above we have shown that $E_{\alpha_1}$ acts
  injectively on $L_3$, $F_{\alpha_2}$ acts injectively on $L_2$ and
  $F_{\alpha_1},E_{\alpha_2},F_{\alpha_i},E_{\alpha_i}$, $i=3,\dots,n$
  act injectively on $L_1$. In conclusion we have shown that all root
  vectors act injectively on the simple submodule $L_1$ of
  $\phi_{F_\Sigma,\mathbf{b}}.L(\lambda)_{F_\Sigma}$ thus $\wt L_1 =
  \Suppess(L_1)=q^Q (\mathbf{b}\inv)^{\Sigma}\lambda$ and therefore
  $L_1 = \phi_{F_\Sigma,\mathbf{b}}.L(\lambda)_{F_\Sigma}$. This shows
  that $\phi_{F_\Sigma,\mathbf{b}}.L(\lambda)_{F_\Sigma}$ is simple
  and torsion free with our assumptions on $\mathbf{b}$.
\end{proof}

\bibliography{lit} \bibliographystyle{amsalpha}

\end{document}